\documentclass[final]{siamart1116}
\usepackage{geometry}
\usepackage{amsmath}
\usepackage{cases}
\usepackage{stmaryrd}
\usepackage{mathrsfs}
\usepackage{amssymb}
\usepackage{lipsum}
\usepackage{multirow}
\usepackage{amsfonts}
\usepackage{graphicx}
\usepackage{epstopdf}
\usepackage{algorithmic}
\usepackage{amsopn}
\usepackage{makecell,booktabs}
\usepackage{array}
\usepackage{subcaption}
\usepackage{url}

\ifpdf
  \DeclareGraphicsExtensions{.eps,.pdf,.png,.jpg}
\else
  \DeclareGraphicsExtensions{.eps}
\fi
\numberwithin{theorem}{section}
\newsiamremark{remark}{Remark}
\newsiamremark{assumption}{Assumption}
\newtheorem{alemma}{Lemma}

\ifpdf
\hypersetup{ pdftitle={Guide to Using  SIAM'S \LaTeX\ Style} }
\fi
\newcommand{\TheTitle}{Convergence of a class of nonmonotone descent methods for KL optimization problems}

\title{{\TheTitle}\thanks{{July 16, 2022.}
\funding{This work was supported by the National Natural Science Foundation of China
under projects No.11971177 and Guangdong Basic and Applied Basic Research Foundation (2020A1515010408).}}}

\author{Yitian Qian\thanks{School of Mathematics, South China University of Technology, Guangzhou, China
  (\email{mayttqian@mail.scut.edu.cn}).}\and
  Shaohua Pan\thanks{School of Mathematics, South China University of Technology, Guangzhou, China (\email{shhpan@scut.edu.cn}).}
 }

\begin{document}

 \maketitle

 \begin{abstract}
  This paper is concerned with a class of nonmonotone descent methods for minimizing
  a proper lower semicontinuous KL function $\Phi$, which generates a sequence 
  satisfying a nonmonotone decrease condition and a relative error tolerance. 
  Under suitable assumptions, we prove that the whole sequence converges to a limiting 
  critical point of $\Phi$ and, when $\Phi$ is a KL function of exponent $\theta\in[0,1)$, 
  the convergence admits a linear rate if $\theta\in[0,1/2]$ and a sublinear rate 
  associated to $\theta$ if $\theta\in(1/2,1)$. The required assumptions are shown 
  to be sufficient and necessary if $\Phi$ is also weakly convex on 
  a neighborhood of stationary point set. Our results resolve the convergence 
  problem on the iterate sequence generated by a class of nonmonotone line search 
  algorithms for nonconvex and nonsmooth problems, and also extend the convergence 
  results of monotone descent methods for KL optimization problems.
  As the applications, we achieve the convergence of the iterate sequence for the nonmonotone
  line search proximal gradient method with extrapolation and the nonmonotone line search proximal alternating minimization method with extrapolation. Numerical experiments
  are conducted for zero-norm and column $\ell_{2,0}$-norm regularized problems 
  to validate their efficiency.
 \end{abstract}

\begin{keywords}
 KL optimization problems, nonmonotone descent methods, global convergence, convergence rate
\end{keywords}

\begin{AMS}
  90C26, 65K05, 49M27
\end{AMS}

 \section{Introduction}\label{sec1}
 
 Let $\mathbb{X}$ represent a finite dimensional real vector space endowed with
  the inner product $\langle\cdot,\cdot\rangle$ and its induced norm $\|\cdot\|$.
  Consider the nonconvex and nonsmooth problem
  \begin{equation}\label{Fprob}
   \min_{x\in\mathbb{X}}F(x):=f(x)+g(x),
  \end{equation}
  where $f\!:\mathbb{X}\to\mathbb{R}$ is an $L_{\!f}$-smooth function
  and $g\!:\mathbb{X}\to\overline{\mathbb{R}}\!:=(-\infty,\infty]$ is a proper 
  lower semicontinuous (lsc) function. The nonmonotone descent method dates back 
  to the nonmonotone line search Newton's method proposed by Grippo et al. \cite{Grippo86} 
  for problem \eqref{Fprob} with $g\equiv 0$, aiming to improve the performance of 
  the monotone Armijo line search Newton's method. Owing to its better empirical 
  performance \cite{Grippo89}, this line search technique was later widely 
  applied to gradient-type methods (see, e.g., \cite{Raydan97, Birgin00, Grippo02}) 
  and extended to proximal gradient (PG) methods for \eqref{Fprob}. 
 \subsection{Main motivation}\label{sec1.1}
  
  For the nonmonotone line search Newton's method, Grippo et al. \cite{Grippo86,Grippo89}
  achieved the convergence of the whole iterate sequence under the restricted 
  assumption that the number of stationary points is finite. For the nonmonotone 
  line search gradient-type methods, to the best of our knowledge, there are no 
  convergence results on the whole iterate sequence even for unconstrained smooth 
  convex programs. Dai \cite{Dai02} showed that the objective value sequence of 
  any iterative method with the nonmonotone line search is R-linearly convergent 
  if the smooth objective function is strongly convex. As is well known, 
  the strong convexity is also very restricted.

  For problem \eqref{Fprob} with a finite convex $g$, Wright et al. \cite{Wright09} 
  proposed an efficient method, called SpaRSA, by the nonmonotone line search technique 
  in \cite{Grippo02}, and achieved the convergence of the objective value sequence 
  and proved that every cluster point of the iterate sequence is a (limiting) critical 
  point of $F$. For SPaRSA, Hager et al. \cite{Hager11} later obtained the sublinear 
  convergence rate and the R-linear convergence rate of the objective value sequence 
  under the convexity and the strong convexity of $F$, respectively. SpaRSA is actually 
  a PG method with a nonmonotone line search (NPG, for short) and has been extended to 
  solving \eqref{Fprob} with a continuous nonconvex $g$ (see \cite{Gong13,Kanzow21}) 
  and problem \eqref{Fprob} itself (see \cite{LuZhang12} and \cite[Appendix A]{Chen16}). 
  Lu and Zhang \cite{LuZhang12} obtained the Q-linear convergence rate of 
  a special objective value subsequence (implying the R-linear convergence rate of 
  the whole objective value sequence) under an assumption stronger than the KL property 
  of exponent $1/2$ of $F$ by \cite[Theorem 4.1]{LiPong18},  
  and Kanzow et al. \cite{Kanzow21} removed the common Lipschitz assumption 
  on $\nabla\!f$ in the convergence analysis of PG methods and proved that 
  every cluster point yielded by the NPG method is a critical one of $F$.  
  The NPG method was also applied to the DC program (see, e.g., \cite{LiuPong19, LuZhou19}) 
  and the block structured composite optimization (see, e.g., \cite{LuLin17, Yang18}). 
  Although the NPG method for nonconvex and nonsmooth composite optimization exhibits
  the promising performance, there is no convergence certificate for the generated 
  iterate sequence. Recently, Yang \cite{Yang21} proposed a nonmonotone descent method 
  for \eqref{Fprob} by combining the nonmonotone line search in \cite{Grippo02} 
  with the extrapolation technique \cite{Beck09}, but only established the convergence 
  rate of the objective value sequence for the extrapolation case and the monotone 
  line search case by assuming that $F$ is a KL function of exponent $\theta\in[0,1)$. 
  It is still unclear whether the iterate sequence is convergent. 
    
  To sum up, the convergence of the iterate sequence generated by the nonmonotone 
  line search descent method \cite{Grippo86,Grippo89} remains open for nonconvex 
  and nonsmooth composite optimization even for unconstrained smooth optimization. 
  Recently, some researchers \cite{Li15,WangLiu21,Nazih21} 
  proposed the nonmonotone accelerated PG methods by combining the nonmonotone 
  line search in \cite{Zhang04} and the extrapolation technique \cite{Beck09}. 
  Wang and Liu \cite{WangLiu21} achieved the convergence rate of 
  the objective value sequence by assuming that $F$ is a KL function 
  of exponent $\theta\in(0,1)$, and that of the iterate sequence by assuming that 
  $F$ is a KL function of exponent $\theta\in(0,3/4)$. However, the nonmonotonicity 
  involved in their methods is caused by the extrapolation strategy rather than 
  the step-size. 
  
  Let $\Phi\!:\mathbb{X}\to\overline{\mathbb{R}}$ be a proper
  lsc function that is coercive and bounded below on its domain.
  We are interested in nonmonotone descent methods for the abstract problem
  $\min_{x\in\mathbb{X}}\Phi(x)$, which generate sequences $\{x^k\}_{k\in\mathbb{N}}$
  satisfying the nonmonotone decrease and relative error conditions:
  \begin{itemize}
    \item [{\bf H1.}] For each $k\in\mathbb{N}$, $\Phi(x^{k+1})+a\|x^{k+1}-x^k\|^2\le
                    \max_{j=[k-m]_{+},\ldots,k}\Phi(x^j)$;

    \item[{\bf H2.}] For each $k\in\mathbb{N}$, $\exists w^{k}\in\partial\Phi(x^k)$ such that
                     $\|w^{k+1}\|\le b\|x^{k+1}\!-\!x^{k}\|$;
  \end{itemize}
  where $m\ge0$ is an integer with $[k\!-\!m]_{+}\!:=\max(0,k\!-\!m)$, $a>0,b>0$ are
  the given constants, and $\partial\Phi(x^{k})$ denotes the set of limiting subgradients
  of $\Phi$ at $x^{k}$. The sequences $\{x^k\}_{k\in\mathbb{N}}$ satisfying the conditions
  H1-H2 are the extension of those studied by Attouch et al. \cite{Attouch13} to
  the nonmonotone descent case. As will be shown in Section \ref{sec4}-\ref{sec5},
  the nonmonotone line search variants of the PG method and
  the proximal alternating minimization (PALM) method \cite{Bolte14} precisely generate
  the sequences satisfying conditions H1-H2. It is well known that the PG method
  (also known as the forward-backward splitting method \cite{Combettes11} or the iterative
  shrinkage-thresholding algorithm \cite{Beck09}) and the PALM method are very popular
  for nonconvex and nonsmooth composite optimization. 
 \subsection{Our contributions}\label{sec1.2}
    
  Let $\{x^k\}_{k\in\mathbb{N}}$ be a sequence complying with conditions H1-H2.
  This work focuses on the convergence analysis of $\{x^k\}_{k\in\mathbb{N}}$
  and achieves the following main results.
  \begin{itemize}
   \item  When $\Phi$ is a KL function satisfying \eqref{Phi-cond1}-\eqref{Phi-cond2}, 
          the sequence $\{x^k\}_{k\in\mathbb{N}}$ converges to a critical point of $\Phi$ 
          under condition \eqref{assump0}, which involves the growth of an objective 
          value subsequence and is shown to be sufficient and necessary for 
          $\sum_{k=0}^{\infty}\|x^{k+1}-x^k\|<\infty$ if $\Phi$ is also weakly convex 
          on a neighborhood of stationary point set, and now a sufficient condition 
          independent of the objective value sequence (see condition \eqref{scond}) 
          is provided for \eqref{assump0} and is shown to hold automatically if $\Phi$ 
          is also $\rho$-weakly convex with $\rho\le\frac{a}{8(m+1)^2}$ 
          on a neighborhood of stationary point set; 

   \item  When $\Phi$ is a KL function of exponent $\theta\in[0,1)$ satisfying 
          \eqref{Phi-cond1}-\eqref{Phi-cond2}, the sequence $\{x^k\}_{k\in\mathbb{N}}$ 
           converges linearly if $\theta\in(0,1/2]$ and sublinearly if $\theta\in(1/2,1)$ 
           to a critical point of $\Phi$ under condition \eqref{rate-cond}, 
           which also involves the growth of an objective value subsequence and 
           is shown to be sufficient and necessary for the conclusion of Theorem 
           \ref{KL-rate} (ii) if $\Phi$ is also weakly convex on a neighborhood of 
          stationary point set, and in this case condition \eqref{scond} is also 
          sufficient for \eqref{rate-cond} and holds automatically if $\Phi$ 
          is $\rho$-weakly convex with $\rho\le\frac{a}{8(m+1)^2}$ on a neighborhood 
          of stationary point set.

   \item  The nonmonotone line search PG method with extrapolation (PGenls) 
          proposed in \cite{Yang21} for solving \eqref{Fprob} with a proper 
          lsc $g\!:\mathbb{X}\to\overline{\mathbb{R}}$ and a nonmonotone line search 
          PALM method with extrapolation (PALMenls) for solving \eqref{PALM-prob} 
          are demonstrated to generate the sequences satisfying the conditions H1-H2,
          and their global convergence and local convergence rate are achieved under suitable
          assumptions. Numerical experiments are conducted to validate their superiority to
          the monotone line search or the accelerated version in some scenarios.
  \end{itemize}
  
  As a byproduct, when $\Phi$ is a KL function of exponent $\theta\in[0,1)$ satisfying conditions \eqref{Phi-cond1}-\eqref{Phi-cond2}, we also obtain the linear convergence rate
  of the objective value sequence $\{\Phi(x^k)\}_{k\in\mathbb{N}}$ if $\theta\in(0,1/2]$ and the sublinear convergence rate if $\theta\in(1/2,1)$. Then, when applying
  PGenls and PALMenls to problems \eqref{Fprob} and \eqref{PALM-prob}, respectively,
  if the objective functions are the KL function of exponent $\theta\in[1/2,1)$,
  the generated objective value sequences have the corresponding convergence rate.

  It is worth pointing out that conditions \eqref{Phi-cond1}-\eqref{Phi-cond2}
  are rather weak, which are satisfied by the objective functions of the composite problems \eqref{Fprob} and \eqref{PALM-prob}. As discussed thoroughly in \cite[Section 4]{Attouch10}, there are a large number of nonconvex nonsmooth optimization problems involving KL functions, which include real semi-algebraic functions and those functions definable in an o-minimal structure \cite{Kurdyka98,Bolte07}. Thus, the obtained convergence results have a wide range of applications.  
 \section{Notation and preliminaries}\label{sec2} 
 
  Throughout this paper, for a sequence $\{a_k\}_{k\in\mathbb{N}}$ and an index set $K\subseteq\mathbb{N}$, $\sum_{K\ni j=k_1}^{k_2}a_j$ denotes the sum of those $a_j$ with
  $j\in K\cap[k_1,k_2]$. For a proper $h\!:\mathbb{X}\!\to\overline{\mathbb{R}}$,
  denote by ${\rm dom}\,h$ its effective domain, $\partial h(\overline{x})$ 
  denotes the (limiting) subdifferential of $h$ at $\overline{x}$, 
  and for any $-\infty<\!\eta_1<\!\eta_2<\!\infty$, 
  write $[\eta_1<h<\eta_2]\!:=\{x\in\mathbb{X}\,|\,\eta_1<h(x)<\eta_2\}$. 
  For a given $\eta\in(0,\infty]$, $\Upsilon_{\!\eta}$ denotes the family 
  of continuous concave $\varphi\!:[0,\eta)\to\mathbb{R}_{+}$ 
  that is continuously differentiable on $(0,\eta)$ with $\varphi'(s)>0$ 
  for all $s\in(0,\eta)$ and $\varphi(0)=0$. For a proper lsc $h\!:\mathbb{X}\!\to\overline{\mathbb{R}}$,
  $\mathcal{P}_{\gamma}h(x):=\mathop{\arg\min}_{z\in\mathbb{X}}\big\{\frac{1}{2\gamma}\|z-x\|^2+h(z)\big\}$ denotes the proximal mapping of $h$ associated to $\gamma>0$. For a matrix $A\in\mathbb{R}^{n\times p}$, $\|A\|$ and $\|A\|_F$ denote its spectral norm 
  and Frobenius norm. 
 \begin{definition}\label{Gsubdiff-def}
  (see \cite[Definition 8.3]{RW98}) Consider a proper function $h\!:\mathbb{X}\to\overline{\mathbb{R}}$
  and a point $x\in{\rm dom}\,h$. The regular subdifferential of $h$ at $x$,
  denoted by $\widehat{\partial}h(x)$, is defined as
  \[
    \widehat{\partial}h(x):=\bigg\{v\in\mathbb{X}\ \big|\
    \liminf_{x\ne x'\to x}\frac{h(x')-h(x)-\langle v,x'-x\rangle}{\|x'-x\|}\ge 0\bigg\};
  \]
  and the (limiting) subdifferential of $h$ at $x$, denoted by $\partial h(x)$, is defined as
  \[
    \partial h(x):=\Big\{v\in\mathbb{X}\ |\  \exists\,x^k\to x\ {\rm with}\ h(x^k)\to h(x)\ {\rm and}\
    v^k\in\widehat{\partial}h(x^k)\ {\rm such\ that}\ v^k\to v\Big\}.
  \]
 \end{definition}

  For any $x\in{\rm dom}\,h$, the set $\widehat{\partial}h(x)$ is closed convex,
  $\partial h(x)$ is closed but generally nonconvex, and they satisfy
  $\widehat{\partial}h(x)\subseteq\partial h(x)$. The inclusion may be strict
  when $h$ is nonconvex. Recall that a function $h\!:\mathbb{X}\to\overline{\mathbb{R}}$
 is said to be weakly convex if there exists a constant $\rho>0$ such that
 the function $x\mapsto h(x)+\frac{\rho}{2}\|x\|^2$ is convex. For such $h$,
 at any $x\in{\rm dom}\,h$, $\widehat{\partial}h(x)=\partial h(x)$.
 In the sequel, the set of those points $\overline{x}$ at which $0\in\partial h(\overline{x})$
 is called the critical point set of $h$, denoted by ${\rm crit}\,h$.
 \begin{definition}\label{KL-def}
  A proper lsc function $h\!:\mathbb{X}\to\overline{\mathbb{R}}$ is said to have
  the KL property at $\overline{x}\in{\rm dom}\,\partial h$ if there exist $\eta\in(0,\infty]$,
  a neighborhood $\mathcal{U}$ of $\overline{x}$, and a function $\varphi\in\Upsilon_{\!\eta}$ 
  such that for all $x\in\mathcal{U}\cap\big[h(\overline{x})<h<h(\overline{x})+\eta\big]$,
  $\varphi'(h(x)\!-\!h(\overline{x})){\rm dist}(0,\partial h(x))\ge 1$.
  If $\varphi$ can be chosen as $\varphi(t)=ct^{1-\theta}$ with $\theta\in[0,1)$
  for some $c>0$, then $h$ is said to have the KL property of exponent $\theta$
  at $\overline{x}$. If $h$ has the KL property (of exponent $\theta$) at each point
  of ${\rm dom}\,\partial h$, then it is called a KL function (of exponent $\theta$).
 \end{definition}
 \begin{remark}\label{KL-remark}
  By \cite[Lemma 2.1]{Attouch10},
  a proper lsc function has the KL property at any noncritical point.
  Thus, to show that a proper lsc $h\!:\mathbb{X}\to\overline{\mathbb{R}}$
  is a KL function, it suffices to check its KL property at critical points.
  On the calculation of KL exponent, refer to the recent works \cite{LiPong18,YuLiPong21}.
 \end{remark}
\section{Convergence results}\label{sec3}

Denote by $\ell(k)$ the maximum index in $\mathop{\arg\max}_{j=[k-m]_{+},\ldots,k}\Phi(x^j)$. 
We provide two technical lemmas that are used in the subsequent analysis. 
Since Lemma \ref{lemma1-sequence} is immediate by using $\sqrt{ab}\le\frac{1}{4}a+b$ 
for any $a\ge 0,b\ge 0$, we omit its proof. 
\begin{lemma}\label{lemma1-sequence}
	Let $\{z^l\}_{l\in\mathbb{N}}\subset\mathbb{X},\{\alpha_l\}_{l\in\mathbb{N}}
	\subset\mathbb{R}_{+}$ and $\{\beta_l\}_{l\in\mathbb{N}}\subset\mathbb{R}_{+}$ 
	be the given sequences, and let $\mathcal{K}$ be an index set. If there exists 
	an index $\overline{l}\in\mathbb{N}$ such that for all $\mathcal{K}\ni l\ge\overline{l}$, 
	$\beta_l>\beta_{l+1}>0$ and $\|z^{l+1}-\!z^l\|\le\!\sqrt{\alpha_l(\beta_l-\!\beta_{l+1})}$, 
	then for any $\nu>l$ with $\mathcal{K}\ni l\ge\overline{l}$,
	$\sum_{\mathcal{K}\ni j=l}^{\nu}\!\big\|z^{j+1}\!-z^j\big\|
	\le\frac{1}{4}\sum_{j=l}^{\nu}\!\alpha_j+\beta_l$.
\end{lemma}
\begin{lemma}\label{lemma1-Phi}
	Let $\{x^k\}_{k\in\mathbb{N}}$ be a sequence satisfying condition H1, 
	and denote by $\varpi(x^0)$ the cluster point set of the sequence 
	$\{x^k\}_{k\in\mathbb{N}}$. Then, the following assertions hold.
	\begin{itemize}
		\item [(i)] The sequence $\{x^k\}_{k\in\mathbb{N}}$ is bounded, 
		and $\varpi(x^0)$ is a nonempty and compact set.
		
		\item[(ii)] The sequence $\{\Phi(x^k)\}_{k\in\mathbb{N}}$ is convergent and
		$\lim_{k\to\infty}x^{k+1}\!-\!x^{k}=0$, provided that
		\begin{equation}\label{Phi-cond1}
			\liminf_{k\to\infty}\Phi(x^{k})\ge\lim_{k\to\infty}\Phi(x^{\ell(k)}).
		\end{equation}
		
		\item[(iii)] The function $\Phi$ keeps constant on the set $\varpi(x^0)$ 
		when inequality \eqref{Phi-cond1} is satisfied and
		\begin{equation}\label{Phi-cond2}
			\limsup_{j\to\infty}\Phi(x^{k_j})\le\Phi(\widehat{x})
			\quad{\rm for\ each}\ \{x^{k_j}\}_{j\in\mathbb{N}}
			\ {\rm with}\ \lim_{j\to\infty}x^{k_j}=\widehat{x}.
		\end{equation}
		
		\item[(iv)] If $\Phi(x^{\ell(k)})=\Phi(x^{\ell(\overline{k})})$ for all $k\ge\overline{k}$,
		then all $x^{k}$ for $k>\overline{k}\!+\!m$ are the same.
		
		\item[(v)] If $\{x^k\}_{k\in\mathbb{N}}$ also satisfies H2,
		then under conditions \eqref{Phi-cond1}-\eqref{Phi-cond2} 
		$\varpi(x^0)\subseteq{\rm crit}\Phi$.
	\end{itemize}
\end{lemma}
\begin{proof}
	{\bf(i)} From condition H1 and the definition of $\ell(k)$,
	it follows that for each $k\in\mathbb{N}$,
	\begin{equation}\label{lk-monotone}
		\Phi(x^{\ell(k+1)})\le\!\max\{\Phi(x^{\ell(k)}),\Phi(x^{k+1})\}
		\le\!\max\{\Phi(x^{\ell(k)}),\Phi(x^{\ell(k)})\!-\!a\|x^{k+1}\!-\!x^k\|^2\}
		=\!\Phi(x^{\ell(k)}).
	\end{equation}
	This means that $\{x^k\}_{k\in\mathbb{N}}\subseteq\{x\in\mathbb{X}\,|\,\Phi(x)\le\Phi(x^0)\}$.
	The result follows by the coerciveness of $\Phi$.
	
	\noindent
	{\bf(ii)} Note that \eqref{lk-monotone} implies $\limsup_{k\to\infty}\Phi(x^{k})
	\le\limsup_{k\to\infty}\Phi(x^{\ell(k)})$ and the convergence of 
	$\{\Phi(x^{\ell(k)})\}_{k\in\mathbb{N}}$. Together with \eqref{Phi-cond1},
	the sequence $\{\Phi(x^k)\}_{k\in\mathbb{N}}$ is convergent. 
	For each $k\ge m$ and $j_k\in\{0,1,\ldots,m\!+\!1\}$, by condition H1,
	$\Phi(x^{\ell(k)-j_k})\le \Phi(x^{\ell(\ell(k)-j_k-1)})
	-a\|x^{\ell(k)-j_k}-x^{\ell(k)-j_k-1}\|^2$,
	which by the convergence of $\{\Phi(x^k)\}_{k\in\mathbb{N}}$ implies that
	$\lim_{k\to\infty}(x^{\ell(k)-j_k}-x^{\ell(k)-j_k-1})=0$.
	Since $\ell(k)\in\{k\!-\!m,\ldots,k\}$ for each $k\ge m$,
	we have $k\!-\!(m\!+\!1)=\ell(k)\!-\!j_k$ for some $j_k\in\{0,\ldots,m\!+\!1\}$. 
	Then, $\lim_{k\to\infty}(x^{k-m}\!-\!x^{k-m-1})=0$.
	
	\noindent
	{\bf(iii)} Pick any $\overline{x}\in\varpi(x^0)$. There exists a subsequence
	$\{x^{k_j}\}_{j\in\mathbb{N}}$ such that $\lim_{j\to\infty}x^{k_j}=\overline{x}$.
	By combining \eqref{Phi-cond2} and the lower semicontinuity of $\Phi$, 
	we obtain $\lim_{j\to\infty}\Phi(x^{k_j})=\Phi(\overline{x})$.
	From part (ii), $\Phi(\overline{x})=\lim_{k\to\infty}\Phi(x^k)$.
	By the arbitrariness of $\overline{x}$, $\Phi$ keeps constant on $\varpi(x^0)$.
	
	\noindent
	{\bf(iv)} Let $\Phi(x^{\ell(k)})=\omega$ for all $k\ge\overline{k}$.
	Suppose that the conclusion does not hold. Then there exist $k_2>k_1>\overline{k}+m$
	such that $x^{k_1}\ne x^{k_2}$. By condition H1, for each $j\in\{0,1,\ldots,k_2\!-\!k_1\}$,
	we have $\Phi(x^{k_1+j})+a\|x^{k_1+j}\!-x^{k_1+j-1}\|^2\le\Phi(x^{\ell(k_1+j-1)})=\omega$.
	Since $x^{k_1}\ne x^{k_2}$, there exists $\overline{j}\in\{0,\ldots,k_2-\!k_1\}$
	such that $\|x^{k_1+\overline{j}}-x^{k_1+\overline{j}-1}\|\ne 0$ and then
	$\Phi(x^{k_1+\overline{j}})<\omega$. Together with 
	$\Phi(x^{k_1+\overline{j}+1})+a\|x^{k_1+\overline{j}+1}-x^{k_1+\overline{j}}\|^2
	\le\Phi(x^{\ell(k_1+\overline{j})})=\omega$, we have
	$\Phi(x^{k_1+\overline{j}+1})=\Phi(x^{k_1+\overline{j}})<\omega$.
	Using the similar arguments leads to $\max\{\Phi(x^{k_1+\overline{j}}),
	\Phi(x^{k_1+\overline{j}+1}),\ldots,\Phi(x^{k_1+\overline{j}+m})\}<\omega$. 
	This yields a contradiction $\omega=\Phi(x^{\ell(k_1+\overline{j}+m)})
	\le\max_{0\le i\le m}\Phi(x^{k_1+\overline{j}+i})<\omega$.
	The desired result holds.
	
	\noindent
	{\bf(v)} Pick any $\overline{x}\in\varpi(x^0)$. There exists a subsequence
	$\{x^{k_j}\}_{j\in\mathbb{N}}$ such that $\lim_{j\to\infty}x^{k_j}=\overline{x}$.
	From condition H2, for each $j\in\mathbb{N}$, there exists $w^{k_j}\in\partial\Phi(x^{k_j})$ 
	with $\|w^{k_j}\|\le b\|x^{k_j}-x^{k_j-1}\|$. By part (ii),
	$\lim_{j\to\infty}w^{k_j}=0$. In addition, since $\Phi$ is lsc, 
	from \eqref{Phi-cond2} we have $\lim_{j\to\infty}\Phi(x^{k_j})=\Phi(\overline{x})$.
	Thus, by the definition of the limiting subdifferential, 
	$0\in\partial\Phi(\overline{x})$ and the inclusion follows.
\end{proof}

Inequalities \eqref{Phi-cond1}-\eqref{Phi-cond2} are easily satisfied 
by some specific lsc $\Phi$; see Sections \ref{sec4}-\ref{sec5}. 
Then, Lemma \ref{lemma1-Phi} (ii) provides the convergence of $\{\Phi(x^k)\}_{k\in\mathbb{N}}$
under a weaker condition than the continuity of $\Phi$ as required in 
\cite{Wright09,Hager11,Gong13,Kanzow21}. In the sequel, we let $\{x^k\}_{k\in\mathbb{N}}$ 
be a sequence satisfying conditions H1-H2 and denote by $\omega(x^0)$ its cluster point set, 
and write $\Xi_k:=\|x^{\ell(k)}\!-\!x^{\ell(k)-1}\|$ for each $k\in\mathbb{N}$. 
\subsection{Global convergence}\label{sec3.1}
By the proof of \cite[Lemma 5]{Bolte14}, under condition \eqref{Phi-cond1}, 
the set $\omega(x^0)$ is also connected. Together with Lemma \ref{lemma1-Phi} (i), 
when $\Phi$ satisfying \eqref{Phi-cond1} is such that every point of $\varpi(x^0)$ 
is isolated, it is immediate to obtain the convergence of $\{x^k\}_{k\in\mathbb{N}}$.
This section focuses on the convergence of $\{x^k\}_{k\in\mathbb{N}}$ 
under the case that $\varpi(x^0)$ has at least a non-isolated point.
Denote by $\Phi^*$ the limit of $\{\Phi(x^k)\}_{k\in\mathbb{N}}$.    
We need the following technical lemma. 
\begin{lemma}\label{lk-sequence}
	If $\Phi$ is a KL function satisfying \eqref{Phi-cond1}-\eqref{Phi-cond2},
	then $\sum_{k=1}^\infty\Xi_k<\infty$.
\end{lemma}
\begin{proof}
	By Lemma \ref{lemma1-Phi} (i), $\varpi(x^0)$ is a nonempty compact set. 
	By invoking \cite[Lemma 6]{Bolte14} with $\Omega=\varpi(x^0)$ and 
	Lemma \ref{lemma1-Phi} (iii), there exist $\delta>0$, $\eta>0$
	and a function $\varphi\in\Upsilon_{\!\eta}$ such that 
	for all $\overline{x}\in\Omega$ and
	all $x\in[\Phi(\overline{x})<\Phi<\Phi(\overline{x})+\eta]
	\cap\big\{z\in\mathbb{X}\,|\,{\rm dist}(z,\Omega)<\delta\big\}$, 
	$\varphi'(\Phi(x)-\Phi(\overline{x})){\rm dist}(0,\partial\Phi(x))\ge 1$.
	Pick a point $\widetilde{x}\in\varpi(x^0)$. Then, there exists a subsequence 
	$\{x^{k_j}\}_{j\in\mathbb{N}}$ such that $\lim_{j\to\infty}x^{k_j}=\widetilde{x}$. 
	Also, from the proof of Lemma \ref{lemma1-Phi} (iii), 
	$\lim_{j\to\infty}\Phi(x^{k_j})=\Phi(\widetilde{x})=\lim_{k\to\infty}\Phi(x^{k})=\Phi^*$. 
	
	\noindent
	{\bf Case 1: there exists $\overline{k}\in\mathbb{N}$ such that
		$\Phi(x^{\ell(\overline{k})})=\Phi^*$.} 
	From \eqref{lk-monotone} and $\lim_{k\to\infty}\Phi(x^{\ell(k)})=\Phi^*$,
	we have $\Phi(x^{\ell(k)})=\Phi^*$ for all $k\ge\overline{k}$.
	The result then follows by Lemma \ref{lemma1-Phi} (iv).
	
	\noindent
	{\bf Case 2: $\Phi(x^{\ell(k)})\ne\Phi^*$ for every $k\in\mathbb{N}$.}
	In this case, from \eqref{lk-monotone} and $\lim_{k\to\infty}\Phi(x^{\ell(k)})=\Phi^*$,
	$\Phi(x^{\ell(k)})>\Phi^*$ for every $k\in\mathbb{N}$ 
	and there exists $\widehat{k}\in\mathbb{N}$ such that 
	$\Phi(x^{\ell(k)})<\Phi^*+\eta$ for all $k\ge\widehat{k}$. 
	Since $\lim_{k\to\infty}{\rm dist}(x^k, \varpi(x^0))=0$, 
	for all $k\ge\widehat{k}$ (if necessary by increasing $\widehat{k}$), 
	${\rm dist}(x^k, \varpi(x^0))<\delta$. Thus, for all $k\ge\widehat{k}$, 
	$\varphi'(\Phi(x^{\ell(k)})-\Phi^*){\rm dist}(0,\partial\Phi(x^{\ell(k)}))\ge 1$,
	which along with condition H2 implies that 
	$b\Xi_k\varphi'\big(\Phi(x^{\ell(k)})-\Phi^*\big)\ge 1$. 
	By \eqref{lk-monotone} and the concavity of $\varphi$ on $[0,\eta)$, 
	for all $k\ge\widehat{k}$,
	\begin{equation}\label{Phi-ineq1}
		b\Xi_k\big[\varphi\big(\Phi(x^{\ell(k)})-\Phi^*\big)
		\!-\!\varphi\big(\Phi(x^{\ell(k+m+1)})\!-\!\Phi^*\big)\big]
		\ge\Phi(x^{\ell(k)})\!-\!\Phi(x^{\ell(k+m+1)}).
	\end{equation}
	In addition, from condition H1 and $\Phi(x^{\ell(\ell(k+m+1)-1)})\le\Phi(x^{\ell(k)})$,
	it follows that for each $k\in\mathbb{N}$,
	\begin{equation*}
		\Phi(x^{\ell(k+m+1)})\le\Phi(x^{\ell(k)})\!-\!a\|x^{\ell(k+m+1)}\!-\!x^{\ell(k+m+1)-1}\|^2.
	\end{equation*}
	From the last two inequalities, it is not hard to obtain that 
	for every $k\ge\widehat{k}$,
	\begin{equation*}
		\|x^{\ell(k+m+1)}\!-\!x^{\ell(k+m+1)-1}\|
		\le \sqrt{ba^{-1}\Xi_k[\varphi(\Phi(x^{\ell(k)})\!-\!\Phi^*)
			\!-\!\varphi(\Phi(x^{\ell(k+m+1)})\!-\!\Phi^*)]}.
	\end{equation*}
	By using Lemma \ref{lemma1-sequence} for the sequence
	$\{x^{\ell(k+m+1)-1}\}_{k\in\mathbb{N}}$ and $\mathcal{K}=\mathbb{N}$, 
	for any $\nu>k\ge\widehat{k}$,  
	\begin{equation}\label{rate-ineq1}
		({3}/{4}){\textstyle\sum_{j=k+m+1}^{\nu+m+1}}\,\Xi_j
		\le({1}/{4}){\textstyle\sum_{j=k}^{k+m}}\,\Xi_j
		+ba^{-1}{\textstyle\sum_{j=k}^{k+m}}
		\varphi\big(\Phi(x^{\ell(j)})\!-\!\Phi^*\big).
	\end{equation}
	By passing the limit $\nu\to\infty$ to the both sides of this inequality,
	we obtain $\sum_{k=1}^\infty\Xi_k<\infty$.
\end{proof}
\begin{theorem}\label{KL-converge}
	Let $\Phi$ be a KL function satisfying \eqref{Phi-cond1}-\eqref{Phi-cond2}.
	Suppose that
	\begin{equation}\label{assump0}
		\sum_{K_1\ni k=0}^{\infty}\!\sqrt{\Phi(x^{\ell(k+1)})\!-\!\Phi(x^{k+1})}<\infty
		\ \ {\rm when}\ 
		\liminf_{K_1\ni k\to\infty}\frac{\Phi(x^{\ell(k)})-\Phi(x^{\ell(k+1)})}{\|x^{k+1}-x^k\|^2}=0,
	\end{equation}
	where $K_1\!:=\!\big\{k\in\mathbb{N}\,|\,\Phi(x^{\ell(k+1)})-\Phi(x^{k+1})
	\ge\frac{a}{2}\|x^{k+1}\!-\!x^{k}\|^2\big\}$. 
	Then $\sum_{k=0}^\infty\|x^{k+1}\!-\!x^k\|<\infty$.
\end{theorem}
\begin{proof}
	By the proof Lemma \ref{lk-sequence}, it suffices to consider that
	$\Phi(x^{\ell(k)})\ne\Phi^*$ for every $k\in\mathbb{N}$.
	Now inequality \eqref{Phi-ineq1} holds with $m=0$ for all $k\ge\widehat{k}$.
	We proceed the arguments by two cases.
	
	\noindent
	{\bf Case 1: ${\displaystyle\liminf_{k\to\infty}}\frac{\Phi(x^{\ell(k)})-\Phi(x^{\ell(k+1)})}
		{\|x^{k+1}-x^k\|^2}>0$.}
	Now there exists $\gamma>0$ such that for all $k\ge\widehat{k}$ 
	(if necessary by increasing $\widehat{k}$), 
	$\Phi(x^{\ell(k)})-\Phi(x^{\ell(k+1)})\ge\gamma\|x^{k+1}\!-\!x^k\|^2$.
	Along with \eqref{Phi-ineq1} with $m=0$, for all $k\ge\widehat{k}$,
	\begin{equation*}
		\|x^{k+1}\!-\!x^k\|
		\le\sqrt{b\gamma^{-1}\Xi_k[\varphi(\Phi(x^{\ell(k)})\!-\!\Phi^*)
			-\varphi(\Phi(x^{\ell(k+1)})\!-\!\Phi^*)]}.
	\end{equation*}
	By invoking Lemma \ref{lemma1-sequence} for the sequence
	$\{x^{k}\}_{k\in\mathbb{N}}$ and $\mathcal{K}=\mathbb{N}$, 
	it then follows that 
	\[
	{\textstyle\sum_{k=\widehat{k}}^{\nu}}\,\|x^{k+1}\!-\!x^k\|
	<(1/4){\textstyle\sum_{k=\widehat{k}}^\nu}\,\Xi_k
	+b\gamma^{-1}\varphi(\Phi(x^{\ell(\widehat{k})})\!-\!\Phi^*).
	\]
	Passing the limit $\nu\to\infty$ to the both sides and 
	using Lemma \ref{lk-sequence} yields the result.
	
	\noindent
	{\bf Case 2: ${\displaystyle\liminf_{k\to\infty}}
		\frac{\Phi(x^{\ell(k)})-\Phi(x^{\ell(k+1)})}{\|x^{k+1}-x^k\|^2}=0$.}
	From \eqref{Phi-ineq1} for $m=0$ and condition H1, for all $k\ge\widehat{k}$,
	\begin{align*}
		a\|x^{k+1}\!-\!x^k\|^2+\Phi(x^{k+1})\!-\!\Phi(x^{\ell(k+1)})
		\le b\Xi_k[\varphi(\Phi(x^{\ell(k)})\!-\!\Phi^*)
		\!-\!\varphi(\Phi(x^{\ell(k+1)})\!-\!\Phi^*)].
	\end{align*}
	Suppose that $\overline{K}_1\!:=\mathbb{N}\backslash K_1$
	is an infinite set. For all $\overline{K}_1\ni k>\widehat{k}$ 
	(if necessary by increasing $\widehat{k}$),
	$\Phi(x^{k+1})\!-\!\Phi(x^{\ell(k+1)})>\!-\frac{a}{2}\|x^{k+1}\!-\!x^{k}\|^2$,
	which along with the last inequality implies that
	\begin{equation*}
		\|x^{k+1}\!-\!x^k\|
		<\sqrt{2ba^{-1}\Xi_k[\varphi(\Phi(x^{\ell(k)})\!-\!\Phi^*)
			-\varphi(\Phi(x^{\ell(k+1)})\!-\!\Phi^*)]}.
	\end{equation*}
	By using Lemma \ref{lemma1-sequence} for the sequence
	$\{x^{k}\}_{k\in\mathbb{N}}$ and $\mathcal{K}=\overline{K}_1$, 
	it follows that for any $\nu>k\ge\widehat{k}$,   
	\begin{equation}\label{temp-ineq1}
		{\textstyle\sum_{\overline{K}_1\ni j=k}^{\nu}}\,\big\|x^{j+1}\!-\!x^j\big\|
		\le(1/2){\textstyle\sum_{j=k}^{\nu}}\,\Xi_j
		+ba^{-1}\varphi(\Phi(x^{\ell(k)})\!-\!\Phi^*).
	\end{equation}
	If $\overline{K}_1$ is a finite set, then $j\notin\overline{K}_1$ 
	for all $j>\widehat{k}$ (if necessary by increasing $\widehat{k}$), 
	and inequality \eqref{temp-ineq2} holds automatically. 
	Next we consider that $K_1$ is an infinite set. Obviously, for any $\nu>k\ge\widehat{k}$,
	\begin{equation}\label{temp-ineq2}
		{\textstyle\sum_{K_1\ni j=k}^{\nu}}\|x^{j+1}\!-\!x^{j}\|
		\le\sqrt{2a^{-1}}{\textstyle\sum_{K_1\ni j=k}^{\nu}}\sqrt{\Phi(x^{\ell(j+1)})-\Phi(x^{j+1})},
	\end{equation}
	which holds automatically if $K_1$ is a finite set. 
	For any $\nu>k\ge\widehat{k}$, adding \eqref{temp-ineq1} to \eqref{temp-ineq2} yields that
	\[
	\sum_{j=k}^{\nu}\|x^{j+1}\!-\!x^{j}\|
	\le\frac{1}{2}\sum_{j=k}^{\nu}\Xi_j
	+\frac{b\varphi(\Phi(x^{\ell(k)})\!-\!\Phi^*)}{a}
	+\!\sqrt{\frac{2}{a}}\!\sum_{K_1\ni j=k}^{\nu}\!\sqrt{\Phi(x^{\ell(j+1)})-\Phi(x^{j+1})}.
	\]
	Passing the limit $\nu\to\infty$ and using Lemma \ref{lk-sequence} 
	and the given assumption yields the result.
\end{proof}

	Now we take a closer look at condition \eqref{assump0}. We first show that 
	it is sufficient and necessary for $\sum_{k=0}^{\infty}\|x^{k+1}-x^k\|<\infty$ 
	if $\Phi$ is a KL function that is weakly convex on a neighborhood of $\omega(x^0)$.
	\begin{proposition}\label{SN-cond1}
		Suppose that $\Phi$ is a KL function satisfying \eqref{Phi-cond1}-\eqref{Phi-cond2} 
		and having a $\rho$-weak convexity on a neighborhood of $\omega(x^0)$. 
		Then, $\sum_{k=0}^{\infty}\|x^{k+1}-x^k\|<\infty$ iff
		condition \eqref{assump0} holds.
	\end{proposition}
	\begin{proof}
		By Theorem \ref{KL-converge}, it suffices to prove the necessity. 
		Suppose that $\sum_{k=0}^{\infty}\|x^{k+1}-x^k\|<\infty$. Let 
		$\mathcal{N}_0$ be a neighborhood of $\omega(x^0)$ such that 
		$\Phi$ is $\rho$-weakly convex on $\mathcal{N}_0$. 
		From Lemma \ref{lemma1-Phi} (ii), we have
		$\lim_{k\to\infty}\|x^{k}-x^{\ell(k)}\|=0$. Note that 
		$\lim_{k\to\infty}{\rm dist}(x^k,\omega(x^0))=0$. It is easy to
		argue that there exists $\widehat{k}_1\ge m$ such that 
		for all $k\ge\widehat{k}_1$, $x^{k},x^{l(k)}\in\mathcal{N}_0$. 
		Then, for each $k\ge\widehat{k}_1$ and $w^{\ell(k)}\in\partial\Phi(x^{\ell(k)})$,
		\begin{equation*}
			\Phi(x^{\ell(k)})\!-\!\Phi(x^k)
			\le\langle w^{\ell(k)},x^{\ell(k)}\!-\!x^k\rangle+\frac{\rho}{2}\|x^k\!-\!x^{\ell(k)}\|^2
			\le\frac{b^2}{2}\|x^{\ell(k)}\!-\!x^{\ell(k)-1}\|^2
			+\frac{\rho\!+\!1}{2}\|x^k\!-\!x^{\ell(k)}\|^2,
		\end{equation*}
		where the last inequality is using condition H2. 
		Together with the definitions of $\ell(k)$ and $\Xi_k$, 
		$\sqrt{\Phi(x^{\ell(k)})\!-\!\Phi(x^k)}\le\sqrt{b^2/2}\,\Xi_k
		+\sqrt{(\rho\!+\!1)/2}\sum_{j=k-\widehat{k}_1}^{k-1}
		\big\|x^{j+1}\!-\!x^{j}\big\|$. 
		Then, for each $\nu>\widehat{k}_1$, 
		\begin{equation}\label{temp-ineq3}
			{\textstyle\sum_{k=\widehat{k}_1}^{\nu}}\sqrt{\Phi(x^{\ell(k)})-\Phi(x^k)}
			\le \sqrt{b^2/2}\,{\textstyle\sum_{k=\widehat{k}_1}^{\nu}}\Xi_k+\sqrt{(\rho\!+\!1)/2}\,
			{\textstyle\sum_{k=\widehat{k}_1}^{\nu}\sum_{j=k-\widehat{k}_1}^{k-1}}
			\big\|x^{j+1}\!-\!x^{j}\big\|.
		\end{equation}
		Passing the limit $\nu\to\infty$ to this inequality and using 
		Lemma \ref{lk-sequence} and $\sum_{k=0}^{\infty}\|x^{k+1}-x^k\|<\infty$ yields that
		$\sum_{k=\widehat{k}_1}^{\infty}\!\sqrt{\Phi(x^{\ell(k)})-\Phi(x^k)} <\infty$. 
		That is, condition \eqref{assump0} holds.
	\end{proof}
	
	Next we provide a condition independent of $\{x^k\}_{k\in\mathbb{N}}$ 
	to ensure that \eqref{assump0} holds. This condition is satisfied by a class of 
	$\rho$-weakly convex functions with $\rho\le\frac{a}{2(m+1)^2}$ and so
	by convex functions. 
	\begin{lemma}\label{Scond1-prop}
		If $\Phi$ is a KL function satisfying \eqref{Phi-cond1}-\eqref{Phi-cond2}, 
		then condition \eqref{assump0} holds whenever there exist a constant 
		$a_0\in[0,\frac{a}{8(m+1)^2}]$ and a neighborhood 
		of $\omega(x^0)$, denoted by $\mathcal{N}_0$, such that 
		\begin{equation}\label{scond}
			{\rm dist}(0,\partial\Phi(z))\|y-z\|+a_0\|y-z\|^2\ge\Phi(z)-\Phi(y)
			\quad\forall y,z\in\mathcal{N}_0\backslash\omega(x^0).
		\end{equation}
	\end{lemma}
	\begin{proof}
		It suffices to consider that $K_1$ is an infinite set and 
		$\Phi(x^{\ell(k)})\ne\Phi^*$ for every $k\in K_1$.
		By the proof of Lemma \ref{SN-cond1}, there exists $\widehat{k}_1\ge m$ such that 
		for all $k\ge\widehat{k}_1$, $x^{k},x^{l(k)}\in\mathcal{N}_0\backslash\omega(x^0)$. 
		For each $K_1\ni j\ge\widehat{k}_1$, let $w^{\ell(j+1)}\in\partial\Phi(x^{\ell(j+1)})$ 
		be such that ${\rm dist}(0,\partial\Phi(x^{\ell(j+1)}))=\|w^{\ell(j+1)}\|$. 
		Together with condition H2 and the condition in \eqref{scond}, 
		for each $j\ge\widehat{k}_1$, it holds that 
		\begin{align*}
			\sqrt{\Phi(x^{\ell(j+1)})-\Phi(x^{j+1})}
			&\le\sqrt{b\|x^{\ell(j+1)}\!-\!x^{\ell(j+1)-1}\|\|x^{\ell(j+1)}\!-\!x^{j+1}\|}
			\!+\!\sqrt{\frac{a}{8(m\!+\!1)^2}}\|x^{j+1}\!-\!x^{\ell(j+1)}\|\\
			&\le \frac{\sqrt{2}(m\!+\!1)^2b}{\sqrt{a}}\,\Xi_{j+1}
			+\Big[\frac{\sqrt{a}}{4\sqrt{2}(m\!+\!1)^2}\!+\!\sqrt{\frac{a}{8(m\!+\!1)^2}}\Big]
			\sum_{l=j-m}^{j}\|x^{l+1}\!-\!x^{l}\|.
		\end{align*}
		For any $K_1\!\ni\nu>k\!\ge\widehat{k}_1$, summing the last inequality from 
		$j=k$ to $\nu$ yields that 
		\begin{align}\label{temp-ineq4}
			\sum_{K_1\ni j=k}^{\nu}\!\sqrt{\Phi(x^{\ell(j+1)})-\Phi(x^{j+1})}
			\le\frac{\sqrt{2}(m\!+\!1)^2b}{\sqrt{a}}\sum_{j=k}^{\nu}\!\Xi_{j+1}
			+a_1\sum_{K_1\ni j=k}^{\nu}\sum_{l=j-m}^{j}\|x^{l+1}\!-\!x^{l}\|.
		\end{align}
		with $a_1=\!\frac{\sqrt{a}}{4\sqrt{2}(m+1)^2}\!+\!\sqrt{\frac{a}{8(m+1)^2}}$. 
		Note that $\sum_{l=j-m}^{j}\|x^{l+1}\!-\!x^{l}\|= 
		\sum_{K_1\ni l=j-m}^{j}\|x^{l+1}\!-\!x^{l}\|
		+\sum_{\overline{K}_1\ni l=j-m}^{j}\|x^{l+1}\!-\!x^{l}\|$, 
		while 
		$\sum_{K_1\ni j=k}^{\nu}\sum_{K_1\ni l=j-m}^{j}\|x^{l+1}\!-\!x^{l}\|
		\le(m\!+\!1)\!\sum_{K_1\ni j=k-m}^{\nu}\|x^{j+1}\!-\!x^{j}\|$ and 
		$\sum_{K_1\ni j=k}^{\nu}\sum_{\overline{K}_1\ni l=j-m}^{j}\|x^{l+1}\!-\!x^{l}\|
		\le(m\!+\!1)\!\sum_{\overline{K}_1\ni j=k-m}^{\nu}\|x^{j+1}\!-\!x^{j}\|$. 
		Then, together with $\|x^{j+1}\!-\!x^j\|\le
		\!\sqrt{\frac{2}{a}}\sqrt{\Phi(x^{\ell(j+1)})\!-\!\Phi(x^{j+1})}$ 
		for each $j\in K_1$ and inequalities \eqref{temp-ineq4} and \eqref{temp-ineq1}, 
		\begin{align*}
			\frac{2m\!+\!1}{4(m\!+\!1)}
			\!\sum_{K_1\ni j=k}^{\nu}\!\sqrt{\Phi(x^{\ell(j+1)})-\Phi(x^{j+1})}
			&\le\!\frac{\sqrt{2}(m\!+\!1)^2b}{\sqrt{a}}\sum_{j=k}^{\nu}\!\Xi_{j+1}
			+a_1(m\!+\!1)\sum_{j=k-m}^{k-1}\|x^{j+1}-x^j\|\\ 
			&\quad\!+\!\frac{a_1(m+1)}{2}\Big[\sum_{j=k}^{\nu}\,\Xi_j
			\!+\!2ba^{-1}\varphi(\Phi(x^{\ell(k)})\!-\!\Phi^*)\Big].
		\end{align*}
		This by Lemma \ref{lk-sequence} implies that condition \eqref{assump0} holds.
		The proof is completed.
	\end{proof} 

\subsection{Convergence rate}\label{sec3.2}

In this subsection, we establish the convergence rate of $\{x^k\}_{k\in\mathbb{N}}$ 
under the assumption that $\Phi$ is a KL function associated to 
$\varphi(t):=ct^{1-\theta}$ for $t\in[0,\infty)$ with $\theta\in[0,1)$ 
and $c>0$. For this purpose, we need the following two technical lemmas.
\begin{lemma}\label{lemma-sequence}
	Let $\{\Gamma_{\!l}\}_{l\in\mathbb{N}}$ be a nonnegative nonincreasing sequence 
	such that for all $l\ge\overline{l}$ with some $\overline{l}\in\mathbb{N}$,
	$\Gamma_{\!l}\le C\max\big\{l^{\frac{1-\theta}{1-2\theta}},
	(\Gamma_{\!l-m-1}\!-\!\Gamma_{\!l})^{\frac{1-\theta}{\theta}}\big\}$, 
	where $\theta\in(1/2,1)$ and $C>0$ are the constants. 
	Then, there exists $\widetilde{\gamma}_1>0$ such that for all $l\ge\overline{l}$, 
	$\Gamma_{\!l}\le\max\big\{C,\widetilde{\gamma}_1^{\frac{1}{\mu}}\big\}
	(\frac{l-\overline{l}}{m+2})^{\frac{1}{\mu}}$ 
	with $\mu:=\frac{1-2\theta}{1-\theta}$.  
\end{lemma}
\begin{proof}
	Fix any $l\ge\overline{l}$. If there exists
	$i\in[\frac{l-\overline{l}}{m+2}+\overline{l},l]\cap\mathbb{N}$ such that
	$i^{\frac{1}{\mu}}\!\ge(\Gamma_{i-m-1}\!-\!\Gamma_i)^{\frac{1-\theta}{\theta}}$,
	then $\Gamma_{\!l}\le\Gamma_{\!i}\le C{i}^{\frac{1}{\mu}}
	\le C(\frac{l-\overline{l}}{m+2}\!+\!\overline{l})^{\frac{1}{\mu}}
	\le C(\frac{l-\overline{l}}{m+2})^{\frac{1}{\mu}}$,
	where the last two inequalities are due to $\frac{1}{\mu}<0$.  
	Thus, the conclusion holds for this case. 
	Suppose that for all $i\in[\frac{l-\overline{l}}{m+2}\!+\!\overline{l},l]\cap\mathbb{N}$,
	${i}^{\frac{1}{\mu}}\le(\Gamma_{i-m-1}\!-\!\Gamma_i)^{\frac{1-\theta}{\theta}}$. 
	Then $\Gamma_{i}\le C\big(\Gamma_{i-m-1}\!-\!\Gamma_i\big)^{\frac{1-\theta}{\theta}}$.
	By using the same analysis technique as in \cite[Page 14]{Attouch09}, 
	for all $i\in[\frac{l-\overline{l}}{m+2}\!+\!\overline{l},l]\cap\mathbb{N}$,  
	we have $\Gamma_{i}^{\mu}-\Gamma_{i-m-1}^{\mu}\ge\widetilde{\gamma}_1$ 
	for some $\widetilde{\gamma}_1>0$ (if necessary by increasing $\overline{l}$), 
	which implies that $\Gamma_{l}^{\mu}-\Gamma_{l-q(m+1)}^{\mu}\ge q\widetilde{\gamma}_1$ 
	with $q=\lfloor\frac{l-\frac{l-\overline{l}}{m+2}-\overline{l}}{m+1}\rfloor+1
	=\lfloor\frac{l-\overline{l}}{m+2}\rfloor+1$, and consequently, 
	\(
	\Gamma_{l}\le(\Gamma_{l-q(m+1)}+q\widetilde{\gamma}_1)^{1/\mu}
	\le\widetilde{\gamma}_1^{1/\mu}q^{1/\mu}
	\le\widetilde{\gamma}_1^{1/\mu}\lfloor\frac{l-\overline{l}}{m+2}\rfloor^{1/\mu}.
	\)
	The desired result then follows.   
\end{proof} 
\begin{lemma}\label{KL-Xirate}
	Suppose that $\Phi$ is a KL function associated to $\varphi$ and satisfies
	\eqref{Phi-cond1}-\eqref{Phi-cond2}. Then there exist $\overline{k}\in\mathbb{N}$
	and constants $\widehat{\varrho}\in(0,1),\widehat{\gamma}>0$ 
	and $\gamma'>0$ such that for all $k\ge\overline{k}$,
	\[
	\sum_{j=k}^{\infty}\Xi_{j}\le\left\{\begin{array}{cl}
		\!\widehat{\gamma}\widehat{\varrho}^{\lfloor\frac{k-1}{m+1}\rfloor} 
		&{\rm if}\ \theta\in(0,\frac{1}{2}],\\
		\widehat{\gamma}{k}^{\frac{1-\theta}{1-2\theta}}&{\rm if}\ \theta\in(\frac{1}{2},1)
	\end{array}\right.\,{\rm and}\ 
	\Phi(x^k)-\Phi^*\le\left\{\begin{array}{cl}
		\!\gamma'\widehat{\varrho}^{\lfloor\frac{k-1}{m+1}\rfloor} 
		&{\rm if}\ \theta\in(0,\frac{1}{2}],\\
		\gamma'{k}^{\frac{1-\theta}{1-2\theta}}&{\rm if}\ \theta\in(\frac{1}{2},1).
	\end{array}\right.    
	\]
\end{lemma}
\begin{proof}
	It suffices to consider that $\Phi(x^{\ell(k)})\ne\Phi^*$ 
	for every $k\in\mathbb{N}$. By the proof of Lemma \ref{lk-sequence}, 
	$b\Xi_k\varphi'(\Phi(x^{\ell(k)})-\Phi(\widetilde{x}))\ge 1$ 
	for all $k\ge\widehat{k}$. Together with the expression of $\varphi$
	and condition H2, for all $k\ge\widehat{k}$,  
	$(\Phi(x^{\ell(k)})-\Phi^*)^{\theta}
	\le bc(1\!-\!\theta)\|x^{\ell(k)}\!-\!x^{\ell(k)-1}\|$,
	and consequently,
	\[
	\varphi(\Phi(x^{\ell(k)})\!-\!\Phi^*)
	=c(\Phi(x^{\ell(k)})\!-\!\Phi^*)^{1-\theta}
	\le c\big[bc(1\!-\!\theta)\|x^{\ell(k)}-x^{\ell(k)-1}\|\big]^{\frac{1-\theta}{\theta}}.
	\]
	For each $k\in\mathbb{N}$, let $\Lambda_k:=\sum_{j=k}^{\infty}\Xi_{j}$. 
	By combining the last inequality with \eqref{rate-ineq1}, 
	for any $\nu>k\ge\widehat{k}$,
	\begin{equation}\label{Xik-ineq1}
		\frac{3}{4}\Lambda_{k+m+1}
		\le\frac{1}{4}{\textstyle\sum_{j=k}^{k+m}}\,\Xi_j
		+{bc}a^{-1}\big[bc(1-\theta)\big]^{\frac{1-\theta}{\theta}}
		{\textstyle\sum_{j=k}^{k+m}}\,\Xi_j^{\frac{1-\theta}{\theta}}.
	\end{equation}
	When $\theta\in(0,1/2]$, since $\frac{1-\theta}{\theta}\ge 1$ and
	$\Xi_k<1$ for all $k\ge\widehat{k}$ (if necessary by increasing $\widehat{k}$), 
	we have $\Lambda_{k+m+1}\le M(\Lambda_{k}-\Lambda_{k+m+1})$ 
	with $M=\frac{1}{3}+\frac{4bc}{3a}[bc(1-\theta)]^{\frac{1-\theta}{\theta}}$,
	which implies that $\Lambda_{k}\le\frac{M}{1+M}\Lambda_{k-m-1}$ 
	for all $k\ge\widehat{k}+m\!+\!1$. From this recursion formula, we obtain 
	$\Lambda_{k}\le(\frac{M}{1+M})^{\lfloor\frac{k-1}{m+1}\rfloor}\Lambda_1$.
	The result holds with $\widehat{\varrho}=\frac{M}{1+M}$ and $\widehat{\gamma}=\Lambda_1$.
	When $\theta\in(1/2,1)$, from \eqref{Xik-ineq1} it follows that 
	for all $k\ge\widehat{k}+m\!+\!1$,
	\[
	\Lambda_{k+m+1}\le M{\textstyle\sum_{j=k}^{k+m}}\,\Xi_j^{\frac{1-\theta}{\theta}}
	\le M(m\!+\!1)^{\frac{2\theta-1}{\theta}}
	\big[{\textstyle\sum_{j=k}^{k+m}}\,\Xi_j\big]^{\frac{1-\theta}{\theta}}
	\le M(m\!+\!1)^{\frac{2\theta-1}{\theta}}
	\big(\Lambda_{k}\!-\!\Lambda_{k+m+1}\big)^{\frac{1-\theta}{\theta}},
	\]
	where the second inequality is by the concavity of $t^{\frac{1-\theta}{\theta}}\ (t>0)$.
	Using Lemma \ref{lemma-sequence} yields the result. 
	
	The second part follows by noting that $\Phi(x^k)\!-\!\Phi^*
	\le[bc(1\!-\!\theta)]^{\frac{1}{\theta}}\|x^{\ell(k)}\!-\!x^{\ell(k)-1}\|^{\frac{1}{\theta}}$
	for all $k\ge\widehat{k}$, and $\|x^{\ell(k)}\!-\!x^{\ell(k)-1}\|^{\frac{1}{\theta}}\le\Xi_k$
	because $\|x^{\ell(k)}\!-\!x^{\ell(k)-1}\|<1$ (if necessary by increasing $\widehat{k}$).
\end{proof}

Now we are ready to analyze the convergence rate of $\{x^k\}_{k\in\mathbb{N}}$ under a suitable assumption.
\begin{theorem}\label{KL-rate}
	Suppose that $\Phi$ is a KL function associated to $\varphi$ and satisfies
	\eqref{Phi-cond1}-\eqref{Phi-cond2}.
	\begin{itemize}
		\item [(i)] When $\theta=0$, $\{x^k\}_{k\in\mathbb{N}}$ converges to 
		a point $\widetilde{x}\in\varpi(x^0)$ in a finite number of steps.
		
		\item [(ii)]When $\theta\in(0,1)$, if there exist $\widetilde{k}_0\in\mathbb{N}$,
		$\widetilde{\gamma}>0$ and $\widetilde{\tau}\in(0,1)$ such that
		for all $k\ge\widetilde{k}_0$,
		\begin{align}\label{rate-cond}
			\sum_{K_2\cup K_{31}\ni j=k}^{\infty}\!
			\sqrt{\Phi(x^{\ell(j+1)})\!-\!\Phi(x^{j+1})}
			\le\left\{\begin{array}{cl}
				\widetilde{\gamma}\widetilde{\tau}^k &{\rm if}\ \theta\!\in(0,1/2],\\
				\!\widetilde{\gamma}k^{\frac{1-\theta}{1-2\theta}}&{\rm if}\ \theta\!\in(1/2,1),
			\end{array}\right.
		\end{align}
		where $K_{2}\!:=\!\big\{k\in\mathbb{N}\,|\,\frac{a}{2}\|x^{k+1}\!-\!x^{k}\|^2\!
		\le\Phi(x^{\ell(k+1)})\!-\!\Phi(x^{k+1})
		<\frac{a}{2}\|x^{k+1}\!-\!x^{k}\|^{\frac{1}{\theta}}\big\}$
		and $K_{31}\!:=\!\big\{k\in K_1\backslash K_{2}\ |\ 
		\Phi^*\!-\!\Phi(x^{k+1})
		>\frac{a}{4}\|x^{k+1}\!-\!x^{k}\|^{\frac{1}{\theta}}\big\}$,
		then the sequence $\{x^k\}_{k\in\mathbb{N}}$ converges to 
		a point $\widetilde{x}\in\varpi(x^0)$ and there exist $\gamma>0$ 
		and $\varrho\in(0,1)$ such that for all sufficiently large $k$,
		\begin{equation}\label{aim-ineq-rate}
			\|x^k-\widetilde{x}\|\le\Delta_k\le\left\{\begin{array}{cl}
				\gamma\varrho^{k} &{\rm if}\ \theta\in(0,1/2],\\
				\gamma k^{\frac{1-\theta}{1-2\theta}}&{\rm if}\ \theta\in(1/2,1)
			\end{array}\right.{\rm with}\ \Delta_k:=\sum_{j=k}^{\infty}\|x^{j+1}\!-\!x^j\|.
		\end{equation}
	\end{itemize}
\end{theorem}
\begin{proof}
	{\bf(i)} We argue that there exists $\overline{k}\in\mathbb{N}$ 
	such that $\Phi(x^{\ell(\overline{k})})=\Phi^*$, 
	and the result then follows by the proof of Lemma \ref{lk-sequence}. If not, 
	by the proof of Lemma \ref{lk-sequence}, $\sum_{k=1}^{\infty}\Xi_k<\infty$. 
	On the other hand, since $\Phi$ has the KL property of exponent $\theta=0$ at $\widetilde{x}$, 
	for all $k\ge\widehat{k}$, $b\Xi_k\varphi'(\Phi(x^{\ell(k)})-\Phi^*)\ge 1$ 
	holds with $\varphi(t)=ct$ for $t\in[0,\infty)$. Then, for any $\nu>\widehat{k}$, 
	$\sum_{k=\widehat{k}}^{\nu}\Xi_k\ge\frac{\nu-\widehat{k}+1}{bc}$.
	Passing the limit $\nu\to\infty$ to this inequality yields that
	$\sum_{k=\widehat{k}}^{\infty}\Xi_k=\infty$. Thus, we get a contradiction.
	
	\noindent
	{\bf(ii)} It suffices to consider that $\Phi(x^{\ell(k)})\ne\Phi^*$ 
	for every $k\in\mathbb{N}$. From the proof of Lemma \ref{lk-sequence},
	\begin{equation}\label{ineq-Phi}
		(\Phi(x^{\ell(k)})\!-\!\Phi^*)^{\theta}\le bc(1\!-\!\theta)\Xi_k
		\quad\ {\rm for\ all}\ k\ge\widehat{k}.
	\end{equation}
	
	\noindent
	{\bf Step 1: to deal with the summation associated to $\overline{K}_1$.}
	From Case 2 in the proof of Theorem \ref{KL-converge}, 
	inequality \eqref{temp-ineq2} holds for any $\nu\!>k\ge\widehat{k}$. 
	In addition, from \eqref{ineq-Phi}, for all $k\ge\widehat{k}$,
	$\varphi(\Phi(x^{\ell(k)})\!-\!\Phi^*)
	=c(\Phi(x^{\ell(k)})\!-\!\Phi^*)^{1-\theta}
	\le c[bc(1-\theta)\Xi_k]^{\frac{1-\theta}{\theta}}$.
	By substituting this inequality into \eqref{temp-ineq2} 
	and writing $\widehat{c}_1(\theta):={bc}a^{-1}[bc(1-\theta)]^{\frac{1-\theta}{\theta}}$,
	for any $\nu>k\ge\widehat{k}$ it holds that
	\begin{equation}\label{ineq-K1}
		{\textstyle\sum_{\overline{K}_{\!1}\ni j=k}^{\nu}}\|x^{j+1}\!-\!x^j\|
		\le({1}/{2}){\textstyle\sum_{j=k}^{\nu}}\Xi_j
		+\widehat{c}_1(\theta)\Xi_k^{\frac{1-\theta}{\theta}}.
	\end{equation}
	
	\noindent
	{\bf Step 2: to deal with the summation associated to $K_2\cup K_{31}$.}
	Since $K_2\cup K_{31}\subseteq K_1$, inequality \eqref{temp-ineq2} continues 
	to hold for any $K_2\cup K_{31}\ni k\ge\widehat{k}$ and any $\nu>k$, i.e., 
	\[
	{\textstyle \sum_{K_2\cup K_{31}\ni j=k}^{\nu}}\,\|x^{j+1}\!-\!x^{j}\|
	\le\sqrt{2{a}^{-1}}{\textstyle\sum_{K_2\cup K_{31}\ni j=k}^{\nu}}
	\sqrt{\Phi(x^{\ell(j+1)})-\Phi(x^{j+1})}.
	\]
	By combining this inequality with the given assumption in \eqref{rate-cond},
	for any $\nu>k\ge\widehat{k}$, we have
	\begin{equation}\label{ineq-K2}
		\sum_{K_2\cup K_{31}\ni j=k}^{\nu}\!\big\|x^{j+1}\!-\!x^{j}\big\|
		\le\left\{\begin{array}{cl}
			\sqrt{2{a}^{-1}}\,\widetilde{\gamma}\widetilde{\tau}^k &{\rm if}\ \theta\in(0,1/2],\\
			\!\sqrt{{2}{a}^{-1}}\,\widetilde{\gamma}k^{\frac{1-\theta}{1-2\theta}}
			&{\rm if}\ \theta\in(1/2,1).
		\end{array}\right.
	\end{equation}
	
	\noindent
	{\bf Step 3: to deal with the summation associated to 
		$K_{32}\!:=K_1\backslash(K_2\cup K_{31})$.}
	By the definition of $K_{32}$, for any $K_{32}\ni k\ge\widehat{k}$, 
	we have $\Phi^*\!-\!\Phi(x^{k+1})
	\le\frac{a}{4}\|x^{k+1}\!-\!x^k\|^{\frac{1}{\theta}}$ 
	and $\Phi(x^{\ell(k+1)})\!-\!\Phi(x^{k+1})
	\ge\frac{a}{2}\|x^{k+1}\!-\!x^{k}\|^{\frac{1}{\theta}}$.
	Together with inequality \eqref{ineq-Phi}, it follows that 
	for any $K_{32}\ni k\ge\widehat{k}$,
	\[
	\Xi_{k+1}\ge[bc(1\!-\!\theta)]^{-1}(\Phi(x^{\ell(k+1)})\!-\!\Phi^*)^{\theta}
	\ge [bc(1\!-\!\theta)]^{-1}(a/4)^\theta\|x^{k+1}\!-\!x^k\|.
	\]
	Then, for any $\nu>k\ge\widehat{k}$, summing the last inequality 
	from $k$ to $\nu$ yields that
	\begin{equation}\label{ineq-K3}
		{\textstyle \sum_{K_{32}\ni j=k}^{\nu}}\,\|x^{j+1}\!-\!x^j\|
		\le \widehat{c}_2(\theta){\textstyle\sum_{K_{32}\ni j=k}^{\nu}}\,\Xi_{j+1}
		\ \ {\rm with}\ \ \widehat{c}_2(\theta):=bc(1\!-\!\theta)(4/a)^{\theta}.
	\end{equation}
	By adding inequalities \eqref{ineq-K1}-\eqref{ineq-K3} together,
	for any $\nu>k\ge\widetilde{k}:=\max\{\widetilde{k}_0,\widehat{k}\}$ 
	it holds that
	\begin{equation}\label{temp-keyineq}
		\!\sum_{j=k}^{\nu}\|x^{j+1}\!-\!x^j\|\le
		\!\left\{\begin{array}{cl}
			\!\frac{1}{2}\sum_{j=k}^{\nu}\Xi_j
			\!+\!\widehat{c}_1(\theta)\,\Xi_{k}^{\frac{1-\theta}{\theta}}
			\!+\widehat{c}_2(\theta)\!{\displaystyle\sum_{K_{32}\ni j=k}^{\nu}}\!\Xi_{j+1}
			\!+\!\sqrt{\frac{2}{a}}\widetilde{\gamma}\widetilde{\tau}^{k}
			&{\rm if}\ \theta\in(0,\frac{1}{2}],\\
			\!\frac{1}{2}\sum_{j=k}^{\nu}\Xi_j\!+\!\widehat{c}_1(\theta)\,
			\Xi_{k}^{\frac{1-\theta}{\theta}}
			\!+\widehat{c}_2(\theta)\!{\displaystyle\sum_{K_{32}\ni j=k}^{\nu}}\!\Xi_{j+1}
			\!+\!\sqrt{\frac{2}{a}}\widetilde{\gamma}k^{\frac{1-\theta}{1-2\theta}}
			&{\rm if}\ \theta\in(\frac{1}{2},1).
		\end{array}\right.
	\end{equation}
	Passing the limit $\nu\to\infty$ and using Lemma \ref{lk-sequence} yields that 
	$\sum_{j=k}^{\infty}\|x^{j+1}\!-\!x^j\|<\infty$.
	
	For the second part, by the definition of $\Delta_k$ and the triangle inequality,
	$\|x^k-\widetilde{x}\|\le\Delta_k$, so we only need to prove the second inequality
	in \eqref{aim-ineq-rate} by the two cases $\theta\in(0,\frac{1}{2}]$ 
	and $\theta\in(\frac{1}{2},1)$.
	
	\noindent
	{\bf Case 1: $\theta\in(0,\frac{1}{2}]$.} Since $\{x^k\}_{k\in\mathbb{N}}$ is convergent,
	we have $\Xi_k<1$ for all $k\ge\widetilde{k}$ (if necessary by increasing $\widetilde{k}$).
	Note that $\frac{1-\theta}{\theta}\ge 1$. From \eqref{temp-keyineq} and
	the definitions of $\Xi_k$ and $\Delta_k$, for any $\nu>k\ge\widetilde{k}$,
	\begin{equation*}
		{\textstyle\sum_{j=k}^{\nu}}\,\|x^{j+1}\!-\!x^j\|
		\le(1/2){\textstyle\sum_{j=k}^{\nu}}\,\Xi_j
		+\widehat{c}_1(\theta)(\Delta_{k-m-1}\!-\!\Delta_{k})
		+\sqrt{2{a}^{-1}}\,\widetilde{\gamma}\widetilde{\tau}^k
		+\widehat{c}_2(\theta){\textstyle\sum_{j=k+1}^{\nu+1}}\Xi_{j}.
	\end{equation*}
	By passing the limit $\nu\to\infty$ and using Lemma \ref{KL-Xirate},
	there exist $\widehat{\varrho}\in(0,1)$ and $\widehat{\gamma}>0$ such that 
	for all $k\ge\widetilde{k}$,
	$\Delta_{k}\le\frac{1}{2}\widehat{\gamma}\widehat{\varrho}^{\lfloor\frac{k-1}{m+1}\rfloor}
	+\widehat{c}_1(\theta)\big(\Delta_{k-m-1}\!-\!\Delta_{k}\big)
	+\sqrt{2{a}^{-1}}\widetilde{\gamma}\widetilde{\tau}^k
	+\widehat{c}_2(\theta)\widehat{\gamma}\widehat{\varrho}^{\lfloor\frac{k}{m+1}\rfloor}$.
	Let $\varrho_1\!:=\!\frac{\widehat{c}_1(\theta)}{1+\widehat{c}_1(\theta)}$
	and $\beta=\widehat{\gamma}(0.5\widehat{\rho}^{-\frac{m+2}{m+1}}
	\!+\!\widehat{\rho}^{-1}\widehat{c}_2(\theta))\!+\!\sqrt{2{a}^{-1}}\,\widetilde{\gamma}$.
	For all $k\ge\widetilde{k}$, we have $\Delta_{k}\le\varrho_1\Delta_{k-m-1}+\beta\tau^k$
	with $\tau=\max(\widehat{\varrho}^{\frac{1}{m+1}},\widetilde{\tau})$.
	By using this recursion formula, it then follows that
	\begin{equation}\label{convrate1}
		\Delta_{k}\le\varrho_1^{\lfloor\frac{k-\widetilde{k}}{m+1}\rfloor}
		\Delta_{k-\lfloor\frac{k-\widetilde{k}}{m+1}\rfloor(m+1)}
		+\beta\tau^{k}\Big[1+\frac{\varrho_1}{\tau^{m+1}}+\cdots
		+\big(\frac{\varrho_1}{\tau^{m+1}}\big)^{\lfloor\frac{k-\widetilde{k}}{m+1}\rfloor-1}\Big].
	\end{equation}
	After an elementary calculation respectively for $\frac{\varrho_1}{\tau^{m+1}}>1,
	\frac{\varrho_1}{\tau^{m+1}}=1$ and $\frac{\varrho_1}{\tau^{m+1}}<1$, 
	there exist $\gamma>0$ and $\varrho\in(0,1)$ such that 
	$\Delta_k\le\gamma_1\varrho^k$ holds for all sufficiently large $k$.
	
	\noindent
	{\bf Case 2: $\theta\in(\frac{1}{2},1)$.} In this case, $\frac{1-\theta}{\theta}\le 1$.
	From \eqref{temp-keyineq}, it follows that for any $\nu>k\ge\widetilde{k}$,
	\begin{equation*}
		{\textstyle\sum_{j=k}^{\nu}}\,\|x^{j+1}\!-\!x^j\|
		\le[{1}/{2}+\widehat{c}_2(\theta)]{\textstyle\sum_{j=k}^{\nu+1}}\,\Xi_j
		+\widehat{c}_1(\theta)(\Delta_{k-m-1}\!-\!\Delta_k)^{\frac{1-\theta}{\theta}}
		+\!\sqrt{2{a}^{-1}}\widetilde{\gamma}k^{\frac{1-\theta}{1-2\theta}}.
	\end{equation*}
	By passing the limit $\nu\to\infty$ and using Lemma \ref{KL-Xirate},
	there exists $\widehat{\gamma}>0$ such that for all $k\ge\widetilde{k}$ 
	(if necessary by increasing $\widetilde{k}$),
	$\Delta_{k}\le\big[\widehat{\gamma}\big({1}/{2}+\widehat{c}_2(\theta)\big)
	\!+\!\sqrt{2a^{-1}}\widetilde{\gamma}\big]{k}^{\frac{1-\theta}{1-2\theta}}
	+\widehat{c}_1(\theta)\big(\Delta_{k-m-1}\!-\!\Delta_k\big)^{\frac{1-\theta}{\theta}}$.
	Write $C_1\!:=\widehat{\gamma}({1}/{2}+\widehat{c}_2(\theta))\!+\!\sqrt{2a^{-1}}\widetilde{\gamma}
	+\widehat{c}_1(\theta)$. Then, for all $k\ge\widetilde{k}$, 
	\(
	\Delta_{k}\le C_1\max\big\{{k}^{\frac{1-\theta}{1-2\theta}},
	(\Delta_{k-m-1}\!-\!\Delta_k)^{\frac{1-\theta}{\theta}}\big\}.
	\)
	Now by invoking Lemma \ref{lemma-sequence}, we obtain the desired result. 
	The proof is completed.
\end{proof}

Similar to Proposition \ref{SN-cond1}, we can establish the following conclusion 
for condition \eqref{rate-cond}. 
\begin{proposition}\label{SN-cond2}
	Suppose that $\Phi$ is a KL function associated to $\varphi$ with 
	$\theta\in(0,1)$, and that satisfies \eqref{Phi-cond1}-\eqref{Phi-cond2} 
	and has a $\rho$-weak convexity on a neighborhood of $\omega(x^0)$. 
	Then, condition \eqref{rate-cond} is sufficient and necessary 
	for the conclusion of Theorem \ref{KL-rate} (ii).
\end{proposition}
\begin{proof}
	It suffices to prove the necessity. From inequality \eqref{temp-ineq3}, 
	for any $\nu>\widehat{k}_1$ it holds that
	\[
	{\textstyle\sum_{j=k}^{\nu}}\,\sqrt{\Phi(x^{\ell(j)})\!-\!\Phi(x^j)}
	\le \sqrt{b^2/2}\,\Xi_k+m\sqrt{(\rho\!+\!1)/2}\,
	{\textstyle\sum_{j=k-\widehat{k}_1}^{\nu}}\,\|x^{j+1}\!-\!x^{j}\|.
	\]
	By passing the limit $\nu\to\infty$ and using Lemma \ref{KL-Xirate}
	and \eqref{aim-ineq-rate}, it is immediate to obtain \eqref{rate-cond}.  
\end{proof} 

The following lemma shows condition \eqref{scond} also implies \eqref{rate-cond}. 
\begin{lemma}\label{Scond-rate}
	If $\Phi$ is a KL function associated to $\varphi$ with 
	$\theta\in(0,1)$, and that satisfying \eqref{Phi-cond1}-\eqref{Phi-cond2}, 
	then condition \eqref{rate-cond} holds whenever there exist
	$a_0\in[0,\frac{a}{8(m+1)^2}]$ and a neighborhood 
	of $\omega(x^0)$, denoted by $\mathcal{N}_0$, such that \eqref{scond} holds.
\end{lemma}
\begin{proof}
	From \eqref{ineq-Phi}, for all $k\ge\widehat{k}$,
	$\varphi(\Phi(x^{\ell(k)})\!-\!\Phi^*)
	=c(\Phi(x^{\ell(k)})\!-\!\Phi^*)^{1-\theta}
	\le c[bc(1-\theta)\Xi_k]^{\frac{1-\theta}{\theta}}$.
	Together with the proof of Lemma \ref{Scond1-prop}, for all $k\ge\widehat{k}$,
	\begin{align*}
		\frac{2m\!+\!1}{4(m\!+\!1)}
		\!\sum_{K_1\ni j=k}^{\infty}\!\sqrt{\Phi(x^{\ell(j+1)})-\Phi(x^{j+1})}
		&\le\!\frac{\sqrt{2}(m\!+\!1)^2b}{\sqrt{a}}\sum_{j=k}^{\infty}\!\Xi_{j+1}
		+a_1(m\!+\!1)\sum_{j=k-m}^{k-1}\|x^{j+1}-x^j\|\\ 
		&\quad\!+\!\frac{a_1(m+1)}{2}\Big[\sum_{j=k}^{\infty}\,\Xi_j
		\!+\!2ba^{-1}\varphi(\Phi(x^{\ell(k)})\!-\!\Phi^*)\Big]\\
		&\le c_1\sum_{j=k}^{\infty}\!\Xi_{j}+c_2\sum_{j=k}^{\infty}\!\Xi_{j}^{\frac{1-\theta}{\theta}}
		+a_1(m\!+\!1)\!\sum_{j=k-m}^{k-1}\|x^{j+1}\!-\!x^j\|,
	\end{align*}
	where $c_1=\frac{\sqrt{2}(m\!+\!1)^2b}{\sqrt{a}}+\frac{a_1(m+1)}{2}$ and $c_2=a_1(m+1)ba^{-1}c[bc(1-\theta)]^{\frac{1-\theta}{\theta}}$. 
	By noting that $\sqrt{\frac{a}{2}}\|x^{j+1}\!-\!x^j\|\le
	\!\sqrt{\Phi(x^{\ell(j+1)})\!-\!\Phi(x^{j+1})}$ 
	for each $j\in K_1$, for all $k\ge\widehat{k}$, 
	\[
	\frac{2m\!+\!1}{4(m\!+\!1)}\sqrt{\frac{a}{2}}\!\sum_{K_1\ni j=k}^{\infty}\|x^{j+1}\!-\!x^j\|\le c_1\sum_{j=k}^{\infty}\!\Xi_{j}+c_2\sum_{j=k}^{\infty}\!\Xi_{j}^{\frac{1-\theta}{\theta}}
	+a_1(m\!+\!1)(\Delta_{k-m}-\Delta_{k}).
	\]
	Together with \eqref{ineq-K1} and using the same arguments as those for Case 1 and 2 in the proof of Theorem \ref{KL-rate}, then \eqref{aim-ineq-rate} holds for sufficiently large $k$. Combining  with Proposition \ref{SN-cond2}, the conclusion holds. 
\end{proof}

Finally, we claim that when $\theta\in(0,{1}/{2}]$, 
if there exists $\widetilde{k}_1\in\mathbb{N}$ 
such that $\Phi^*\le\Phi(x^{k})$ for all $k\ge\widetilde{k}_1$,  
then condition \eqref{rate-cond} holds automatically. Indeed, 
in this case, $K_{31}$ contains a finite number of indices by its definition, 
while the set $K_2$ contains a finite number of indices since 
$\lim_{k\to\infty}x^{k+1}\!-\!x^k=0$ by Lemma \ref{lemma1-Phi} (ii). 
Hence, the claimed fact holds. 

 \section{Nonmonotone line search PG with extrapolation}\label{sec4}

  Consider the problem \eqref{Fprob} with a proper lsc $g\!:\mathbb{X}\to\overline{\mathbb{R}}$,
  which is found to arise in many applications such as variable selection
  (see, e.g., \cite{Tibshirani96,Fan01,Zhang10}) in statistics, classification/regression
  in machine learning \cite{Sra12,Curtis17}, and signal processing (see, e.g., \cite{Donoho06,Candes08,Chartrand07}).
  We assume that $g$ is lower bounded and the function $F$ is coercive and is bounded below.
  For this class of nonconvex nonsmooth problems, Yang \cite{Yang21} recently proposed
  a nonmonotone line search PG method with extrapolation (PGenls) and established the convergence
  rate of the objective value sequence respectively for the monotone case and the case
  without extrapolation, under the assumption that $F$ is a KL function of exponent $\theta\in[0,1)$.
  In this section, we apply the convergence results in Section \ref{sec3} to the iterate
  sequence generated by PGenls, and establish its global convergence
  and convergence rate. For any given $\delta>0$, define the function
  \begin{equation}\label{Hdelta}
   H_{\delta}(z):=F(x)+({\delta}/{2})\|x-u\|^2
   \quad{\rm for}\ z:=(x,u)\in\mathbb{X}\times\mathbb{X}.
  \end{equation}
  The detailed iterate steps of the PGenls are described as follows.
 \begin{algorithm}[H]
 \caption{\label{PGMenls}{\bf\,(Nonmonotone line search PG with extrapolation)}}
 \textbf{Initialization:} Select $m\in\!\mathbb{N},\delta\in(0,{1}/{2}),
 0<\alpha<\frac{\delta}{2},0<\!\tau_{\rm min}\!\le\frac{1}{2(\alpha+\delta)+L_{\!f}}\!<\!\tau_{\rm max},
 \beta_{\rm max}\!\ge0$, $\eta_1\in(0,1)$ and $\eta_2\in(0,1)$.
 Choose $x^0\in{\rm dom}g$. Let $x^{-1}=x^0,z^0=(x^0,x^{-1})$ and set $k:=0$.\\
 \textbf{while} the termination condition is not satisfied \textbf{do}
 \begin{enumerate}
  \item  Choose $\beta_{k,0}\in[0,\beta_{\rm max}]$ and $\tau_{k,0}\in[\tau_{\rm min},\tau_{\rm max}]$.

  \item  \textbf{For} $l=0,1,2,\ldots$ \textbf{do}

  \item \quad Let $\beta_k=\beta_{k,0}\eta_1^{l},\tau_k=\max\{\tau_{k,0}\eta_2^{l},\tau_{\rm min}\}$ and
              $y^k=x^k\!+\!\beta_k(x^k\!-\!x^{k-1})$.

  \item \quad Compute $x^{k+1}\in\mathcal{P}_{\!\tau_k}g(y^k\!-\!\tau_k\nabla\!f(y^k))$ and set $z^{k+1}:=(x^{k+1},x^k)$.

  \item \quad If $H_{\delta}(z^{k+1})\le \max_{j=[k-m]_{+},\ldots,k}H_{\delta}(z^j)-\frac{\alpha}{2}\|z^{k+1}-z^k\|^2$, go to Step 7.

  \item  \textbf{end for}

  \item  Set $k\leftarrow k+1$ and go to Step 1.
 \end{enumerate}
 \textbf{end (while)}
 \end{algorithm}
 \begin{remark}\label{remark-PGMenls}
 {\bf(a)} Algorithm \ref{PGMenls} has a little difference from the PGenls proposed by Yang \cite{Yang21}
 in the setting of parameters and the definition of the potential function $H_{\delta}$.
 By Lemma \ref{step-size} below, Algorithm \ref{PGMenls} is well defined.
  When $m=0$, Algorithm \ref{PGMenls} becomes a monotone line search descent method
  with extrapolation and now by using the decrease of $\{H_{\delta}(z^k)\}_{k\in\mathbb{N}}$
  and the analysis technique as in \cite{Attouch13,Bolte14}, one can obtain the convergence
  of the sequence $\{x^k\}_{k\in\mathbb{N}}$ if $F$ is a KL function and its convergence rate
  if $F$ is a KL function of exponent $\theta\in[0,1)$. When $\beta_{\rm max}=0$,
  Algorithm \ref{PGMenls} is a nonmonotone line search PG method, and to the best of our knowledge,
  there are no global convergence and convergence rate results on the sequence
  $\{x^k\}_{k\in\mathbb{N}}$ even though $f$ and $g$ are convex.

 \noindent
 {\bf(b)} A good step-size initialization at each outer iteration can greatly
  reduce the line search cost. Inspired by \cite{Wright09}, we initialize
  $\tau_{k,0}$ for $k\ge 1$ in Step 1 by the Barzilai-Borwein (BB) rule \cite{Barzilai88}:
  \begin{equation}\label{BB-stepsize}
   \tau_{k,0}=\max\Big\{\min\Big\{\frac{\|\Delta z^k\|^2}{\langle\Delta z^k,\Delta\zeta^k\rangle},
    \frac{\langle\Delta z^k,\Delta\zeta^k\rangle}{\|\Delta\zeta^k\|^2},{\tau}_{\rm max}\Big\},{\tau}_{\rm min}\Big\},
  \end{equation}
 where $\Delta z^k\!:=z^k\!-\!z^{k-1}$ and $\Delta\zeta^k\!:=\nabla\!\widetilde{f}(z^k)\!-\!\nabla\!\widetilde{f}(z^{k-1})$
 with $\widetilde{f}(z):=f(x)+({\delta}/{2})\|x-u\|^2$ for $z=(x,u)\in\mathbb{X}\times\mathbb{X}$.
 Inspired by the good performance of the Nesterov's acceleration strategy \cite{Nesterov83},
 we initialize the extrapolation parameter $\beta_{k,0}$ in Step 1 by this rule, that is,
 \begin{equation}\label{Nesterov}
  \beta_{k,0}={(t_{k-1}-1)}/{t_k}\ \ {\rm with}\ \
  t_{k+1}=\big(1+\!\sqrt{1+4t_k^2}\big)/2\ \ {\rm for}\ \ t_{-1}=t_0=1.
  \end{equation}
 \end{remark}
 \begin{lemma}\label{step-size}
  Let $\{x^k\}_{k\in\mathbb{N}}$ be the sequence generated by Algorithm \ref{PGMenls}.
  Then, for each $k\in\mathbb{N}$, when $\beta_k\le\sqrt{\frac{\delta(\tau_k-\tau_k^2L_{\!f})}{4(1+\tau_kL_{\!f})^2}}$,
  the line search criterion in Step 5 is satisfied when $\tau_k\le\frac{1}{2\alpha+2\delta+L_{\!f}}$.
 \end{lemma}
 \begin{proof}
 Using the definition of $x^{k+1}$ and following the analysis of \cite[Lemma 3.1]{Yang21} yields
 \[
  F(x^{k+1})-F(x^k)
  \le -\frac{\tau_k^{-1}\!-\!L_{\!f}}{4}\|x^{k+1}\!-\!x^k\|^2
   +\frac{(\tau_k^{-1}\!+\!L_{\!f})^2}{\tau_k^{-1}\!-\!L_{\!f}}\|x^{k}\!-\!y^k\|^2,
 \]
 which along with $y^{k}=x^k\!+\!\beta_{k}(x^k\!-\!x^{k-1})$ and
 $\beta_k\le\sqrt{\frac{\delta(\tau_k-\tau_k^2L_{\!f})}{4(1+\tau_kL_{\!f})^2}}$ implies that
 \[
  F(x^{k+1})-F(x^k)
  \le -\frac{\tau_k^{-1}\!-\!L_{\!f}}{4}\|x^{k+1}\!-\!x^k\|^2
   +\frac{\delta}{4}\|x^{k}\!-\!x^{k-1}\|^2.
 \]
 Together with the expression of $H_{\delta}$ and $z^k=(x^{k},x^{k-1})$,
 it then follows that
 \begin{align*}
  H_{\delta}(z^{k+1})-H_{\delta}(z^{k})
  &\le -\frac{1\!-\!\tau_k(2\delta\!+\!L_{\!f})}{4\tau_k}\|x^{k+1}\!-\!x^k\|^2
      -\frac{\delta}{4}\|x^{k}\!-\!x^{k-1}\|^2\\
  &\le-\min\Big\{\frac{1\!-\!\tau_k(2\delta\!+\!L_{\!f})}{4\tau_k},\frac{\delta}{4}\Big\}
       \big\|z^{k+1}-z^k\big\|^2
 \end{align*}
 Notice that $\delta\in(0,{1}/{2})$ and $0<\alpha<\delta/2$.
 The line search criterion on Step 5 is satisfied for $m=0$ whenever
 $\tau_k\le\frac{1}{2\alpha+2\delta+L_{\!f}}$, so is the line search criterion
 on Step 5 for a general $m\in\mathbb{N}$.
 \end{proof}
 \subsection{Convergence results of Algorithm \ref{PGMenls}}\label{sec4.1}

 From lines 2-6 of Algorithm \ref{PGMenls}, the sequence $\{z^k\}_{k\in\mathbb{N}}$
 generated by Algorithm \ref{PGMenls} with $\delta\in(0,1/2)$ satisfies the condition H1
 for $\Phi=H_{\delta}$. Recall that $F$ is assumed to be coercive and lower bounded.
 Clearly, $H_{\delta}$ is coercive and lower bounded. By Lemma \ref{lemma1-Phi} (i),
 the sequence $\{z^k\}_{k\in\mathbb{N}}$ is bounded. The following lemma shows that
 it also satisfies the condition H2 for $\Phi=H_{\delta}$, and moreover, $\Phi=H_{\delta}$
 satisfies the conditions \eqref{Phi-cond1}-\eqref{Phi-cond2}.
 \begin{lemma}\label{lemma-PGMenls}
  Let $\{x^k\}_{k\in\mathbb{N}}$ be generated by Algorithm \ref{PGMenls}.
  Then, the following results hold.
  \begin{itemize}
   \item [(i)] $\liminf_{k\to\infty}H_{\delta}(z^{k})
               \ge\lim_{k\to\infty}H_{\delta}(z^{\ell(k)})$.

   \item[(ii)] For each $\{x^{k_q}\}_{q\in\mathbb{N}}$ with $\lim_{q\to\infty}x^{k_q}\to\widehat{x}$,
               \(
                \limsup_{q\to\infty}H_{\delta}(z^{k_q})\le\!H_{\delta}(\widehat{z})
               \)
               for $\widehat{z}=(\widehat{x},\widehat{x})$.

   \item[(iii)] For each $k$, there is $w^{k}\in\partial H_{\delta}(z^{k})$ such that
                $\|w^{k}\|\!\le\!\sqrt{2}\big[(L_{\!f}+\tau_{\rm min}^{-1})(1+\beta_{\rm max})+2\delta\big]\|z^{k}\!-\!z^{k-1}\|$.
  \end{itemize}
 \end{lemma}
 \begin{proof}
 {\bf(i)} For each $k\in\!\mathbb{N}$ and $j\in\!\{0,\ldots,\ell(k)\!-\!1\}$,
 by the definition of $x^{\ell(k)-j}$ in Step 4,
 \begin{align*}
  &\langle\nabla\!f(y^{\ell(k)-j-1}),x^{\ell(k)-j}\!-\!y^{\ell(k)-j-1}\rangle
  +\frac{\|x^{\ell(k)-j}\!-\!y^{\ell(k)-j-1}\|^2}{2\tau_{\ell(k)-j-1}}+g(x^{\ell(k)-j})\\
  &\le \langle\nabla\!f(y^{\ell(k)-j-1}),x^{\ell(k)-j-1}\!-\!y^{\ell(k)-j-1}\rangle
       +\frac{\|x^{\ell(k)-j-1}\!-\!y^{\ell(k)-j-1}\|^2}{2\tau_{\ell(k)-j-1}}+g(x^{\ell(k)-j-1}),
 \end{align*}
 for which, by the definition of $F$, a suitable rearrangement yields that
 \begin{align*}
  F(x^{\ell(k)-j})&\le\! F(x^{\ell(k)-j-1})\!+\!f(x^{\ell(k)-j})\!-\!f(x^{\ell(k)-j-1})
   \!+\!\langle\nabla\!f(y^{\ell(k)-j-1}),x^{\ell(k)-j-1}\!-\!x^{\ell(k)-j}\rangle\\
  &\quad -\frac{1}{2\tau_{\ell(k)-j-1}}\|x^{\ell(k)-j}\!-\!y^{\ell(k)-j-1}\|^2
    +\frac{1}{2\tau_{\ell(k)-j-1}}\|x^{\ell(k)-j-1}\!-\!y^{\ell(k)-j-1}\|^2.
 \end{align*}
 Together with the expression of $H_{\delta}$, for each $k\in\mathbb{N}$ and
 each $j\in\{0,1,\ldots,\ell(k)-1\}$,
 \begin{align}\label{Hdeta-tempineq1}
  &H_{\delta}(z^{\ell(k)-j})-H_{\delta}(z^{\ell(k)-j-1})\nonumber\\
  &\le f(x^{\ell(k)-j})-f(x^{\ell(k)-j-1})+\langle\nabla\!f(y^{\ell(k)-j-1}),
       x^{\ell(k)-j-1}\!-\!x^{\ell(k)-j}\rangle\nonumber\\
  &\quad -\frac{1}{2\tau_{\ell(k)-j-1}}\big\|x^{\ell(k)-j}\!-\!y^{\ell(k)-j-1}\big\|^2
        +\frac{1}{2\tau_{\ell(k)-j-1}}\big\|x^{\ell(k)-j-1}\!-\!y^{\ell(k)-j-1}\big\|^2\nonumber\\
  &\quad +\frac{\delta}{2}\|x^{\ell(k)-j}\!-\!x^{\ell(k)-j-1}\|^2
        -\frac{\delta}{2}\|x^{\ell(k)-j-1}\!-\!x^{\ell(k)-j-2}\|^2.
 \end{align}
 We next show that for each $i\in\{0,\ldots,\ell(k)-k+m\}$ the following relations hold:
 \begin{equation}\label{aim-ineq}
   \liminf_{k\to\infty}H_{\delta}(z^{\ell(k)-i})\ge\lim_{k\to\infty}H_{\delta}(z^{\ell(k)})
   \ \ {\rm and}\ \
   \lim_{k\rightarrow\infty}\|z^{\ell(k)-i}-z^{\ell(k)-i-1}\|=0.
 \end{equation}
 When $i=0$, the inequality in \eqref{aim-ineq} clearly holds.
 From the lines 2-6 of Algorithm \ref{PGMenls}, we have
 $H_{\delta}(z^{\ell(k)})-H_{\delta}(z^{\ell(\ell(k)-1)})
  \le-\frac{\alpha}{2}\|z^{\ell(k)}\!-\!z^{\ell(k)-1}\|^2$.
 This implies that $\lim_{k\to\infty}z^{\ell(k)}\!-\!z^{\ell(k)-1}=0$
 because $\{H_{\delta}(z^{\ell(k)})\}_{k\in\mathbb{N}}$ is convergent
 by Lemma \ref{lemma1-Phi} (i), and the equality in \eqref{aim-ineq} holds for $i=0$.
 Assume that the relations in \eqref{aim-ineq} hold for some $i\in\{0,1,\ldots,\ell(k)-k+m\!-\!1\}$.
 From the lines 2-6 of Algorithm \ref{PGMenls}, it follows that
 \(
  H_{\delta}(z^{\ell(k)-i})-H_{\delta}\big(z^{\ell(\ell(k)-i-1)}\big)
  \le-\frac{\alpha}{2}\big\|z^{\ell(k)-i}\!-\!z^{\ell(k)-i-1}\big\|^2.
 \)
 This implies that
 $\limsup_{k\to\infty}H_{\delta}(z^{\ell(k)-i})\le\lim_{k\to\infty}H_{\delta}(z^{\ell(k)})$
 because $\lim_{k\to\infty}z^{\ell(k)-i}\!-\!z^{\ell(k)-i-1}=0$ and
 $\{H_{\delta}(z^{\ell(k)})\}_{k\in\mathbb{N}}$ is convergent.
 Along with $\liminf_{k\to\infty}H_{\delta}(z^{\ell(k)-i})\!\ge\!\lim_{k\to\infty}H_{\delta}(z^{\ell(k)})$,
 we have $\lim_{k\to\infty}H_{\delta}(z^{\ell(k)-i})\!=\!\lim_{k\to\infty}H_{\delta}(z^{\ell(k)})$.
 By combining this with \eqref{Hdeta-tempineq1} for $j=i$ and
 $\lim_{k\to\infty}z^{\ell(k)-i}\!-\!z^{\ell(k)-i-1}=0$, we obtain
 the first inequality in \eqref{aim-ineq} for $i+1$. Using the first inequality
 in \eqref{aim-ineq} for $i+1$ and noting that
 \(
  H_{\delta}(z^{\ell(k)-i-1})-H_{\delta}\big(z^{\ell(\ell(k)-i-2)}\big)
  \le-\frac{\alpha}{2}\big\|z^{\ell(k)-i-1}\!-\!z^{\ell(k)-i-2}\big\|^2,
 \)
 we deduce that the equality in \eqref{aim-ineq} holds for $i+1$.
 Thus, the relations in \eqref{aim-ineq} hold for each $i\in\{0,\ldots,\ell(k)-k+m\}$.
 By combining \eqref{aim-ineq} and \eqref{Hdeta-tempineq1}, we have
 $\limsup_{k\to\infty}\sum_{i=0}^{\ell(k)-k+m}\big[H_{\delta}(z^{\ell(k)-i})-H_{\delta}(z^{\ell(k)-i-1})\big]\le 0$.
 Along with $H_{\delta}(z^{\ell(k)})-H_{\delta}(z^{k-m-1})
  ={\textstyle\sum_{i=0}^{\ell(k)-k+m}}\big[H_{\delta}(z^{\ell(k)-i})-H_{\delta}(z^{\ell(k)-i-1})\big]$
 and the convergence of $\{H_{\delta}(z^{\ell(k)})\}_{k\in\mathbb{N}}$,
 it then follows that $\liminf_{k\to\infty}H_{\delta}(z^{k-m-1})\ge\lim_{k\to\infty}H_{\delta}(z^{\ell(k)})$.

 \noindent
 {\bf(ii)} Fix any $k\in\mathbb{N}$. For each $q\in\mathbb{N}$, from the definition of $x^{k_q}$,
 it follows that
 \begin{align*}
  &\langle\nabla\!f(y^{k_q-1}),x^{k_q}\!-\!\widehat{x}\rangle
  +\frac{1}{2\tau_{k_q-1}}\|x^{k_q}\!-\!y^{k_q-1}\|^2+g(x^{k_q})
  \le \frac{1}{2\tau_{k_q-1}}\|\widehat{x}\!-\!y^{k_q-1}\|^2+g(\widehat{x}).
 \end{align*}
 After a suitable rearrangement, we obtain the following inequality
 \begin{align*}
  F(x^{k_q})&\le F(\widehat{x})\!+\!f(x^{k_q})\!-\!f(\widehat{x})\!-\!\frac{1}{2\tau_{k_q-1}}\|x^{k_q}\!-\!y^{k_q-1}\|^2
  + \langle\nabla\!f(y^{k_q-1}),\widehat{x}\!-\!x^{k_q}\rangle
       +\frac{1}{2\tau_{k_q-1}}\|\widehat{x}\!-\!y^{k_q-1}\|^2.
 \end{align*}
 From part (i) and Lemma \ref{lemma1-Phi} (ii) with $\Phi=H_{\delta}$,
 we have $\lim_{k\to\infty}z^{k+1}-z^k=0$, which implies that $\lim_{q\to\infty}x^{k_q-1}=\widehat{x}$
 and $\lim_{q\to\infty}(x^{k_q-1}-x^{k_q-2})=0$. Then, from the last inequality,
 we have $\limsup_{q\to\infty}F(x^{k_q})\le F(\widehat{x})$ and
 $\limsup_{q\to\infty}H_{\delta}(z^{k_q})\le H_{\delta}(\widehat{z})$.

  \noindent
  {\bf(iii)} For each $k\in\mathbb{N}$, by the definition of $x^{k}$,
  $0\in\nabla\!f(y^{k-1})+\tau_{k-1}^{-1}(x^{k}-y^{k-1})+\partial g(x^{k})$. Then
  \[
    w^k:=\!\left(\begin{matrix}
     \nabla\!f(x^{k})\!-\!\nabla\!f(y^{k-1})-\frac{1}{\tau_{k-1}}(x^{k}\!-\!y^{k-1})+\delta(x^{k}\!-\!x^{k-1})\\
     \delta(x^{k-1}-x^{k})
     \end{matrix}\right)\in\partial H_{\delta}(z^{k}).
  \]
  By the definition of $w^{k}$, the expression of $y^k$ in Step 3
  and $\tau_k\ge\tau_{\rm min}$, it is not hard to check that
 \begin{align*}
  \|w^{k}\|
  &\le\big(L_{\!f}\!+\!\tau_{\rm min}^{-1}+2\delta\big)\|x^{k}\!-\!x^{k-1}\|
    +(L_{\!f}\!+\!\tau_{\rm min}^{-1})\beta_{\rm max}\|x^{k-1}\!-\!x^{k-2}\|.
 \end{align*}
 This implies that the desired inequality holds. The proof is then completed.
 \end{proof}

 By \cite[Theorem 3.6]{LiPong18}, if $F$ is a KL function of exponent $\theta\in[1/2,1)$,
 then so is $H_{\delta}$. Combining Lemma \ref{lemma-PGMenls} with Theorem \ref{KL-converge}
 and \ref{KL-rate} for $\Phi=H_{\delta}$, we obtain the following convergence results.
 \begin{theorem}\label{theorem-PGels}
  Suppose that $F$ is a KL function. Let $\{x^k\}_{k\in\mathbb{N}}$ be the sequence
  generated by Algorithm \ref{PGMenls} with $\delta\in(0,1/2)$. Then, the following
  statements hold.
  \begin{itemize}
    \item [(i)] If $\sum_{K_1\ni k=0}^{\infty}\sqrt{H_{\delta}(z^{\ell(k+1)})\!-\!H_{\delta}(z^{k+1})}<\infty$
                when $\liminf_{K_1\ni k\to\infty}\frac{H_{\delta}(z^{\ell(k)})-H_{\delta}(z^{\ell(k+1)})}{\|z^{k+1}-z^k\|^2}=0$,
                where $K_{1}\!:=\!\big\{k\in\mathbb{N}\,|\,
                H_{\delta}(z^{\ell(k+1)})\!-\!H_{\delta}(z^{k+1})\ge\frac{\alpha}{4}\|z^{k+1}\!-\!z^{k}\|^2\big\}$,
                then $\sum_{k=0}^{\infty}\|z^k\!-\!z^{k-1}\|<\infty$.

  \item[(ii)] Suppose that $F$ is a KL function of exponent $\theta\in[1/2,1)$, and that
               there exist $\widetilde{k}_0\in\mathbb{N}$ and constants $\widetilde{\gamma}>0$
               and $\widetilde{\tau}\in(0,1)$ such that for all $k\ge\widetilde{k}_0$,
               \begin{align*}
                \sum_{K_2\cup K_{31}\ni j=k}^{\infty}\!\sqrt{H_{\delta}(z^{\ell(j+1)})\!-\!H_{\delta}(z^{j+1})}
                \le\left\{\begin{array}{cl}
                \widetilde{\gamma}\widetilde{\tau}^k &{\rm if}\ \theta=1/2,\\
                \!\widetilde{\gamma}k^{\frac{1-\theta}{1-2\theta}}&{\rm if}\ \theta\!\in(1/2,1),
                \end{array}\right.
              \end{align*}
              where $K_{2}\!:=\!\big\{k\in\mathbb{N}\,|\,\frac{\alpha}{4}\|z^{k+1}\!-\!z^{k}\|^2\!
               \le H_{\delta}(z^{\ell(k+1)})\!-\!H_{\delta}(z^{k+1})<\frac{\alpha}{4}\|z^{k+1}\!-\!z^{k}\|^{\frac{1}{\theta}}\big\}$
              and $K_{31}\!:=\!\big\{k\in K_1\backslash K_{2}\,|\,\omega^*\!-\!H_{\delta}(z^{k+1})
              >\frac{\alpha}{8}\|z^{k+1}\!-\!z^{k}\|^{\frac{1}{\theta}}\big\}$ with $\omega^*=\lim_{k\to\infty}H_{\delta}(z^k)$.
              Then $\{z^k\}_{k\in\mathbb{N}}$ converges to some $\widetilde{z}\in{\rm crit}H_{\delta}$
              and there exist $\overline{k}\in\mathbb{N},\gamma>0$ and $\varrho\in(0,1)$ such that
              \begin{equation*}
                 \|z^k-\widetilde{z}\|\le\sum_{j=k}^{\infty}\|z^{j+1}\!-\!z^j\|
                 \le\left\{\begin{array}{cl}
                 \gamma\varrho^{k} &{\rm if}\ \theta=1/2,\\
                 \gamma k^{\frac{1-\theta}{1-2\theta}}&{\rm if}\ \theta\in(1/2,1)
                 \end{array}\right.\ \ {\rm for\ all}\ k\ge\overline{k}.
              \end{equation*}
  \end{itemize}
 \end{theorem}

 From the expression of $H_{\delta}$ and $\lim_{k\to\infty}\|x^k-x^{k-1}\|=0$,
 we get $\lim_{k\to\infty}H_{\delta}(z^k)=\lim_{k\to\infty}F(x^k)$.
 Then, by Proposition \ref{KL-Xirate} with $\Phi=H_{\delta}$,
 the following convergence rate result holds for $\{F(x^k)\}_{k\in\mathbb{N}}$.
 \begin{corollary}\label{obj-rate1}
  If $F$ is a KL function of exponent $\theta\in[1/2,1)$, then there exist
  $\widehat{\varrho}\in(0,1)$ and $\gamma'>0$ such that for all sufficiently large $k$,
  the following inequality holds with $\omega^*=\lim_{k\to\infty}F(x^k)$:
  \[
    F(x^k)-\omega^*\le\left\{\begin{array}{cl}
    \gamma'\widehat{\varrho}^{\lceil\frac{k-1}{m+1}\rceil} &{\rm if}\ \theta=1/2,\\
    \gamma'{k}^{\frac{1-\theta}{1-2\theta}}&{\rm if}\ \theta\in(1/2,1).
    \end{array}\right.
  \]
 \end{corollary}
 \begin{remark}
  Yang \cite{Yang21} achieved the convergence rate of $\{F(x^k)\}_{k\in\mathbb{N}}$
  yielded by PGenls only for the monotone line search case (i.e., $m=0$)
  and the case without the extrapolation (i.e., $\beta_{\rm max}=0$),
  under the assumption that $F$ is a KL function of exponent $\theta\!\in[0,1)$.
  Here, Corollary \ref{obj-rate1} establishes the convergence rate of $\{F(x^k)\}_{k\in\mathbb{N}}$
  for $m>0,\beta_{\rm max}=0$, or $m=0,\beta_{\rm max}>0$, or $m\ge0,\beta_{\rm max}\ge 0$.
 \end{remark}

 \subsection{Numerical experiments for Algorithm \ref{PGMenls}}\label{sec4.2}

 We test the performance of Algorithm \ref{PGMenls} for solving the zero-norm
 regularized logistic regression problem. Given $a_i\in\mathbb{R}^p$ and
 $b_i\in\{-1,1\}$ for $i=1,2,\ldots,n$ with $n<p$, the zero-norm
 regularized logistic regression problem has the form
 \begin{equation}\label{L0-LRP}
  \min_{x=(\widetilde{x},x_0)\in\mathbb{R}^{p+1}}
  \Big\{F_{\lambda,\mu}(x):=\sum_{i=1}^n\log\big(1+\exp(-b_i(a_i^{\mathbb{T}}\widetilde{x}+x_0))\big)+\frac{\mu}{2}\|x\|^2
  +\lambda\|\widetilde{x}\|_0\Big\},
 \end{equation}
 where $\lambda>0$ is a regularization parameter and $\mu>0$ is a tiny constant.
 Let $\widetilde{A}\in\mathbb{R}^{n\times(p+1)}$ be a matrix with $i$th row given by
 $(a_i^{\mathbb{T}},1)$ and $h(z):=\sum_{i=1}^n\log(1+\exp(-b_iz_i))$ for $z\in\mathbb{R}^n$.
 Clearly, the problem \eqref{L0-LRP} takes the form of \eqref{Fprob} with
 $f(x)=h(\widetilde{A}x)+\frac{\mu}{2}\|x\|^2$ and $g(x)=\lambda\|\widetilde{x}\|_0$.
 One can check that $\nabla\!f$ is Lipschitz continuous with the constant
 $L_{\!f}=\frac{1}{4}\|\widetilde{A}\|$. Note that $F_{\lambda,\mu}$ is coercive and
 lower bounded, and it is also a KL function of exponent $1/2$ by \cite[Theorem 4.1]{WuPanBi21}.
 Since $g$ is lsc but discontinuous on its domain, the convergence results in \cite{Yang21}
 are inapplicable to the problem \eqref{L0-LRP}.

 All trials in the subsequent experiments are generated randomly by following
 the same way as in \cite[Section 4.1]{Wen17}. Fix a triple $(n,p,s)$.
 We first generate a matrix $A=[a_1;\cdots;a_n]\in\mathbb{R}^{n\times p}$
 with i.i.d. standard Gaussian entries. Then, we choose a subset $S\subset\{1,2,\ldots,p\}$ of
 size $s$ uniformly at random and generate an $s$-sparse vector $\widehat{x}\in\mathbb{R}^p$
 which has i.i.d. standard Gaussian entries on $S$ and zeros on $\{1,2,\ldots,p\}\backslash S$.
 Finally, we generate the vector $b\in\mathbb{R}^n$ by setting $b={\rm sign}(A\widehat{x}+\varepsilon e)$
 where $\varepsilon$ is chosen uniformly at random from $[0,1]$ and $e\in\mathbb{R}^p$ denotes
 the vector of all ones. For the subsequent experiments, we choose $\mu=10^{-10}$
 and the parameters of Algorithm \ref{PGMenls} as follows
 \begin{equation}\label{parameter1}
   \beta_{\rm max}=1,\,\alpha=10^{-5},\,\eta_2=0.1,\,
   \tau_{\rm min}={10^{-3}}/{[2(\alpha\!+\!\delta)\!+\!L_{\!f}]},\,\tau_{\rm max}=10^{6},\
   \tau_{0,0}={10}/{\|\widetilde{A}\|}.
 \end{equation}

 We evaluate the performances of different algorithms by using an evolution of
 objective values as in \cite{Gillis12,Yang18}. To introduce this evolution,
 let $F_{\lambda,\mu}(x^k)$ denote the objective value at $x^k$ yielded by an algorithm,
 and let $F_{\lambda,\mu}^{\rm min}$ denote the minimum of the terminating objective values
 obtained by all algorithms in a trial. By letting $T(k)$ denote
 the total computation time of an algorithm to yield $x^k$,
 we define the evolution of objective values obtained by this algorithm with respect to time $t$ as
 \[
   E(t):=\min\bigg\{\frac{F_{\lambda,\mu}(x^k)-F_{\lambda,\mu}^{\rm min}}{F_{\lambda,\mu}(x^0)-F_{\lambda,\mu}^{\rm min}}
   \,|\,k\in\big\{i\!: T(i)\le t\big\}\bigg\}.
 \]
 Note that $E(t)\in[0,1]$ and it is nonincreasing with respect to $t$.
 It can be viewed as a normalized measure of the reduction of the objective value
 with respect to time. One can take the average of $E(t)$ over several independent trials,
 and plot the average $E(t)$ within time $t$ for an algorithm.
  \begin{figure}[h]
  \vspace{-0.3cm}
  \begin{subfigure}[b]{0.5\textwidth}
   \centering
  \includegraphics[width=6cm,height=5cm]{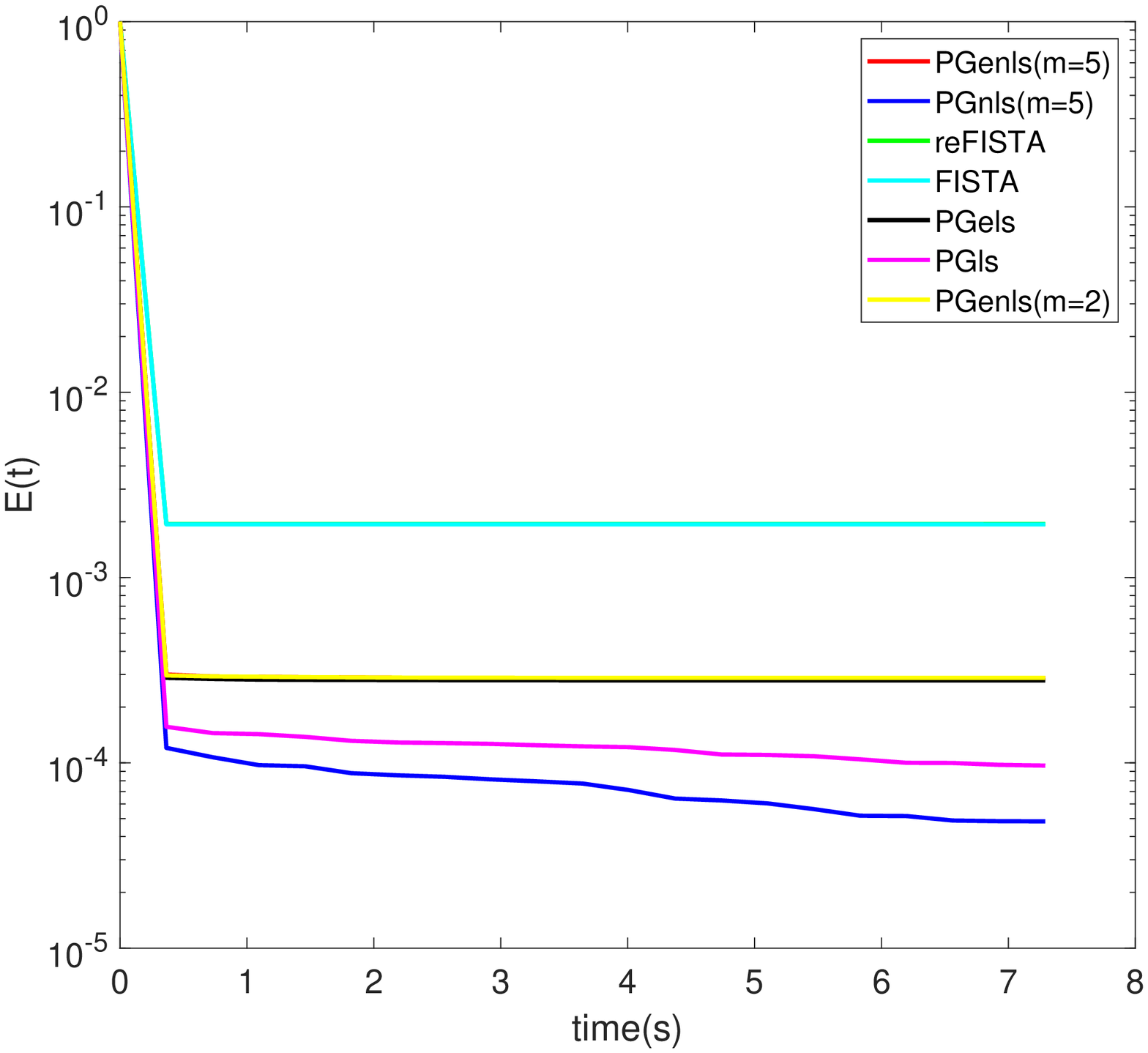}
   \caption{$\lambda=0.001$}
   \label{fig2:1_1}
  \end{subfigure}
  \begin{subfigure}[b]{0.5\textwidth}
   \centering
  \includegraphics[width=6cm,height=5cm]{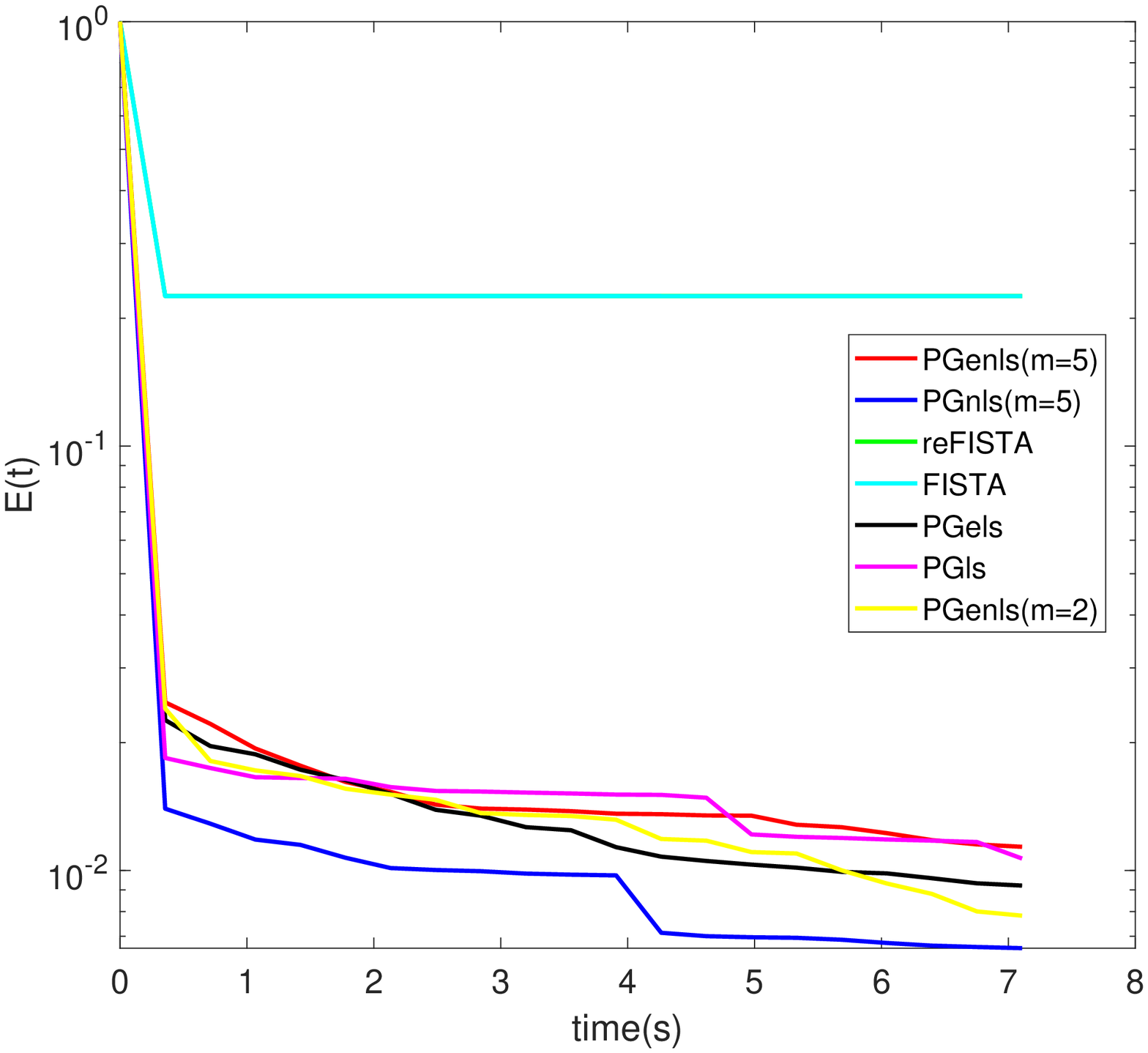}
   \caption{$\lambda=0.1$}
   \label{fig2:1_2}
  \end{subfigure}
  \newline
  \begin{subfigure}[b]{0.5\textwidth}
   \centering
  \includegraphics[width=6cm,height=5cm]{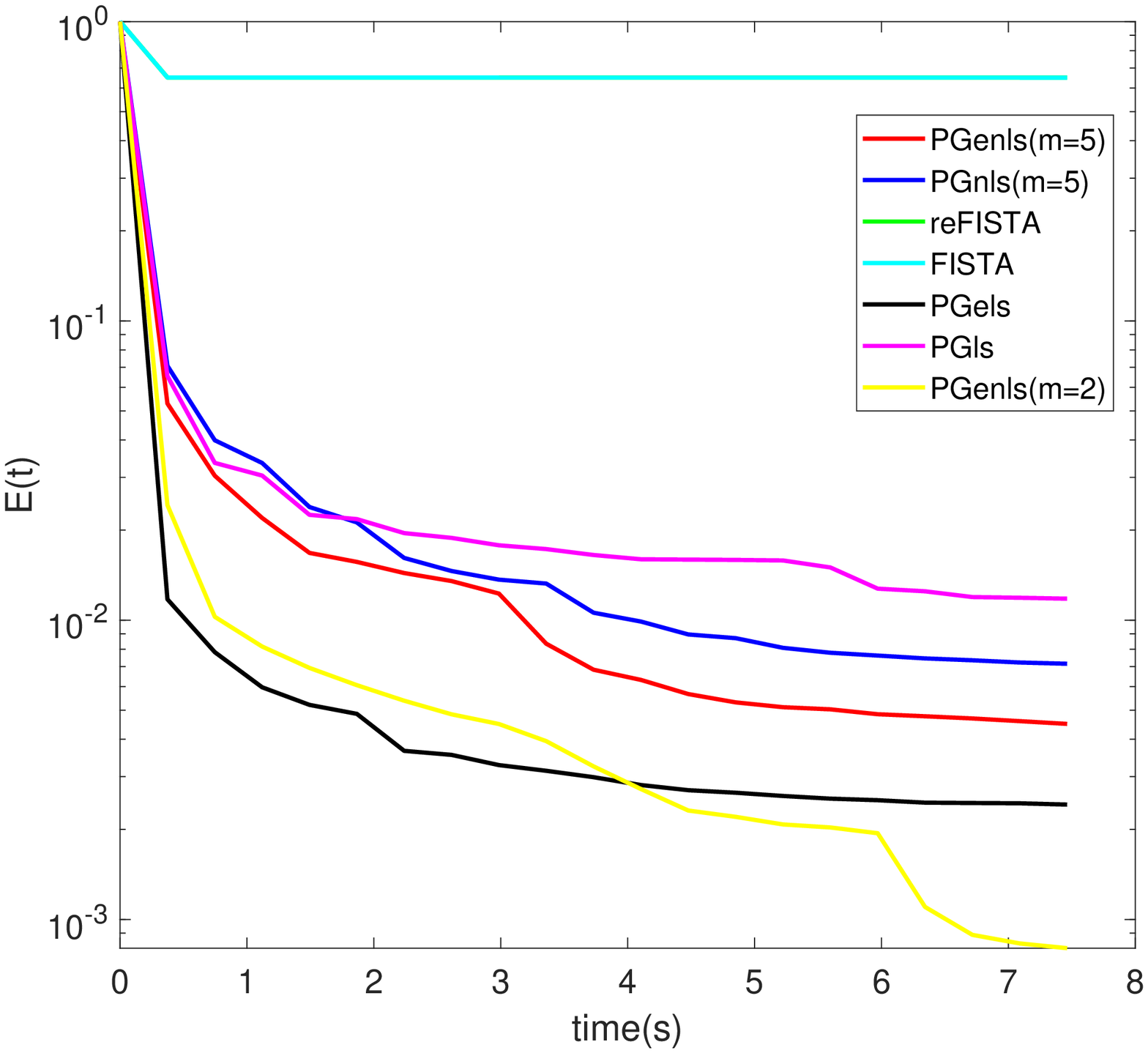}
   \caption{$\lambda=1$}
   \label{fig2:2_1}
  \end{subfigure}
  \begin{subfigure}[b]{0.5\textwidth}
   \centering
  \includegraphics[width=6cm,height=5cm]{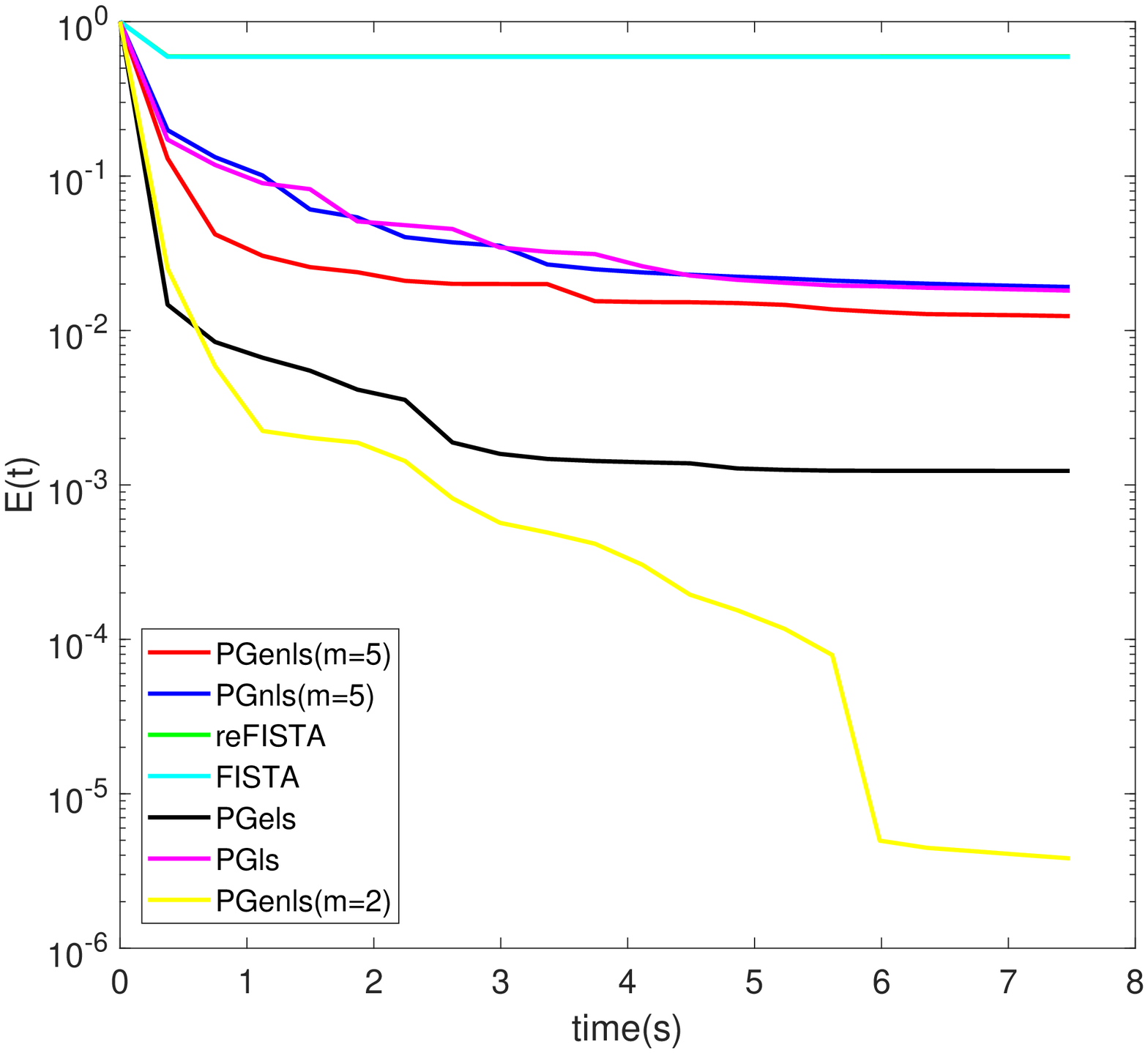}
   \caption{$\lambda=10$}
   \label{fig2:2_2}
  \end{subfigure}
  \caption{\small Several methods for solving the nonconvex nonsmooth problem \eqref{L0-LRP} with different $\lambda$}
  \label{fig1}
 \end{figure}

 Preliminary tests show that $\delta$ and $\eta_1$ have a great influence
 on the performance of Algorithm \ref{PGMenls}, so we first evaluate Algorithm \ref{PGMenls}
 with different $(\delta,\eta_1)$ for $\delta\in\{0.1,0.05,0.01,0.005,0.001\}$
 and $\eta_1\in\{0.1,0.05,0.01,0.005,0.001\}$ by solving \eqref{L0-LRP}
 with $(n,p,s)=(300,3000,30)$ and $\lambda=0.1$. Numerical results for $10$ independent
 trials indicate that Algorithm \ref{PGMenls} with $\delta\in\{0.01,0.005,0.001\}$ and
 $\eta_1\in\{0.1,0.05,0.01\}$ have better performance. Now we apply Algorithm \ref{PGMenls}
 to the problem \eqref{L0-LRP} with $(n,p,s)=(500,5000,50)$
 and different $\lambda$, and compare the performance of Algorithm \ref{PGMenls}
 for $\delta=0.01,\eta_1=0.05$ (PGenls) with the performances of
 Algorithm \ref{PGMenls} for $\delta=0.01,\beta_{\rm max}=0$ (PGnls),
 Algorithm \ref{PGMenls} for $\delta=0.01,m=0$ (PGels), Algorithm \ref{PGMenls}
 for $\delta=0,\beta_{\rm max}=0$, $m=0$ (PGls), FISTA \cite{Wen17}
 and reFISTA \cite{Donoghue15}. Among others, we restart the iterates in reFISTA
 when $k\,{\rm mod}\,250=0$ or $\langle y^k-x^{k+1},x^{k+1}-x^k\rangle>0$.
 From Figure \ref{fig1}, we see that for $\lambda=10^{-3}$ and $0.1$, PGenls is remarkably
 superior to FISTA and reFISTA, and PGnls is superior to PGenls, which is comparable with
 PGels and PGls; for $\lambda=1$ and $10$, PGenls with $m=2$ has much better performance
 than PGenls with $m=5$ and PGels do, and PGnls is now close to PGls. This shows that
 the nonmonotone line search is more efficient for \eqref{L0-LRP} with a smaller
 $\lambda$ while the extrapolation is more efficient for \eqref{L0-LRP} with a larger $\lambda$.
 Note that the problem \eqref{L0-LRP} with a smaller $\lambda$ is more difficult than
 the one with a larger $\lambda$ because the latter will be strongly convex
 at a critical point due to high sparsity.

 Recall that the global convergence of the iterate sequence generated by Algorithm \ref{PGMenls}
 requires an assumption (see Theorem \ref{theorem-PGels} (i)). Figure \ref{fig2} indicates that
 this assumption can be satisfied in practical computation, where the curves are plotted
 by solving \eqref{L0-LRP} with $(n,p,s)=(500,5000,50)$ and $(\delta,\eta_1)=(0.01,0.05)$,
 the red line records the sum $\sum_{K_1\ni k=1}\sqrt{H_{\delta}(z^{\ell(k)})\!-\!H_{\delta}(z^k)}$,
 the blue line records the sum $\sum_{k=1}\frac{3000}{\sqrt{k^{2.1}}}$,
 and $K_1$ is defined as in Theorem \ref{theorem-PGels} (i) for $\alpha=10^{-5}$.
 \begin{figure}[h]
  \vspace{-0.3cm}
  \begin{subfigure}[b]{0.5\textwidth}
   \centering
  \includegraphics[width=6cm,height=5cm]{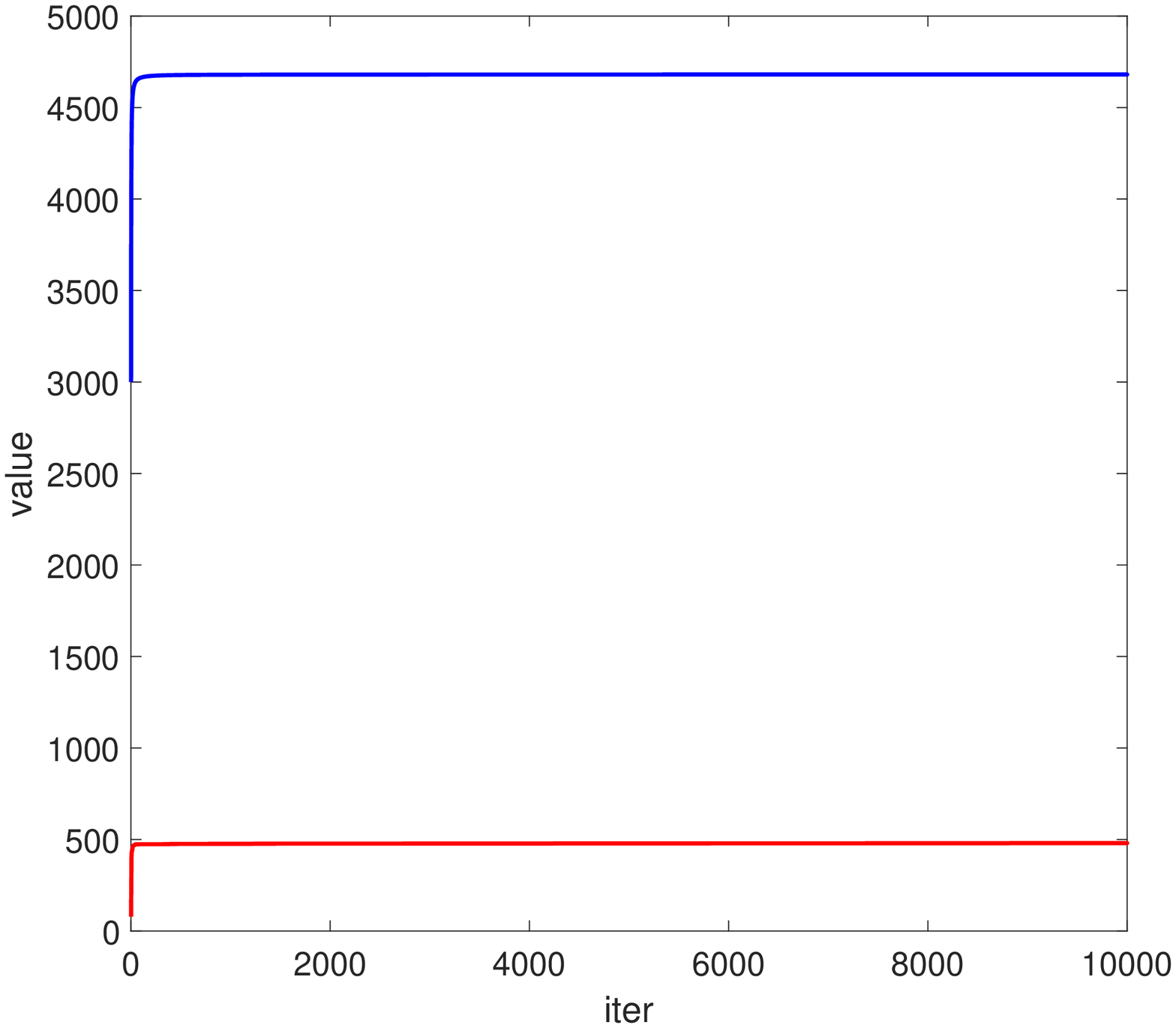}
   \caption{$\lambda=0.01$}
   \label{fig1:1_1}
  \end{subfigure}
  \begin{subfigure}[b]{0.5\textwidth}
   \centering
  \includegraphics[width=6cm,height=5cm]{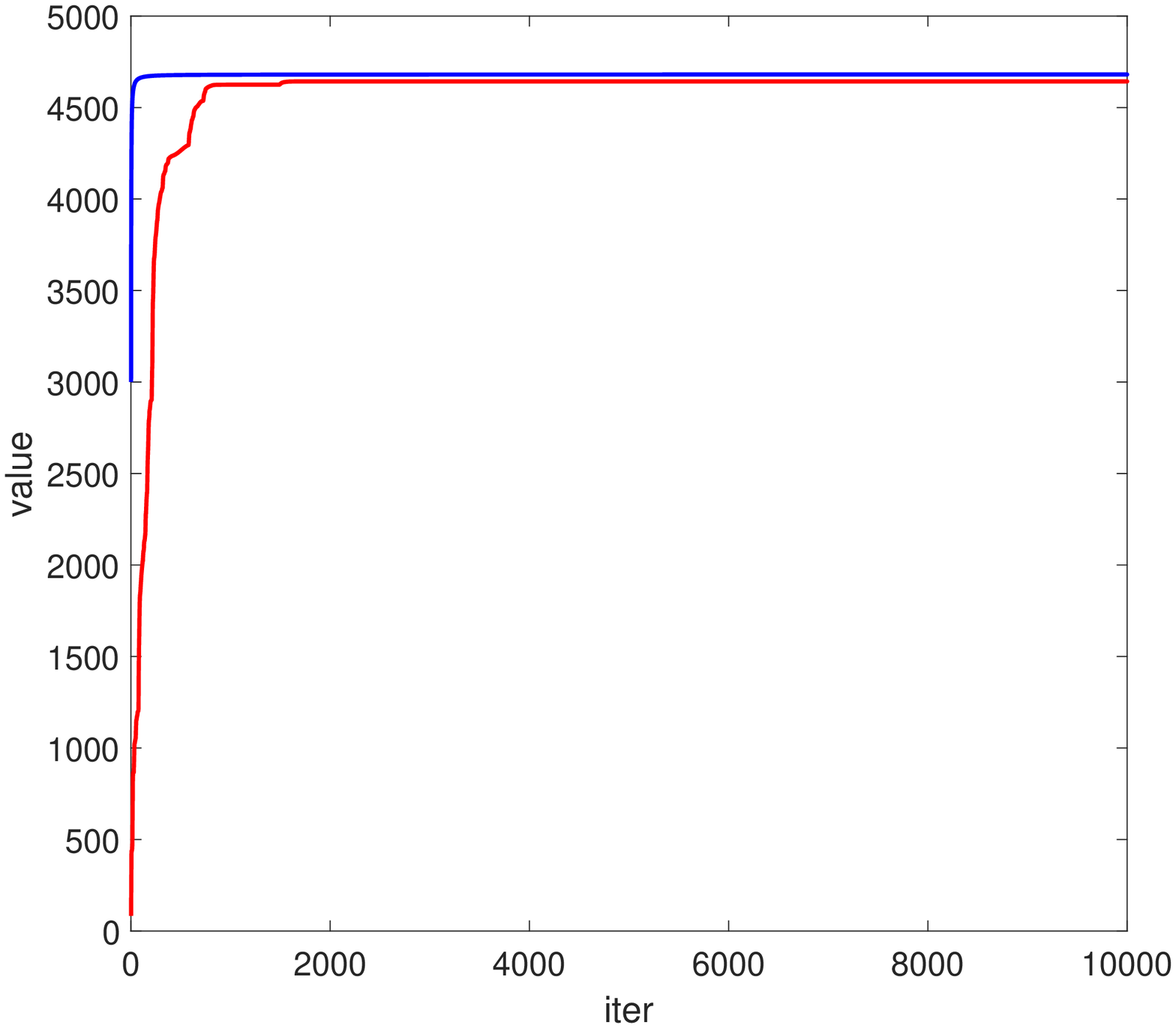}
   \caption{$\lambda=10$}
   \label{fig1:1_2}
  \end{subfigure}
  \caption{The assumption in Theorem \ref{theorem-PGels} (i) is illustrated in practical computation}
  \label{fig2}
 \end{figure}
 \section{Nonmotone line search PALM with extrapolation}\label{sec5}

 Let $f\!:\mathbb{X}\to\overline{\mathbb{R}}$ and $g\!:\mathbb{Y}\to\overline{\mathbb{R}}$
 be proper lsc functions, and let $H\!:\mathbb{X}\times\mathbb{Y}\to\mathbb{R}$ be a smooth
 function with the partial gradients $\nabla_{\!x}H(\cdot,y)$ and $\nabla_{\!y}H(x,\cdot)$
 being $L_1(y)$-Lipschitz and $L_2(x)$-Lipschitz, respectively. Consider the problem
 \begin{equation}\label{PALM-prob}
   \min_{x\in\mathbb{X},y\in\mathbb{Y}}\Psi(x,y):=f(x)+g(y)+H(x,y),
 \end{equation}
 where $f$ and $g$ are assumed to bounded below and $\Psi$ is assumed to be coercive
 and bounded below. Clearly, for any $(x^0,y^0)\in\mathbb{X}\times\mathbb{Y}$,
 the level set $\mathcal{L}_0:=\{(x,y)\in\mathbb{X}\times\mathbb{Y}\,|\, \Psi(x,y)\le\Psi(x^0,y^0)\}$
 is compact. In this section we develop a nonmonotone line search PALM with extrapolation
 (PALMenls), a nonmonotone line search accelerated version of the PALM in \cite{Bolte14},
  for solving the problem \eqref{PALM-prob}.

 For any given $\delta>0$ and any $z=(x,y,u,v)\in\mathbb{X}\times\mathbb{Y}\times\mathbb{X}\times\mathbb{Y}$,
 define the potential function
 \begin{equation}\label{Upsilonz}
  \Upsilon_{\!\delta}(z):=f(x)+g(y)+H(x,y)+({\delta}/{2})\|x-u\|^2+({\delta}/{2})\|y-v\|^2,
 \end{equation}
 and write $z^k\!:=\!(x^{k},y^{k},x^{k-1},y^{k-1})$ for each $k\in\mathbb{N}$.
 The iterates of PALMenls are described as below, where the constant
 $L\!:=\!\max_{(x,y)\in\mathcal{L}_0}\max(L_1(y),L_2(x))$
 depends on the initial $(x^0,y^0)$.
 \begin{algorithm}[h]
 \caption{\label{PALMenls}{\bf\,(Nonmonotone line search PALM with extrapolation)}}
 \textbf{Initialization:} Choose $m\in\mathbb{N},(x^0,y^0)\in{\rm dom}\Psi,
 \delta\in(0,1),\alpha\in(0,{\delta}/{2}], 0<\underline{\tau}<\!\frac{1}{L+\delta+2\alpha}\!<\overline{\tau}$,
 $\beta_{\rm max}\ge 0,\eta\in(0,1),\eta_1,\eta_2\in(0,1)$.
 Let $(x^{-1},y^{-1})\!=(x^0,y^0),z^{0}\!:=(x^0,y^0,x^{-1},y^{-1})$. Set $k:=0$.\\
 \textbf{while} the stopping condition is not satisfied \textbf{do}
 \begin{enumerate}
  \item  Choose $\beta_{k,0}\in[0,\beta_{\rm max}]$, $\tau_{1,k}^{0}\in[\underline{\tau},\overline{\tau}]$ and
         $\tau_{2,k}^{0}\in[\underline{\tau},\overline{\tau}]$.

  \item  \textbf{For} $k=0,1,2,\ldots$ \textbf{do}

  \item  \quad Let $\beta_k=\beta_{k,0}\eta^{l},\tau_{1,k}=\max\{\tau_{1,k}^0\eta_1^{l},\underline{\tau}\}$ and
               $\tau_{2,k}=\max\{\tau_{2,k}^0\eta_2^{l},\underline{\tau}\}$.

  \item  \quad Let $\widetilde{x}^k=x^k\!+\!\beta_{k}(x^k\!-\!x^{k-1})$ and compute
         $x^{k+1}\in\mathcal{P}_{\!\tau_{1,k}}f(\widetilde{x}^k\!-\!\tau_{1,k}\nabla_{\!x}H(\widetilde{x}^{k},y^k))$.

  \item  \quad Let $\widetilde{y}^k=y^k\!+\!\beta_{k}(y^k\!-\!y^{k-1})$ and compute
         $y^{k+1}\in\mathcal{P}_{\!\tau_{2,k}}g(\widetilde{y}^k\!-\!\tau_{2,k}\nabla_{\!y} H(x^{k+1},\widetilde{y}^k))$.

  \item  \quad If $\Upsilon_{\!\delta}(z^{k+1})\le\max_{j=[k-m]_{+},\ldots,k}\Upsilon_{\!\delta}(z^{j})
         -\frac{\alpha}{2}\|z^{k+1}\!-\!z^k\|^2$, go to Step 8.

  \item  \textbf{end for}

  \item  Set $k\leftarrow k+1$ and go to Step 1.
 \end{enumerate}
 \textbf{end (while)}
 \end{algorithm}
 \begin{remark}\label{remark-PALM}
  From Lemma \ref{PALM-descent} in Appendix, the line search steps
  on lines 2-7 of Algorithm \ref{PALMenls} are well defined.
  Similar to Algorithm \ref{PGMenls},  one can initialize the extrapolation parameter
  $\beta_{k,0}$ in Step 1 by the rule \eqref{Nesterov}, and initialize the step-sizes
  $\tau_{1,k}^0$ and $\tau_{2,k}^0$ for $k\ge 1$ by the following BB rule \cite{Barzilai88}:
  \begin{align}\label{tau1k0}
   \tau_{1,k}^0=\max\Big\{\min\Big\{\frac{\|x^k-x^{k-1}\|^2}{\langle x^k-x^{k-1},\Delta H_x^k\rangle},
    \frac{\langle x^k-x^{k-1},\Delta H_x^k\rangle}{\|\Delta H_x^k\|^2},\overline{\tau}\Big\},\underline{\tau}\Big\},\\
   \label{tau2k0}
   \tau_{2,k}^0=\max\Big\{\min\Big\{\frac{\|y^k-y^{k-1}\|^2}{\langle y^k-y^{k-1},\Delta H_y^k\rangle},
    \frac{\langle y^k-y^{k-1},\Delta H_y^k\rangle}{\|\Delta H_y^k\|^2},\overline{\tau}\Big\},\underline{\tau}\Big\},
  \end{align}
  where $\Delta H_{\!x}^k\!:=\!\nabla_{\!x}H(x^k,y^k)-\nabla_{\!x}H(x^{k-1},y^k)$
  and $\Delta H_{\!y}^k\!:=\!\nabla_{\!y}H(x^{k},y^k)-\nabla_{\!y}H(x^{k},y^{k-1})$.
 \end{remark}
 \subsection{Convergence results of Algorithm \ref{PALMenls}}\label{sec5.1}

 By Step 6 of Algorithm \ref{PALMenls}, the sequence $\{z^k\}_{k\in\mathbb{N}}$
 satisfies the condition H1 for $\Phi=\Upsilon_{\delta}$. By the proof of
 Lemma \ref{lemma1-Phi} (i) and $(x^{-1},y^{-1})=(x^0,y^0)$,
 $\{z^k\}_{k\in\mathbb{N}}\subseteq\{z\,|\,\Upsilon_{\!\delta}(z)\le\Upsilon_{\!\delta}(z^0)\}
 \subseteq\mathcal{L}_0\times\mathcal{L}_0$. Thus, $\{z^k\}_{k\in\mathbb{N}}$ is bounded
 by the compactness of $\mathcal{L}_0$. Let $B_{1}$ and $B_{2}$ be the ball
 centered at the origin containing $\{(1\!+\!2\beta_{\rm max})x^k\}_{k\in\mathbb{N}}$
 and $\{(1\!+\!2\beta_{\rm max})y^k\}_{k\in\mathbb{N}}$, respectively.
 Write $M:=\max_{x\ne x'\in B_1,y\ne y'\in B_2 }\frac{\|\nabla_x H(x,y)-\nabla_x H(x',y')\|}{\|(x,y)-(x',y')\|}$
 and $\overline{L}_2:=\max_{x\in B_1}\{L_2(x)\}$.

 The following lemma demonstrates that the function $\Phi=\Upsilon_{\delta}$
 satisfies the conditions \eqref{Phi-cond1}-\eqref{Phi-cond2}, and moreover,
 the sequence $\{z^k\}_{k\in\mathbb{N}}$ also satisfies
 the condition H2 with $\Phi=\Upsilon_{\delta}$.
 \begin{lemma}\label{lemma5.1}
  Let $\{(x^k,y^k)\}_{k\in\mathbb{N}}$ be generated by Algorithm \ref{PALMenls}. Then,
  the following results hold.
  \begin{itemize}
   \item [(i)] $\liminf_{k\to\infty}\Upsilon_{\delta}(z^{k})\ge\lim_{k\to\infty}\Upsilon_{\delta}(z^{\ell(k)})$.

   \item[(ii)] For each $\{z^{k_q}\}_{q\in\mathbb{N}}$ with $\lim_{q\to\infty}z^{k_q}\to \widehat{z}
               =(\widehat{x},\widehat{y},\widehat{x},\widehat{y})$,
                  $\limsup_{q\to\infty}\Upsilon_{\delta}(z^{k_q})\le\Upsilon_{\delta}(\widehat{z})$.

   \item[(iii)] There exists $w^{k}\in\partial\Upsilon_{\delta}(z^{k})$
                with $\|w^{k+1}\|\le\![2\delta\!+2\max(1,\beta_{\rm max})(M\!+\!2\underline{\tau}^{-1}\!+\!\overline{L}_2)]\|z^{k+1}\!-\!z^k\|$.
  \end{itemize}
 \end{lemma}
 \begin{proof}
  {\bf(i)} From lines 2-7 of Algorithm \ref{PALMenls} and Lemma \ref{lemma1-Phi} (i)
  with $\Phi=\Upsilon_{\delta}$, it follows that $\lim_{k\to\infty}\Upsilon_{\!\delta}(z^{\ell(k)})=\omega^*$
  for some $\omega^*\in\mathbb{R}$, and for each $k\in\mathbb{N}$ and $i\in\{0,1,\ldots,\ell(k)\!-\!1\}$,
 \begin{equation}\label{Edescent-ineq1}
  \Upsilon_{\!\delta}(z^{\ell(k)-i})\le \Upsilon_{\delta}(z^{\ell(\ell(k)-i-1)})
  \!-\!\frac{\alpha}{2}\|z^{\ell(k)-i}\!-\!z^{\ell(k)-i-1}\|^2.
 \end{equation}
  In addition, from the definition of $x^{k}$, for each $k\in\mathbb{N}$ and $i\in\{0,1,\ldots,\ell(k)\!-\!1\}$,
 \begin{align}\label{Edescent-ineq2}
   &\langle\nabla_xH(\widetilde{x}^{\ell(k)-i-1},y^{\ell(k)-i-1}),x^{\ell(k)-i}-x^{\ell(k)-i-1}\rangle\!
     +\frac{1}{2\tau_{1,\ell(k)-i-1}}\|x^{\ell(k)-i}\!-\!\widetilde{x}^{\ell(k)-i-1}\|^2\nonumber\\
   &\le f(x^{\ell(k)-i-1})-f(x^{\ell(k)-i})
     +\frac{1}{2\tau_{1,\ell(k)-i-1}}\|x^{\ell(k)-i-1}\!-\!\widetilde{x}^{\ell(k)-i-1}\|^2.
 \end{align}
 While from the definition of $y^{k}$, for each $k\in\mathbb{N}$ and $i\in\{0,1,\ldots,\ell(k)\!-\!1\}$,
 \begin{align}\label{Edescent-ineq3}
   &\langle\nabla_yH({x}^{\ell(k)-i},\widetilde{y}^{\ell(k)-i-1}),y^{\ell(k)-i}-y^{\ell(k)-i-1}\rangle\!
     +\frac{1}{2\tau_{2,\ell(k)-i-1}}\|y^{\ell(k)-i}\!-\!\widetilde{y}^{\ell(k)-i-1}\|^2\nonumber\\
   &\le g(y^{\ell(k)-i-1})-g(y^{\ell(k)-i})
   +\frac{1}{2\tau_{2,\ell(k)-i-1}}\|y^{\ell(k)-i-1}\!-\!\widetilde{y}^{\ell(k)-i-1}\|^2.
 \end{align}
 In order to achieve the desired result, we first argue by induction that for each $j\in\mathbb{N}$
 \begin{equation}\label{Edeta-ineq4}
  \liminf_{k\to\infty}\Upsilon_{\delta}(z^{\ell(k)-j})\ge\lim_{k\to\infty}\Upsilon_{\delta}(z^{\ell(k)})
  \ \ {\rm and}\ \ \lim_{k\rightarrow\infty}\|z^{\ell(k)-j}-z^{\ell(k)-j-1}\|=0.
 \end{equation}
 Passing the limit $k\rightarrow\infty$ to \eqref{Edescent-ineq1} with $i=0$
 and using $\lim_{k\to\infty}\Upsilon_{\delta}(z^{\ell(k)})=\omega^*$, we obtain
 $\lim_{k\rightarrow\infty}\|z^{\ell(k)}-z^{\ell(k)-1}\|=0$. By combining this limit
 with the boundedness of $\{z^k\}_{k\in\mathbb{N}}$ and
 passing the limit $k\rightarrow\infty$ to \eqref{Edescent-ineq2}
 and \eqref{Edescent-ineq3} with $i=0$ yields
 $\liminf_{k\to\infty}f(x^{\ell(k)-1})\ge\liminf_{k\to\infty}f(x^{\ell(k)})$
 and $\liminf_{k\to\infty}g(y^{\ell(k)-1})\ge\liminf_{k\to\infty}g(y^{\ell(k)})$.
 By the continuity of $H$, the inequality in \eqref{Edeta-ineq4} holds
 for $j=1$. Then, passing the limit $k\rightarrow\infty$ to \eqref{Edescent-ineq1}
 with $i=1$ and using the inequality in \eqref{Edeta-ineq4} for $j=1$ yields that
 the equality in \eqref{Edeta-ineq4} holds for $j=1$.
 Now suppose that the relations in \eqref{Edeta-ineq4} hold for some $j\ge1$.
 Since $\liminf_{k\to\infty}\Upsilon_{\delta}(z^{\ell(k)-j})\ge\omega^*$
 and $\lim_{k\rightarrow\infty}\|z^{\ell(k)-j}-z^{\ell(k)-j-1}\|=0$,
 from \eqref{Edescent-ineq2} and \eqref{Edescent-ineq3} for $i=j$
 and the continuity of $H$, we obtain the inequality in \eqref{Edeta-ineq4}
 for $j+1$. Then, passing the limit $k\rightarrow\infty$ to \eqref{Edescent-ineq1}
 with $i=j+1$ and using the inequality in \eqref{Edeta-ineq4} for $j+1$ yields that
 the equality in \eqref{Edeta-ineq4} holds for $j+1$.
 Thus, the relations in \eqref{Edeta-ineq4} hold.
 From \eqref{Edescent-ineq2} and \eqref{Edescent-ineq3},
 \begin{align*}
   f(x^{\ell(k)})
  &\le f(x^{k-m-1})+\sum_{i=0}^{\ell(k)-(k-m)}\Big[
  \langle\nabla_xH(\widetilde{x}^{\ell(k)-i-1},y^{\ell(k)-i-1}),x^{\ell(k)-i-1}-x^{\ell(k)-i}\rangle \\
  &\quad\ -\frac{1}{2\tau_{1,\ell(k)-i-1}}\|x^{\ell(k)-i}\!-\!\widetilde{x}^{\ell(k)-i-1}\|^2
  +\frac{1}{2\tau_{1,\ell(k)-i-1}}\|x^{\ell(k)-i-1}\!-\!\widetilde{x}^{\ell(k)-i-1}\|^2\Big],\\
   g(y^{\ell(k)})
  &\le g(y^{k-m-1})+\sum_{i=0}^{\ell(k)-(k-m)}\Big[
  \langle\nabla_yH(x^{\ell(k)-i},\widetilde{y}^{\ell(k)-i-1}),y^{\ell(k)-i-1}-y^{\ell(k)-i}\rangle \\
  &\quad\ -\frac{1}{2\tau_{2,\ell(k)-i-1}}\|y^{\ell(k)-i}\!-\!\widetilde{y}^{\ell(k)-i-1}\|^2
  +\frac{1}{2\tau_{2,\ell(k)-i-1}}\|y^{\ell(k)-i-1}\!-\!\widetilde{y}^{\ell(k)-i-1}\|^2\Big].
 \end{align*}
 Recall that $\lim_{k\rightarrow\infty}\|z^{\ell(k)-j}-z^{\ell(k)-j-1}\|=0$ for all $j\in\mathbb{N}$.
 Passing the limit $k\rightarrow\infty$ to the last two inequalities yields $\liminf_{k\to\infty}f(x^{k-m-1})\ge\liminf_{k\to\infty}f(x^{\ell(k)})$
 and $\liminf_{k\to\infty}g(y^{k-m-1})\ge\liminf_{k\to\infty}g(y^{\ell(k)})$.
 The result then follows by the continuity of $H$.

 \noindent
 {\bf(ii)} Fix any $k\in\mathbb{N}$. For each $q\in\mathbb{N}$, from the definition of $x^{k_q}$
 and $y^{k_q}$, it follows that
  \begin{align*}
   &\langle\nabla_xH(\widetilde{x}^{k_q-1},y^{k_q-1}),x^{k_q}-\widehat{x}\rangle\!
     +\frac{1}{2\tau_{1,k_q-1}}\|x^{k_q}\!-\!\widetilde{x}^{k_q-1}\|^2+f(x^{k_q})
   \le\frac{1}{2\tau_{1,k_q-1}}\|\widehat{x}\!-\!\widetilde{x}^{k_q-1}\|^2+f(\widehat{x}),\\
   &\langle\nabla_yH({x}^{k_q},\widetilde{y}^{k_q-1}),y^{k_q}-\widehat{y}\rangle\!
     +\frac{1}{2\tau_{2,k_q-1}}\|y^{k_q}\!-\!\widetilde{y}^{k_q-1}\|^2+g(y^{k_q})
   \le\frac{1}{2\tau_{2,k_q-1}}\|\widehat{y}\!-\!\widetilde{y}^{k_q-1}\|^2+g(\widehat{y}).
 \end{align*}
  Note that $\lim_{q\to\infty}z^{k_q}-z^{k_q-1}=0$ by combining part (i)
  with Lemma \ref{lemma1-Phi} (ii). From the last two inequalities and
  the continuous differentiability of $H$, we obtain the desired result.

 \noindent
 {\bf(iii)} For each $k\in\mathbb{N}$, from the optimality conditions of $x^{k+1}$ and $y^{k+1}$, it follows that
 \[
   0\in\!\nabla_x H(\widetilde{x}^k,y^k)\!+\!\tau_{1,k}^{-1}(x^{k+1}\!-\!\widetilde{x}^k)+\partial f(x^{k+1}),
   0\in\!\nabla_x H(x^{k+1},\widetilde{y}^k)\!+\!\tau_{2,k}^{-1}(y^{k+1}\!-\!\widetilde{y}^k)+\partial g(y^{k+1}).
 \]
  By comparing with the expression of $\partial\Upsilon_{\delta}(z^k)$, it is not hard to obtain that
  \[
    w^k:=\left(\begin{matrix}
     \nabla_x H(x^k,y^k)\!-\!\nabla_x H(\widetilde{x}^{k-1},y^{k-1})
     -\frac{1}{\tau_{1,k-1}}(x^{k}\!-\!\widetilde{x}^{k-1})+\delta(x^{k}-x^{k-1})\\
     \nabla_y H(x^k,y^k)\!-\!\nabla_y H(x^{k},\widetilde{y}^{k-1})
     -\frac{1}{\tau_{2,k-1}}(y^{k}\!-\!\widetilde{y}^{k-1})+\delta(y^{k}-y^{k-1})\\
      \delta(x^{k-1}-x^{k})\\
      \delta(y^{k-1}-y^{k})
     \end{matrix}\right)\in\partial\Upsilon_{\delta}(z^{k}).
  \]
  By the expression of $w^{k+1}$ and the discussion in the paragraph of this section,
  it follows that
 \begin{align*}
  \|w^{k+1}\|\le (M+\tau_{1,k}^{-1})\|x^{k+1}-\widetilde{x}^k\|+M\|y^{k+1}-y^k\|
  +(\overline{L}_2+\tau_{2,k}^{-1})\|y^{k+1}-\widetilde{y}^k\|+2\delta\|z^{k+1}-z^k\|,
 \end{align*}
 which by the expressions of $\widetilde{x}^k$ and $\widetilde{y}^k$ implies the result.
  The proof is then completed.
 \end{proof}

 By invoking \cite[Theorem 3.6]{LiPong18}, if $\Psi$ is a KL function of exponent $\theta\in[1/2,1)$,
 then $\Upsilon_{\delta}$ is also a KL function of exponent $\theta\in[1/2,1)$.
 Thus, by combining Lemma \ref{lemma5.1} with Theorem \ref{KL-converge} and \ref{KL-rate}
 for $\Phi=\Upsilon_{\!\delta}$,
 we obtain the following convergence results for the iterate sequence of Algorithm \ref{PALMenls}.
 \begin{theorem}\label{converge-PALMels}
  Suppose that $\Psi$ is a KL function. Let $\{(x^k,y^k)\}_{k\in\mathbb{N}}$ be the sequence
  generated by Algorithm \ref{PALMenls} with $\delta\in(0,1)$. Then, the following
  statements hold.
  \begin{itemize}
    \item [(i)] If $\sum_{K_1\ni k=0}^{\infty}\sqrt{\Upsilon_{\!\delta}(z^{\ell(k+1)})\!-\!\Upsilon_{\delta}(z^{k+1})}<\infty$
                when $\liminf_{K_1\ni k\to\infty}\frac{\Upsilon_{\delta}(z^{\ell(k)})-\Upsilon_{\delta}(z^{\ell(k+1)})}
                {\|z^{k+1}-z^k\|^2}=0$,
                where $K_{1}\!:=\!\big\{k\in\mathbb{N}\,|\,
                \Upsilon_{\delta}(z^{\ell(k+1)})\!-\!\Upsilon_{\delta}(z^{k+1})\ge\frac{\alpha}{4}\|z^{k+1}\!-\!z^{k}\|^2\big\}$,
                then $\sum_{k=0}^{\infty}\|z^k\!-\!z^{k-1}\|<\infty$.

  \item[(ii)] Suppose that $\Psi$ is a KL function of exponent $\theta\in[1/2,1)$, and that
               there exist $\widetilde{k}_0\in\mathbb{N}$ and constants $\widetilde{\gamma}>0$
               and $\widetilde{\tau}\in(0,1)$ such that for all $k\ge\widetilde{k}_0$,
               \begin{align*}
                \sum_{K_2\cup K_{31}\ni j=k}^{\infty}\!\sqrt{\Upsilon_{\delta}(z^{\ell(j+1)})\!-\!\Upsilon_{\delta}(z^{j+1})}
                \le\left\{\begin{array}{cl}
                \widetilde{\gamma}\widetilde{\tau}^k &{\rm if}\ \theta=1/2,\\
                \!\widetilde{\gamma}k^{\frac{1-\theta}{1-2\theta}}&{\rm if}\ \theta\!\in(1/2,1),
                \end{array}\right.
              \end{align*}
              where $K_{2}\!:=\!\big\{k\in\mathbb{N}\,|\,\frac{\alpha}{4}\|z^{k+1}\!-\!z^{k}\|^2\!
               \le \Upsilon_{\delta}(z^{\ell(k+1)})\!-\!\Upsilon_{\delta}(z^{k+1})
               <\frac{\alpha}{4}\|z^{k+1}\!-\!z^{k}\|^{\frac{1}{\theta}}\big\}$
              and $K_{31}\!:=\!\big\{k\in K_1\backslash K_{2}\,|\,\omega^*\!-\!\Upsilon_{\delta}(z^{k+1})
              >\frac{\alpha}{8}\|z^{k+1}\!-\!z^{k}\|^{\frac{1}{\theta}}\big\}$ with $\omega^*=\lim_{k\to\infty}\Upsilon_{\delta}(z^k)$.
              Then $\{z^k\}_{k\in\mathbb{N}}$ converges to some $\widetilde{z}\in{\rm crit}\Upsilon_{\delta}$
              and there exist $\overline{k}\in\mathbb{N},\gamma>0$ and $\varrho\in(0,1)$ such that
              \begin{align*}
                 \|z^k-\widetilde{z}\|\le\sum_{j=k}^{\infty}\|z^{j+1}\!-\!z^j\|
                 \le\left\{\begin{array}{cl}
                 \gamma\varrho^{k} &{\rm if}\ \theta=1/2,\\
                 \gamma k^{\frac{1-\theta}{1-2\theta}}&{\rm if}\ \theta\in(1/2,1)
                 \end{array}\right.\ \ {\rm for\ all}\ k\ge\overline{k}.
              \end{align*}
  \end{itemize}
 \end{theorem}

 Recalling that $\lim_{k\to\infty}\|z^k-z^{k-1}\|\!=0$,
 we have $\lim_{k\to\infty}\Upsilon_{\!\delta}(z^k)\!=\lim_{k\to\infty}\Psi(x^k,y^k)$.
 Together with Proposition \ref{KL-Xirate} for $\Phi=\Upsilon_{\delta}$,
 the following convergence rate result holds for $\{\Psi(x^k,y^k)\}_{k\in\mathbb{N}}$.
 \begin{corollary}\label{obj-rate2}
  If $\Psi$ is a KL function of exponent $\theta\in[1/2,1)$, then there exist
  $\widehat{\varrho}\in(0,1)$ and $\gamma'>0$ such that for all sufficiently large $k$,
  the following inequality holds with $\omega^*\!=\lim_{k\to\infty}\Psi(x^k,y^k)$:
  \[
    \Psi(x^k,y^k)-\omega^*\le\left\{\begin{array}{cl}
    \gamma'\widehat{\varrho}^{\lceil\frac{k-1}{m+1}\rceil} &{\rm if}\ \theta=1/2,\\
    \gamma'{k}^{\frac{1-\theta}{1-2\theta}}&{\rm if}\ \theta\in(1/2,1).
    \end{array}\right.
  \]
 \end{corollary}
 \subsection{Numerical results of Algorithm \ref{PALMenls}}\label{sec5.2}

  We test the performance of Algorithm \ref{PALMenls} for solving the column
  $\ell_{2,0}$-regularized factorization model of low-rank matrix completion (MC) problems.
  For an index set $\Omega\subseteq\{(i,j)\,|\,i\in[n_1],j\in[n_2]\}$ with
  $[n_1]:=\{1,\ldots,n_1\}$ and $[n_2]:=\{1,\ldots,n_2\}$,
  let $P_{\Omega}(X)\in\mathbb{R}^{n_1\times n_2}$ denote the projection
  of $X\in\mathbb{R}^{n_1\times n_2}$ onto $\Omega$, i.e.,
  $[P_{\Omega}(X)]_{ij}=X_{ij}$ if $(i,j)\in\Omega$, otherwise $[P_{\Omega}(X)]_{ij}=0$.
  Given an upper estimation, say $r\in[1,\min(n_1,n_2)]$, for the rank of the true matrix $M^*$,
  we consider the following $\ell_{2,0}$-regularized factor model of low-rank MC problems
  \begin{equation}\label{MS-FL20}
  \min_{U\in \mathbb{R}^{n_1\times r},V\in\mathbb{R}^{n_2\times r}}
  \frac{1}{2}\big\|P_{\Omega}(UV^{\mathbb{T}}\!-\!M)\big\|_F^2
   +\frac{\mu}{2}\big(\|U\|_F^2+\|V\|_F^2\big)+\lambda\big(\|U\|_{2,0}+\|V\|_{2,0}\big),
  \end{equation}
 where $M\in\mathbb{R}^{n_1\times n_2}$ is an observation matrix, $\lambda>0$ is the regularization parameter,
 and $\mu>0$ is a tiny constant. For the further investigation on the model \eqref{MS-FL20},
 refer to the work \cite{TaoQP21}. The problem \eqref{MS-FL20} has
 the form of \eqref{PALM-prob} with $H(U,V)=\frac{1}{2}\big\|P_{\Omega}(UV^{\mathbb{T}}-M)\big\|_F^2$,
 $f(U)=\frac{\mu}{2}\|U\|_F^2+\lambda\|U\|_{2,0}$ and $g(V)=\frac{\mu}{2}\|V\|_F^2+\lambda\|V\|_{2,0}$.
 Clearly, the objective function of \eqref{MS-FL20} is coercive and lower bounded.
 The partial gradients $\nabla_UH(\cdot,V)$ and $\nabla_VH(U,\cdot)$ are respectively
 $\|V\|^2$ and $\|U\|^2$-Lipschitz continuous.

 All trials in the subsequent experiments are generated randomly with a triple $(n_1,n_2,r)$
 in the following way. Assume that a random index set $\Omega=\big\{(i_t,j_t)\in[n_1]\times [n_2]\,|\,t=1,\ldots,p\big\}$
 is available, and that the samples of the indices are drawn independently
 from a general sampling distribution $\Pi=\{\pi_{kl}\}_{k\in[n_1],l\in[n_2]}$
 on $[n_1]\times[n_2]$. We adopt the same non-uniform sampling scheme as in \cite{Fang18},
 i.e., for each $(k,l)\in[n_1]\times[n_2]$, take $\pi_{kl}=p_kp_l$ with $p_k= 2p_0$
 if $k\le\frac{n_1}{10}$, $p_k= 4p_0$ if $\frac{n_1}{10}\le k\le \frac{n_1}{5}$,
 otherwise $p_k=p_0$, where $p_0>0$ is a constant such that $\sum_{k=1}^{n_1}p_k=1$,
 and $p_l$ is defined in a similar way. The entries $M_{i_t,j_t}$ with $(i_t,j_t)\in\Omega$
 for $t=1,2,\ldots,p$ are generated via the observation model
 \[
   M_{i_t,j_t}=M_{i_t,j_t}^*+\sigma({\xi_{t}}/{\|\xi\|})\|M_{\Omega}^*\|_F,
 \]
 where $M^*\!\in\mathbb{R}^{n_1\times n_2}$ is the true matrix of rank $r^*$,
 $\xi\in\mathbb{R}^p$ is the noisy vector whose entries are i.i.d. and
 obey the standard normal distribution, and $\sigma>0$ represents the noise level.
 In the subsequent experiments, we choose $\mu=10^{-10}$ and the parameters of
 Algorithm \ref{PALMenls} as follows:
 \[
  \underline{\tau}=\!10^{-8},\overline{\tau}\!=10^{8},
  \beta_{\rm max}\!=1,\eta=0.01,\eta_1=\eta_2=0.5,\delta=0.01,
   \alpha=10^{-5},\tau_{1,0}^0=\!\frac{100}{\|V^0\|^2},\tau_{2,0}^0=\!\frac{100}{\|U^0\|^2}.
 \]

 We apply Algorithm \ref{PALMenls} for solving the problem \eqref{MS-FL20} with
 $(n_1,n_2,r)=(1000,1000,100)$, and compare its performance with
 those of Algorithm \ref{PALMenls} with $\beta_{\rm max}=0$ (PALMnls),
 Algorithm \ref{PALMenls} with $m=0$ (PALMels), Algorithm \ref{PALMenls}
 with $\beta_{\rm max}\!=0,m=0$ (PALMls), PALM with extrapolation (PALMe) and PALM.
 We evaluate the performances of different algorithms by an evolution of objective values
 as in Section \ref{sec4.2}. From Figure \ref{fig3}, we see that PALMenls and PALMels
 almost have the same performance for all test problems. In Figure \ref{fig3} (a),
 the ranks yielded by all methods equal $r$ due to a small $\lambda$,
 and now PALMe has a little better performance than PALMenls and PALMels do,
 which are much better than other methods. In Figure \ref{fig3} (b)-(c),
 the ranks yielded by PALMnls and PALMls are much lower than the ranks yielded by
 other methods, and hence they have better performance than other methods do.
 In Figure \ref{fig3} (d)-(e), the ranks yielded by PALMe and PALM are much higher
 than those yielded by other methods, and PALMenls and PALMels have better performance
 though the ranks yielded by them are same as those yielded by PALMnls and PALMls.
 From Figure \ref{fig3}, we conclude that PALMnls and PALMls are more efficient to
 reduce the rank, and PALMenls and PALMels are more efficient for the problem \eqref{MS-FL20}
 with a smaller $\lambda$ or a larger $\lambda$.

  \begin{figure}[h]
  \begin{subfigure}[b]{0.33\textwidth}
   \centering
  \includegraphics[width=5.5cm,height=5cm]{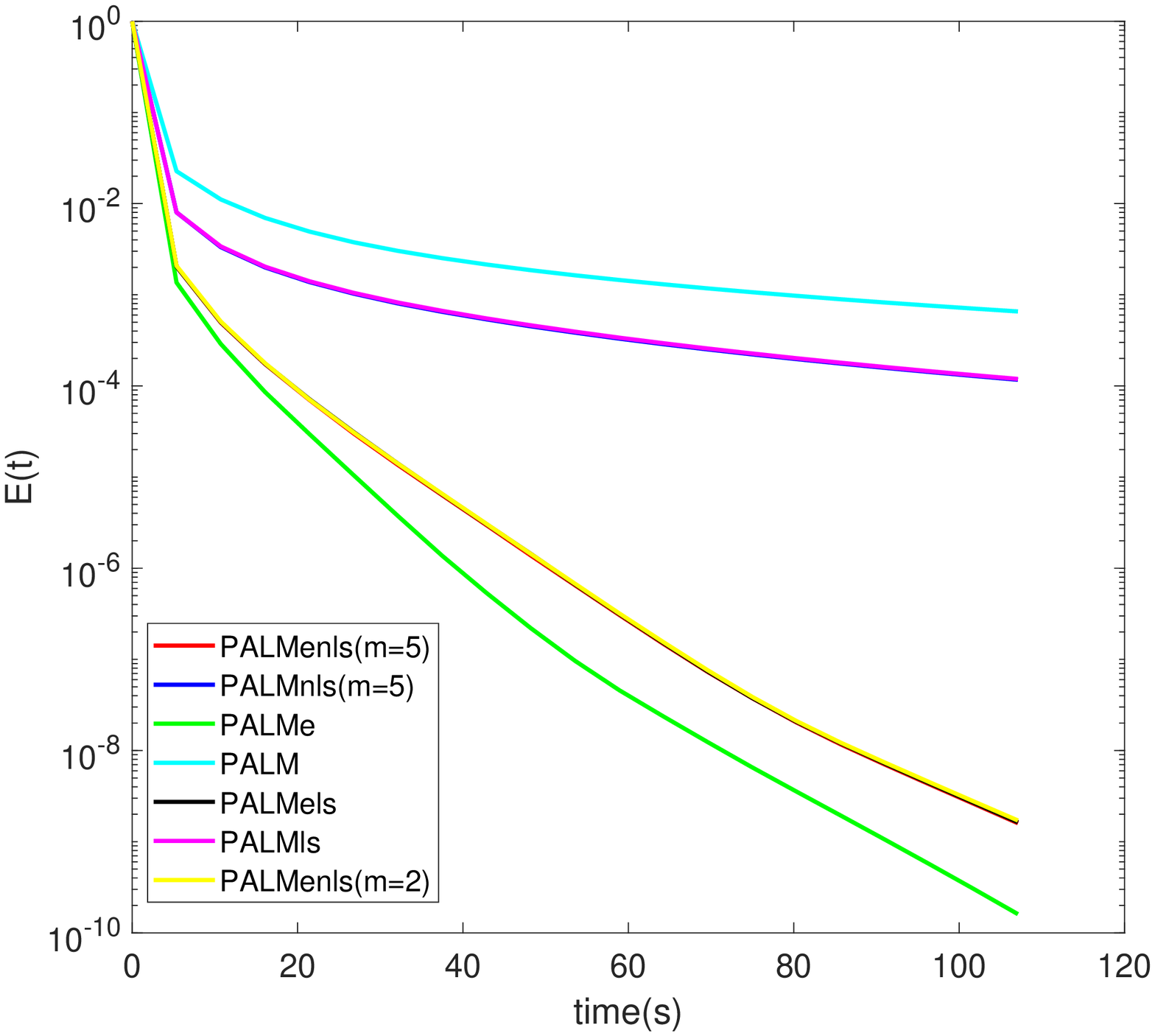}
   \caption{$\lambda=100$}
   \label{fig3:1_1}
  \end{subfigure}
  \begin{subfigure}[b]{0.33\textwidth}
   \centering
  \includegraphics[width=5.5cm,height=5cm]{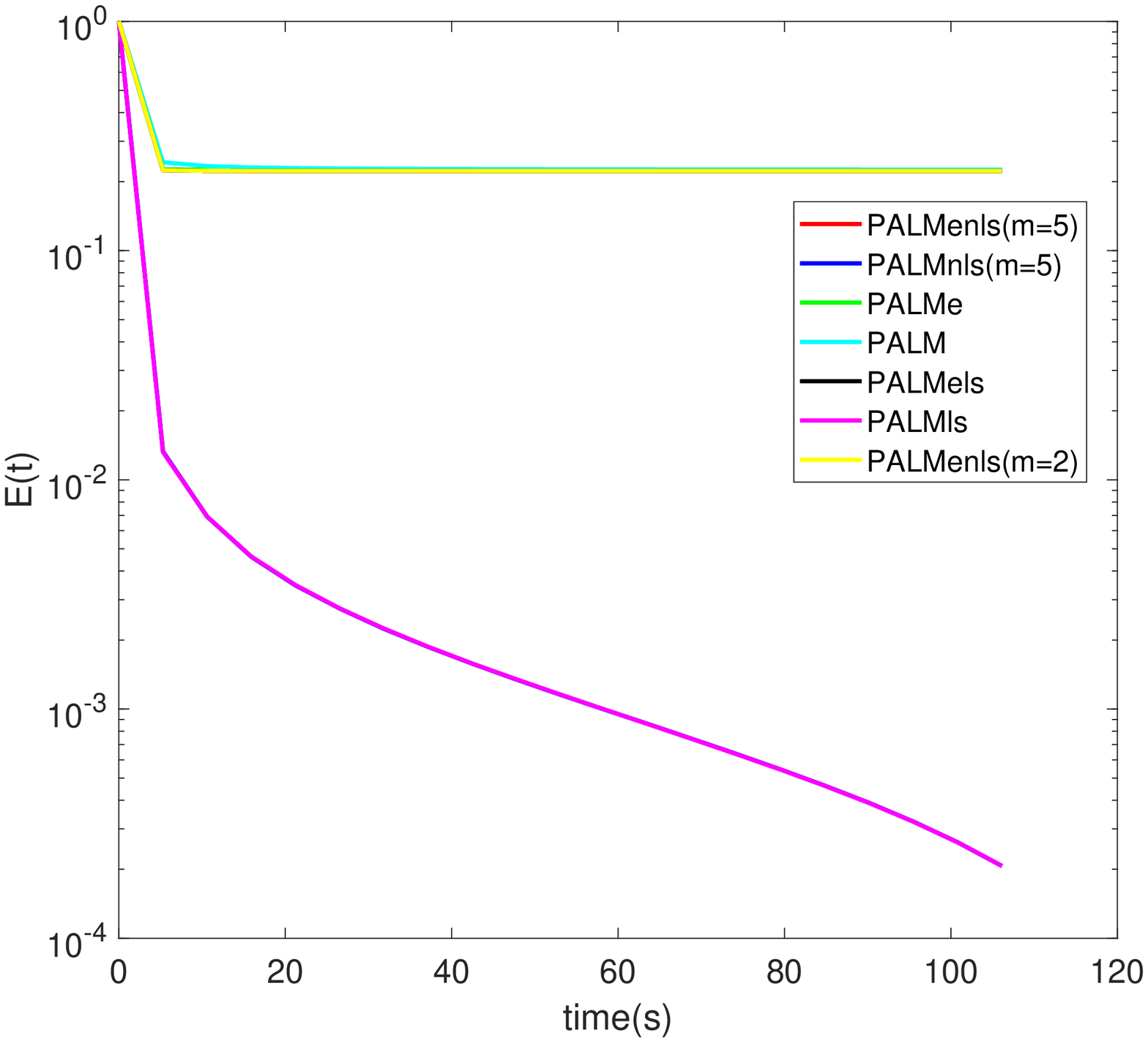}
   \caption{$\lambda=500$}
   \label{fig3:1_2}
  \end{subfigure}
  \begin{subfigure}[b]{0.33\textwidth}
   \centering
  \includegraphics[width=5.5cm,height=5cm]{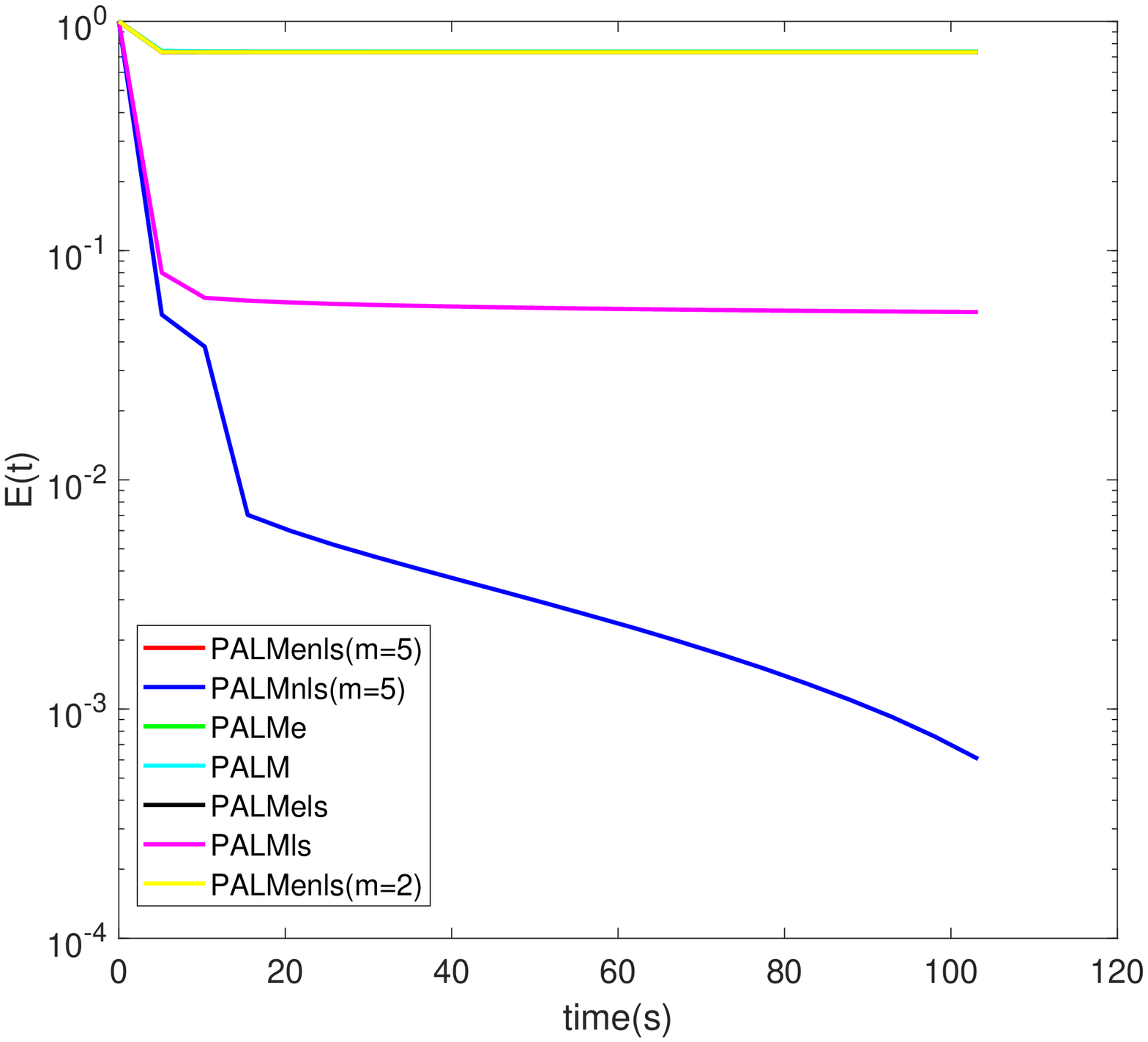}
   \caption{$\lambda=1000$}
   \label{fig3:1_3}
  \end{subfigure}
  \newline
  \begin{subfigure}[b]{0.33\textwidth}
   \centering
  \includegraphics[width=5.5cm,height=5cm]{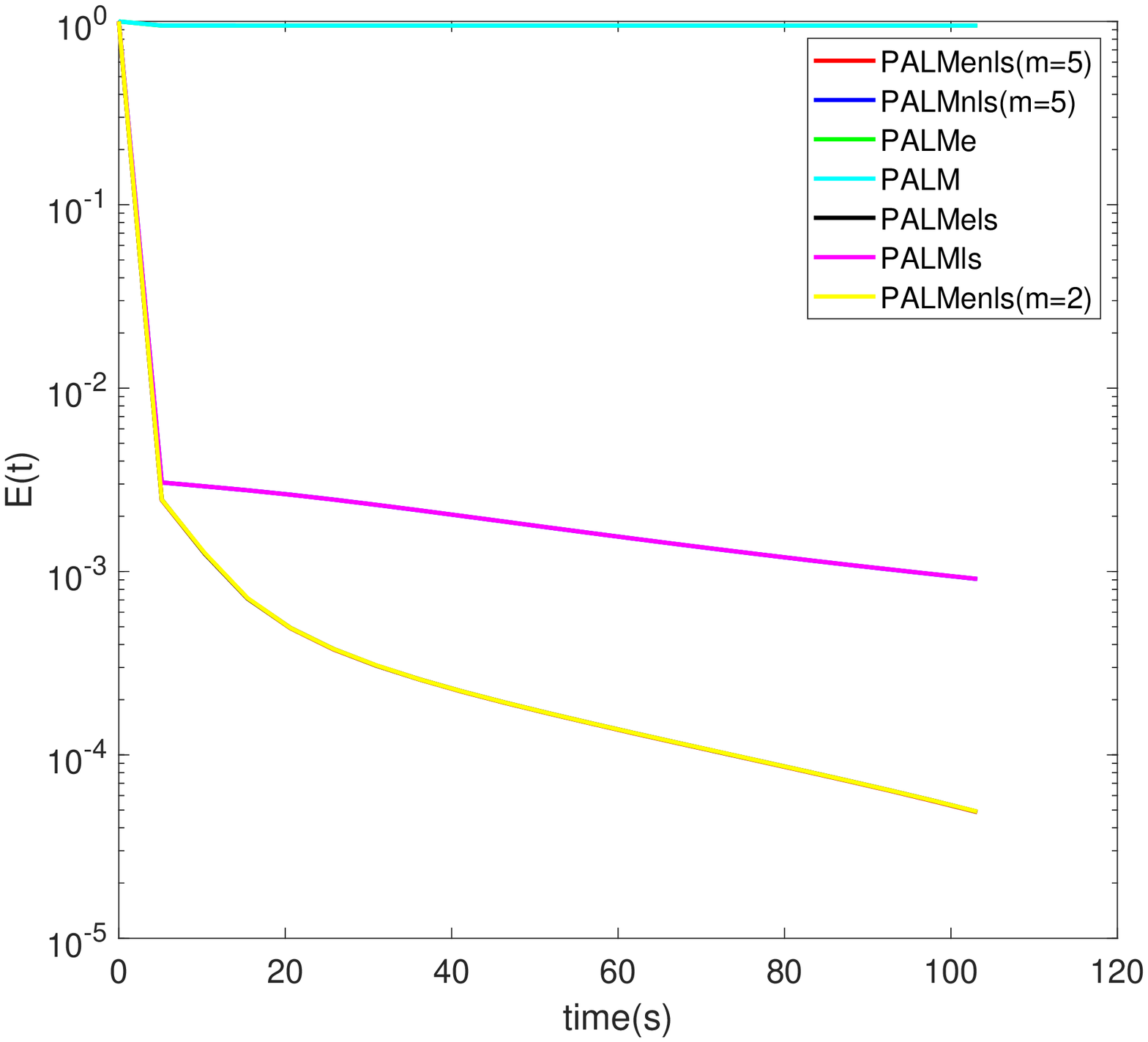}
   \caption{$\lambda=3000$}
   \label{fig3:2_1}
  \end{subfigure}
  \begin{subfigure}[b]{0.33\textwidth}
   \centering
  \includegraphics[width=5.5cm,height=5cm]{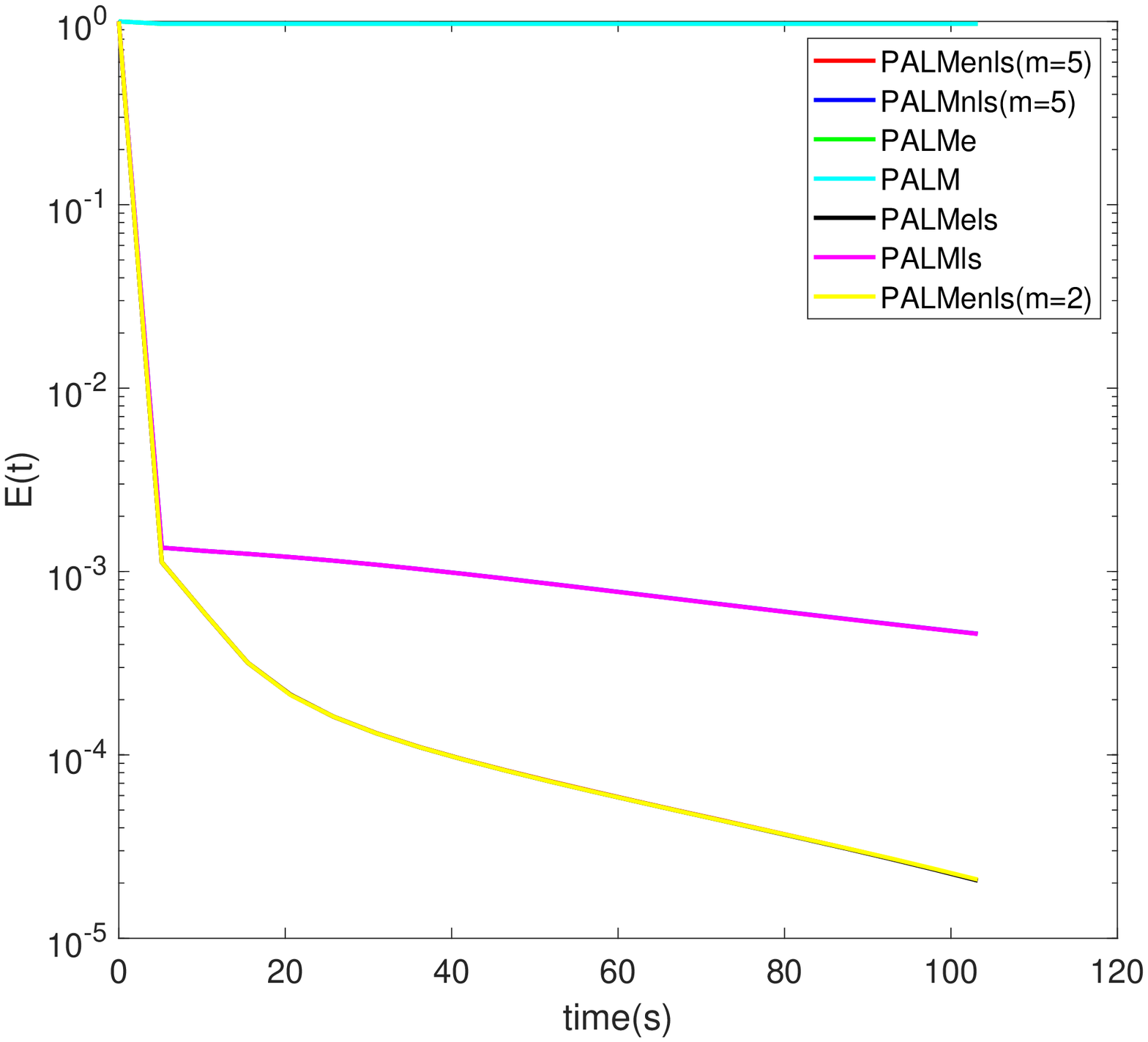}
   \caption{$\lambda=5000$}
   \label{fig3:2_2}
  \end{subfigure}
  \begin{subfigure}[b]{0.33\textwidth}
   \centering
  \includegraphics[width=5.5cm,height=5cm]{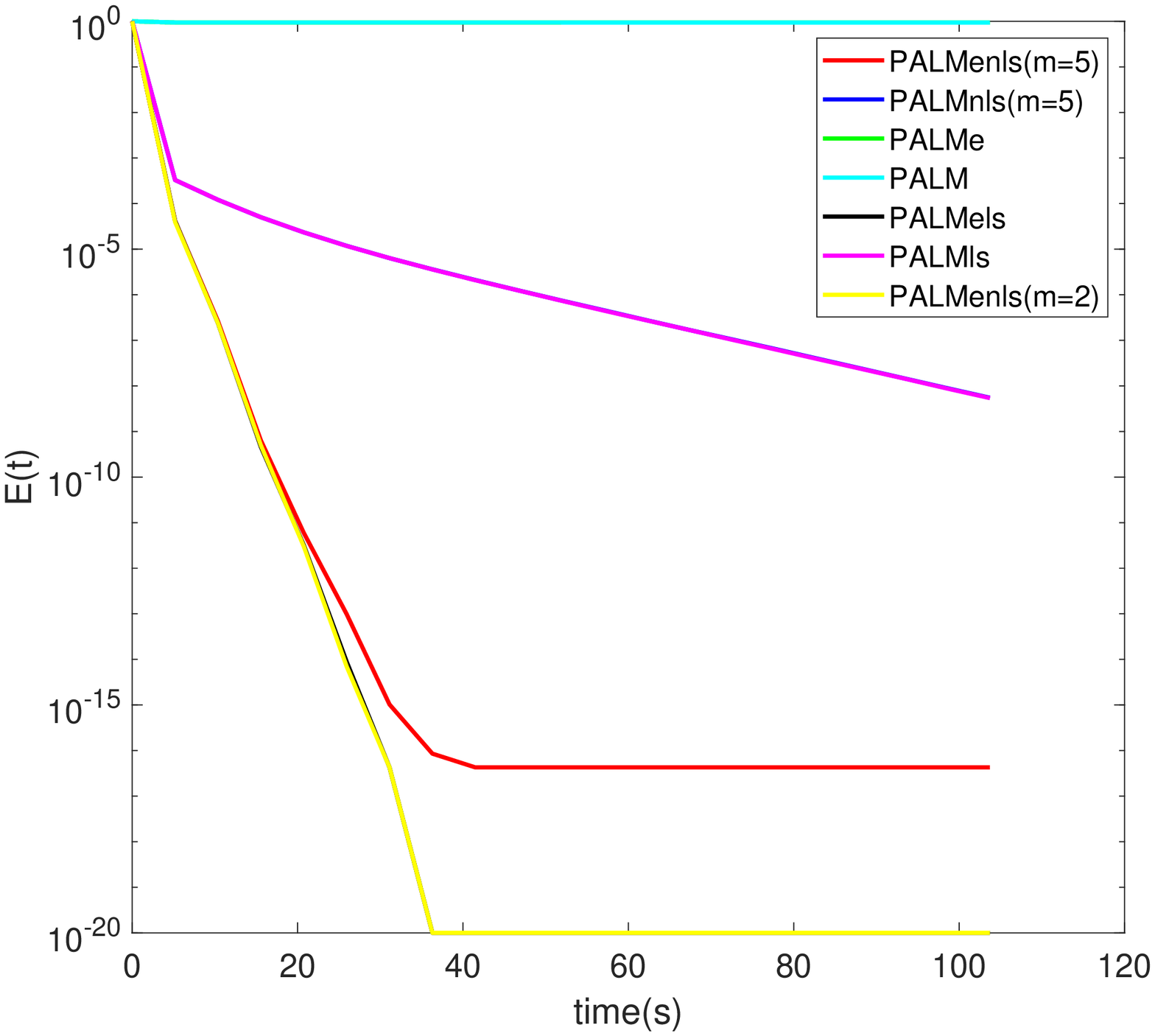}
   \caption{$\lambda=8000$}
   \label{fig3:2_3}
  \end{subfigure}
  \caption{\small Several methods for solving the nonconvex nonsmooth problem \eqref{MS-FL20} with different $\lambda$}
  \label{fig3}
 \end{figure}
 \section{Conclusions}\label{sec6}

  For the iterate sequence satisfying conditions H1-H2, 
  generated by a class of nonmonotone descent methods 
  for minimizing a nonconvex and nonsmooth KL function $\Phi$,
  we established its global convergence and local convergence rate 
  respectively under condition \eqref{assump0} and \eqref{rate-cond},
  which are proved to be sufficient and necessary if $\Phi$ is also 
  weakly convex on a neighborhood of stationary point set. 
  Condition \eqref{assump0} and \eqref{rate-cond} are not easy to check 
  since they involve the growth of an objective value subsequence, 
  though we have provided a sufficient condition (independent of 
  the objective value sequence) for them, which can be satisfied 
  by a class of $\rho$-weakly convex functions with $\rho\le\frac{a}{8(m+1)^2}$ 
  on a neighborhood of stationary point set. 
  We have applied the obtained results to establishing the global convergence and 
  convergence rate of the iterate sequence for PGenls and PALMenls, and numerical 
  results indicate that under some scenarios, they are superior to the monotone line 
  search versions and/or the extrapolation versions.
  Our future work will focus on 
  other nonmonotone descent conditions for the generated iterate
  sequences to be convergent under a weaker or verifiable assumption.

 \bigskip
 \noindent
 {\bf\large Appendix:}\\
 \begin{alemma}\label{PALM-descent}
  Let $\{(x^k,y^k)\}_{k\in\mathbb{N}}$ be the sequence generated by Algorithm \ref{PALMenls}.
  If for each $k\in\mathbb{N}$, $\beta_{k}\!\le\!\min\Big(\!\sqrt{\!\frac{0.25\delta(\tau_{1,k}^{-1}-L_1^k-\delta)}
  {L_1^k(\tau_{1,k}^{-1}-L_1^k-\delta)+(\tau_{1,k}^{-1}-L_1^k)^2}},
  \sqrt{\!\frac{0.25\delta(\tau_{2,k}^{-1}-L_2^{k+1}-\delta)}
  {L_2^{k+1}(\tau_{2,k}^{-1}-L_2^{k+1}-\delta)+(\tau_{2,k}^{-1}-L_2^{k+1})^2}}\Big)$
  with $L_2^{k+1}\!=\!L_2(x^{k+1}),L_1^k\!=\!L_1(y^k)$, then the line search criterion
  in Step 6 is satisfied when $\max(\tau_{1,k},\tau_{2,k})\le\frac{1}{L+\delta+2\alpha}$.
 \end{alemma}
 \begin{proof}
  Fix any $y\in\mathbb{Y}$. Recall that the partial gradient $\nabla_{\!x}H(\cdot,y)$
  is $L_1(y)$-Lipschitz continuous. From the descent lemma, for any $x',x\in\mathbb{X}$
  and $\tau\le1/L_1(y)$ it holds that
  \begin{subnumcases}{}\label{Hx}
   H(x',y)\le H(x,y)+\langle\nabla_{\!x}H(x,y),x'\!-\!x\rangle
    +0.5\tau^{-1}\|x'\!-\!x\|^2,\\
  -H(x',y)\le -H(x,y)-\langle\nabla_{\!x}H(x,y),x'\!-\!x\rangle
    +0.5\tau^{-1}\|x'\!-\!x\|^2.
  \label{nHx-ineq}
  \end{subnumcases}
  Fix any $x\in\!\mathbb{X}$. Since $\nabla_{\!y}H(x,\cdot)$ is $L_2(x)$-Lipschitz continuous,
  for any $y',y\in\!\mathbb{Y}$ and $\tau\le 1/{L_2(x)}$,
  \begin{subnumcases}{}\label{Hy}
   H(x,y')\le H(x,y)+\langle\nabla_{\!y}H(x,y),y'\!-\!y\rangle
    +0.5\tau^{-1}\|y'\!-\!y\|^2,\\
   -H(x,y')\le -H(x,y)-\langle\nabla_{\!y}H(x,y),y'\!-\!y\rangle
    +0.5\tau^{-1}\|y'\!-\!y\|^2.
   \label{nHy-ineq}
  \end{subnumcases}
 For any $(x,y,u,v)\in\mathbb{X}\times\mathbb{Y}\times\mathbb{X}\times\mathbb{Y}$,
 define $\psi_{1,\delta}(x,y,u):=f(x)+H(x,y)+({\delta}/{2})\|x-u\|^2$ and
 $\psi_{2,\delta}(x,y,v):=g(y)+H(x,y)+({\delta}/{2})\|y-v\|^2$.
 From the definition of $x^{k+1}$, it follows that
 \begin{equation}\label{key-ineq2}
   f(x^{k+1})\le \langle\nabla_{\!x}H(\widetilde{x}^{k},y^k),x^{k}-x^{k+1}\rangle+f(x^{k})
      -\frac{1}{2\tau_{1,k}}\|x^{k+1}\!-\!\widetilde{x}^{k}\|^2+\frac{1}{2\tau_{1,k}}\|x^{k}\!-\!\widetilde{x}^{k}\|^2.
  \end{equation}
 Together with the definition of $\psi_{1,\delta}$, it is not difficult to obtain that
 \begin{align}\label{Lf-ineq}
  \psi_{1,\delta}(x^{k+1},y^k,x^k)
   &\le H(x^{k+1},y^k)+({\delta}/{2})\|x^{k+1}-x^k\|^2+\langle\nabla_xH(\widetilde{x}^k,y^k),x^{k}-x^{k+1}\rangle\nonumber\\
   &\quad+f(x^{k})-\frac{1}{2\tau_{1,k}}\|x^{k+1}\!-\!\widetilde{x}^{k}\|^2+\frac{1}{2\tau_{1,k}}\|x^{k}\!-\!\widetilde{x}^{k}\|^2,\nonumber\\
   &\le H(x^{k},y^k)+({\delta}/{2})\|x^{k+1}-x^k\|^2+f(x^{k})\nonumber\\
   &\quad-\frac{1}{2}[\tau_{1,k}^{-1}-L_1(y^k)]\|x^{k+1}\!-\!\widetilde{x}^{k}\|^2
   +\frac{1}{2}[\tau_{1,k}^{-1}+L_1(y^k)]\|x^{k}\!-\!\widetilde{x}^{k}\|^2\nonumber\\
  &\le \psi_{1,\delta}(x^{k},y^k,x^{k-1})-\frac{1}{2}(\tau_{1,k}^{-1}\!-\!L_1^k-\delta)\|x^{k+1}\!-\!x^{k}\|^2\nonumber\\
   &\quad -\frac{\delta\!-\!2L_1^k\beta_{k}^2}{2}\|x^{k}\!-\!x^{k-1}\|^2
   -(\tau_{1,k}^{-1}\!-\!L_1^k)\beta_{k}\langle x^{k+1}\!-\!x^k,x^k\!-\!x^{k-1}\rangle
  \end{align}
  where the second inequality is obtained by using \eqref{Hx} with $x'=x^{k+1},x=\widetilde{x}^k$
  and \eqref{nHx-ineq} with $x'=x^{k},x=\widetilde{x}^k$, and the last one is due to
  $\widetilde{x}^k=x^k+\beta_{k}(x^k\!-\!x^{k-1})$. Since for any $\mu>0$,
  \(
   \big|2(\tau_{1,k}^{-1}\!-\!L_1^k)\beta_{k}\langle x^{k+1}\!-\!x^k,x^k\!-\!x^{k-1}\rangle\big|
   \le \mu(\tau_{1,k}^{-1}\!-\!L_1^k)^2\|x^{k+1}\!-\!x^k\|^2+\frac{\beta_{k}^2}{\mu}\|x^k\!-\!x^{k-1}\|^2,
  \)
  we have
  \begin{align*}
   \psi_{1,\delta}(x^{k+1},y^k,x^k)
   &\le \psi_{1,\delta}(x^{k},y^k,x^{k-1})
     -\frac{1}{2}\big[(\delta\!-\!2L_1^k\beta_{k}^2)-\beta_{k}^2/\mu\big]\big\|x^k\!-\!x^{k-1}\big\|^2\\
  &\quad -\frac{1}{2}\big[(\tau_{1,k}^{-1}\!-\!L_1^k-\delta)
    -\mu(\tau_{1,k}^{-1}\!-\!L_1^k)^2\big]\|x^{k+1}\!-\!x^k\|^2.
   \end{align*}
  By taking $\mu=\frac{\tau_{1,k}^{-1}-L_1^k-\delta}{2(\tau_{1,k}^{-1}-L_1^k)^2}$ and
  using the given assumption on $\beta_{k}$, it follows that
 \begin{equation}\label{psi1}
 \psi_{1,\delta}(x^{k+1},y^k,x^k)
 \le \psi_{1,\delta}(x^{k},y^k,x^{k-1})-({\delta}/{4})\|x^k\!-\!x^{k-1}\|^2
  -\frac{1}{4}(\tau_{1,k}^{-1}\!-\!L_1^k-\delta)\|x^{k+1}\!-\!x^k\|^2.
 \end{equation}
 Similarly, from the definition of $y^{k+1}$ and the expression of $\psi_{2,\delta}$, it follows that
 \begin{align}\label{Lf-ineq2}
   \psi_{2,\delta}(x^{k+1},y^{k+1},y^k)
   &\le H(x^{k+1},y^{k+1})+\frac{\delta}{2}\|y^{k+1}-y^k\|^2+\langle\nabla_yH(x^{k+1},\widetilde{y}^k),y^{k}-y^{k+1}\rangle\nonumber\\
   &\quad +g(y^{k})-\frac{1}{2\tau_{2,k}}\|y^{k+1}\!-\!\widetilde{y}^{k}\|^2
   +\frac{1}{2\tau_{2,k}}\|y^{k}\!-\!\widetilde{y}^{k}\|^2,\nonumber\\
   &\le H(x^{k+1},y^k)+({\delta}/{2})\|y^{k+1}-y^k\|^2+g(y^{k})\nonumber\\
   &\quad-\frac{1}{2}(\tau_{2,k}^{-1}-L_2^{k+1})\|y^{k+1}\!-\!\widetilde{y}^{k}\|^2
   +\frac{1}{2}(\tau_{2,k}^{-1}+L_2^{k+1})\|y^{k}\!-\!\widetilde{y}^{k}\|^2
  \end{align}
  where the second inequality is obtained by using \eqref{Hy} with $y'=y^{k+1},y=\widetilde{y}^k$
  and \eqref{nHy-ineq} with $y'=y^{k},y=\widetilde{y}^k$.
  Substituting $\widetilde{y}^k=y^k+\beta_{k}(y^k\!-\!y^{k-1})$ into \eqref{Lf-ineq2},
  for any $\mu>0$ it holds that
  \begin{align*}
   \psi_{2,\delta}(x^{k+1},y^{k+1},y^k) 
   &\le \psi_{2,\delta}(x^{k+1},y^k,y^{k-1})
     -\frac{1}{2}\big[(\delta\!-\!2L_2^{k+1}\beta_{k}^2)-\beta_{k}^2/\mu\big]\big\|y^k\!-\!y^{k-1}\big\|^2\\
  &\quad -\frac{1}{2}\big[(\tau_{2,k}^{-1}\!-\!L_2^{k+1}\!-\!\delta)
  -\mu(\tau_{2,k}^{-1}\!-\!L_2^{k+1})^2\big]\|y^{k+1}\!-\!y^k\|^2.
  \end{align*}
  By taking $\mu=\frac{\tau_{2,k}^{-1}-L_2^{k+1}-\delta}{2(\tau_{2,k}^{-1}-L_2^{k+1})^2}$
  and using the given assumption on $\beta_{k}$, we obtain that
  \[
    \psi_{2,\delta}(x^{k+1},y^{k+1},y^k)
    \le \psi_{2,\delta}(x^{k+1},y^k,y^{k-1})-\frac{\delta}{4}\|y^k\!-\!y^{k-1}\|^2
     -\frac{1}{4}(\tau_{2,k}^{-1}\!-\!L_2^{k+1}\!-\!\delta)\|y^{k+1}\!-\!y^k\|^2.
  \]
 Together with the inequality \eqref{psi1} and the definition of $\Upsilon_{\delta}$,
 it follows that
 \[
  \Upsilon_{\delta}(z^{k+1})\le\Upsilon_{\delta}(z^{k})-\min\Big\{\frac{\delta}{4},
  \frac{\tau_{1,k}^{-1}\!-\!L_1^k-\delta}{4},\frac{\tau_{2,k}^{-1}\!
  -\!L_2^{k+1}\!-\!\delta}{4}\Big\}\|z^{k+1}\!-\!z^{k}\|^2.
 \]
 Notice that $\delta\in(0,1)$ and $0<\alpha\le{\delta}/{2}$.
 The line search criterion in Step 6 for $m=0$ is satisfied when
 $\tau_{1,k}\le\frac{1}{L+\delta+2\alpha}\le\frac{1}{L_1^k+\delta+2\alpha}$ and
 $\tau_{2,k}\le\frac{1}{L+\delta+2\alpha}\le\frac{1}{L_2^{k+1}+\delta+2\alpha}$,
 so is the criterion in Step 6.
 \end{proof}

 \end{document}